\documentclass{amsart}
\usepackage{mathmacros_numbered_thms}
\usepackage{hyperref}

\allowdisplaybreaks

\numberwithin{equation}{section}

\begin{document}

\title{The energy-critical quantum harmonic oscillator}
\author{Casey Jao}
\begin{abstract}
  We consider the energy critical nonlinear Schr\"{o}dinger equation
  in dimensions $d \ge 3$ with a harmonic oscillator potential $V(x) =
  \tfr{1}{2} |x|^2$. When the nonlinearity is defocusing, we prove
  global wellposedness for all initial data in the energy space
  $\Sigma$, consisting of all functions $u_0$ such that both $\nabla
  u_0$ and $x u_0$ belong to $L^2$. This result extends a theorem of
  Killip-Visan-Zhang \cite{kvz_quadratic_potentials}, which treats the
  radial case. For the focusing problem, we obtain global
  wellposedness for all data satisfying an analogue of the usual size
  restriction in terms of the ground state $W$. The proof uses the
  concentration compactness variant of the induction on energy
  paradigm. In particular, we develop a linear profile decomposition
  adapted to the propagator $\exp[ it(\tfr{1}{2}\Delta -
  \tfr{1}{2}|x|^2)]$ for bounded sequences in $\Sigma$.
\end{abstract}
\maketitle
\tableofcontents
\section{Introduction}

We study the initial value problem for the energy-critical nonlinear 
Schr\"{o}dinger
equation on $\mf{R}^d, \ d \ge 3$, with a harmonic oscillator potential:
\begin{equation}
\label{eqn:nls_with_quadratic_potential}
\left\{ \begin{array}{c} i\partial_t u = (-\tfr{1}{2}\Delta + \tfr{1}{2}|x|^2) u 
+ \mu |u|^{\fr{4}{d-2}} u,  \quad \mu = \pm 1, \\[2mm]
u(0) = u_0 \in \Sigma(\mf{R}
^d).
\end{array}\right.
\end{equation}
The equation is \emph{defocusing} if $\mu = 1$ and \emph{focusing} if
$\mu = -1$. Solutions to this PDE conserve energy, which is defined as
\begin{equation}
E(u(t)) = \int_{\mf{R}^d} \left[\tfr{1}{2} |\nabla u(t)|^2 + \tfr{1}{2} |x|^2
|u(t)|^2 + \tfr{d-2}{d} \mu |u(t)|^{\fr{2d}{d-2}} \right] dx = E(u(0)).
\end{equation}
Indeed \eqref{eqn:nls_with_quadratic_potential} can be viewed as
defining the (formal) Hamiltonian flow of $E$. 
The term ``energy-critical'' refers to the fact that
if we ignore the $|x|^2/2$ term in the equation and the energy, the scaling 
\begin{equation}
\label{eqn:energy-critical_scaling}
u(t, x) \mapsto u^\lambda(t, x) := \lambda^{-\fr{2}{d-2}}
u(\lambda^{-2} t, \lambda^{-1} x)
\end{equation}
preserves both the equation and the energy. 
We take our initial data in the weighted Sobolev space~$\Sigma$, which is 
the natural space of functions associated with the energy functional. This space 
is equipped with the norm
\begin{equation}
\|f\|_{\Sigma}^2 = \| \nabla f\|_{L^2}^2 + \| xf\|_{L^2}^2 = \|
f\|_{\dot{H}^1}^2 + \|f\|_{L^2(|x|^2 \, dx)}^2  \label{eqn:defn_of_Sigma_norm}
\end{equation}
We will frequently employ 
the notation
\[
 H = -\tfr{1}{2}\Delta + \tfr{1}{2}|x|^2, \quad F(z) = \mu |z|^{\fr{4}{d-2}} z.
\] 
Let us first clarify what we mean by a solution.
\begin{define_nonum}
A (strong) \emph{solution} to \eqref{eqn:nls_with_quadratic_potential} is a function 
$u: I \times \mf{R}^d \to 
\mf{C}$ that belongs to $C^0_t(K; \Sigma)$ for 
every compact interval $K \subset I$, and that satisfies the Duhamel formula
\begin{equation}
u(t) = e^{-itH} u(0) - i \int_0^t e^{-i(t-s)H} F(u(s)) \, ds \quad
\text{for all} \quad t \in I. \label{eqn:definition_of_solution}
\end{equation}
The hypothesis on $u$ implies that $F(u) \in C^0_{t, loc} 
L^{\fr{2d}{d+2}}_x(I \times \mf{R}^d)$. Consequently,the right side above is well-defined, at least as a 
weak integral of tempered distributions. 
\end{define_nonum}

Equation \eqref{eqn:nls_with_quadratic_potential} and its variants
\[
i\partial_t u = (-\tfr{1}{2}\Delta + V)u + F(u), \quad V = \pm \tfr{1}{2}|x|^2, 
\quad F(u) = \pm |u|^p u, \quad p > 0
\]
have received considerable attention, especially in the
energy-subcritical regime $p < 4/(d-2)$. The equation with a
confining potential $V = |x|^2/2$ has been used to model
Bose-Einstein condensates in a trap (see \cite{zhang_bec},
for example). Let us briefly review some of the mathematical
literature. 

Carles \cite{carles_confining_oscillator},
\cite{carles_repulsive_oscillator} proved global
wellposedness for a defocusing nonlinearity $F(u) = |u|^p u, \ p <
4/(d-2)$ when the potential $V(x) = |x|^2/2$ is either
confining or repulsive, and obtained various wellposedness
and blowup results for a focusing nonlinearity $F(u) = -|u|^p
u$. In \cite{carles_sign-indefinite_oscillator}, he also studied \cite{carles_sign-indefinite_oscillator} the
case of an anisotropic harmonic oscillator with $V(x) = \sum_j \delta_j
x_j^2/2, \ \delta_j \in \{1, 0, -1\}$.

There has also been interest in more general potentials. The paper
\cite{oh} proves long-time existence in the presence of a focusing, mass-subcritical
nonlinearity $F(u) = - |u|^p u, \ p < 4/d$ when $V(x)$ is merely assumed to grow 
subquadratically (by which we mean
$\partial^\alpha V \in L^\infty$ for all $|\alpha| \ge 2$). More
recently, Carles \cite{carles_time-dependent_potential} considered
\emph{time-dependent} subquadratic potentials $V(t, x)$. Taking
initial data in $\Sigma$, he
established global existence and uniqueness when $4/d \le p < 4/(d-2)$ 
for the defocusing nonlinearity and $0 < p < 4/d$ in the focusing case.

This paper studies the energy-critical problem $p = 4/(d-2)$. While the critical
equation still admits a local theory, the duration of local existence
obtained by the usual fixed-point argument depends on the profile
and not merely on the norm of the initial data $u_0$. Therefore, one
cannot pass directly from local wellposedness to global
wellposedness using conservation laws as in the subcritical case. This issue is 
most evident
if we temporarily discard the potential and consider the equation
\begin{equation}
\label{eqn:ckstt_eqn}
i\partial_t u = -\tfr{1}{2}\Delta u + \mu |u|^{\fr{4}{d-2}}u, \quad u(0) = u_0
\in \dot{H}^1(\mf{R}^d), \quad d \ge 3,
\end{equation}
which has the Hamiltonian
\[
E_{\Delta}(u) = \int \tfr{1}{2} |\nabla u|^2 + \mu\tfr{d-2}{d}
|u|^{\fr{2d}{d-2}} \, dx.
\]
We will refer to this equation in the sequel as the
``potential-free''', ``translation-invariant'', or ``scale-invariant''
problem. Since the spacetime scaling
\eqref{eqn:energy-critical_scaling} preserves both the equation and
the $\dot{H}^1$ norm of the initial data, the time of existence
guaranteed by the local wellposedness theory cannot depend merely on
$\| u_0\|_{\dot{H}^1}$. One cannot iterate the local existence
argument to obtain global existence because with each iteration the
solution could conceivably become more concentrated in space while
remaining bounded in $\dot{H}^1$, so that the duration of local
existence could decrease with each iteration.  The scale invariance
makes the analysis of \eqref{eqn:ckstt_eqn} highly nontrivial.

We mention equation \eqref{eqn:ckstt_eqn} because the original
equation increasingly resembles \eqref{eqn:ckstt_eqn} as the initial
data concentrates at a point; see
sections~\ref{subsection:convergence}
and~\ref{section:localized_initial_data} for more precise statements
concerning this sort of limit. Hence, one would expect the essential
difficulties in the energy-critical NLS to also manifest themselves in
the energy-critical harmonic oscillator. Understanding the
scale-invariant problem is therefore an important step toward
understanding the harmonic oscillator. The last fifteen years have
witnessed intensive study of the former, and the following conjecture
has been verified in all but a few cases:
\begin{conjecture}
\label{conjecture:ckstt}
  When $\mu = 1$, solutions to \eqref{eqn:ckstt_eqn} exist
  globally and scatter. That is, for any $u_0 \in \dot{H}^1(\mf{R}^d)$,
  there exists a unique global solution $u : \mf{R} \times \mf{R}^d \to
  \mf{C}$ to \eqref{eqn:ckstt_eqn} with $u(0) = u_0$, and
  this solution satisfies a spacetime bound
\begin{equation}
\label{eqn:ckstt_spacetime_bound}
S_{\mf{R}}(u) := \int_{\mf{R}} \int_{\mf{R}^d}
|u(t,x)|^{\fr{2(d+2)}{d-2}} \, dx \, dt \le C(E_{\Delta}(u_0)) < \infty.
\end{equation}
Moreover, there exist functions $u_{\pm} \in \dot{H}^1(\mf{R}^d)$ such that
\[
\lim_{t \to \pm \infty} \|u(t) - e^{\pm \fr{it\Delta}{2}} u_{\pm}
\|_{\dot{H}^1} = 0,
\]
and the correspondences $u_0 \mapsto u_{\pm}(u_0)$ are homeomorphisms
of $\dot{H}^1$. 

  When $\mu = -1$, one also has global
  wellposedness and scattering provided that 
\[
E_{\Delta}(u_0) < E_{\Delta}(W), \quad \| \nabla u_0\|_{L^2} < \|
\nabla W\|_{L^2},
\]
where the \emph{ground state}
\[
W(x) = \tfr{1}{ (1 + \fr{2|x|^2}{d(d-2)} )^{\fr{d-2}{2}}} \in
\dot{H}^1 (\mf{R}^d)
\]
solves the elliptic equation $\tfr{1}{2}\Delta + |W|^{\fr{4}{d-2}} W =
0$. 
\end{conjecture}

\begin{thm}
  \label{thm:ckstt}
Conjecture \ref{conjecture:ckstt} holds for the defocusing
equation. For the focusing equation, the conjecture holds for radial
initial data when $d \ge 3$, and for all initial data when $d \ge 5$. 
\end{thm}
\begin{proof}
  See \cite{bourgain_nls_radial_gwp, ckstt, ryckman-visan,
    visan_nls_general_d} for the defocusing case and
  \cite{kenig-merle_focusing_nls, kv_focusing_nls} for the focusing
  case.
\end{proof}

One can formulate a similar conjecture for
\eqref{eqn:nls_with_quadratic_potential}; however, as the linear
propagator is periodic in time, one only expects uniform
local-in-time spacetime bounds.
\begin{conjecture}
\label{conjecture:harmonic_oscillator}
When $\mu = 1$, equation \eqref{eqn:nls_with_quadratic_potential} is
globally wellposed. That is, for each $u_0 \in \Sigma$ there is a
unique global solution $u: \mf{R} \times \mf{R}^d \to \mf{C}$ with
$u(0) = u_0$. This solution obeys the spacetime bound
\begin{equation}
\label{eqn:main_theorem_spacetime_bound}
S_I(u) := \int_I \int_{\mf{R}^d} |u(t,x)|^{\fr{2(d+2)}{d-2}} \, dx \,
dt \le C(|I|, \|u_0\|_{\Sigma} )
\end{equation}
for any compact interval $I \subset \mf{R}$.

If $\mu = -1$, then the same is true provided also that 
\[E(u_0) <
E_{\Delta}(W) \quad \text{and} \quad \| \nabla u_0 \|_{L^2} \le \| \nabla W \|_{L^2}.
\]

\end{conjecture}

In \cite{kvz_quadratic_potentials}, Killip-Visan-Zhang verifed this
conjecture with $\mu = 1$ and spherically symmetric initial data. By
adapting an argument of Bourgain-Tao for the equation without potential
\eqref{eqn:ckstt_eqn}, they proved that the defocusing problem
\eqref{eqn:nls_with_quadratic_potential} is globally wellposed, and
also proved scattering for the repulsive potential. We consider only
the confining potential. In this paper, we remove the assumption of
spherical symmetry for the defocusing harmonic oscillator, and also
establish global wellposedness for the focusing problem under the
assumption that Conjecture \eqref{conjecture:ckstt} holds for all
dimensions. Specifically, we prove
\begin{thm}
\label{thm:main_theorem}
Assume that Conjecture \ref{conjecture:ckstt} holds. Then Conjecture 
\ref{conjecture:harmonic_oscillator} holds.
\end{thm}

By Theorem \ref{thm:ckstt}, this result is conditional only in the
focusing situation for nonradial data in dimensions $3$ and
$4$. Moreover, in the focusing case we have essentially the same
blowup result as for the potential-free NLS with the same proof as in
that case; see \cite{kv_focusing_nls}. We recall the argument in
Section \ref{section:focusing_blowup}.

\begin{thm}[Blowup]
\label{thm:focusing_blowup}
Suppose $\mu = -1$ and $d \ge 3$. If $u_0 \in \Sigma$ satisfies $E(u_0) < E_{\Delta}(W)$ and 
$\| \nabla u_0\|_2 \ge \| \nabla W\|_2$, then the solution to \eqref{eqn:nls_with_quadratic_potential} 
blows up in finite time.
\end{thm}

Mathematically, the energy-critical NLS with quadratic potential has
several interesting properties. On one hand, it is a nontrivial
variant of the potential-free equation. If the quadratic potential is
replaced by a weaker potential, the proof of global wellposedness can
sometimes ride on the coat tails of Theorem \ref{thm:ckstt}.  For
example, we show in Section \ref{section:bounded_linear_potentials}
that for smooth, bounded potentials with bounded derivative, one
obtains global wellposedness by treating the potential as a
perturbation to \eqref{eqn:ckstt_eqn}.  Further, the Avron-Herbst
formula given in \cite{carles_time-dependent_potential} reduces the
problem with a linear potential $V(x) = Ex$ to
\eqref{eqn:ckstt_eqn}. On the other hand, the quantum harmonic
oscillator enjoys the remarkable property that its linear propagator
$e^{-itH}$ has an explicit integral kernel in the form of Mehler's
formula. In view of the preceding remarks, we believe that
\eqref{eqn:nls_with_quadratic_potential} is the most accessible
generalization of \eqref{eqn:ckstt_eqn} which does not come for free.

\textbf{Proof outline}. The local theory for
\eqref{eqn:nls_with_quadratic_potential} shows that global existence
is equivalent to the uniform \emph{a priori} spacetime bound
\eqref{eqn:main_theorem_spacetime_bound}.  To prove this bound for all
solutions, we apply the general strategy of induction on energy
pioneered by Bourgain \cite{bourgain_nls_radial_gwp} and refined over
the years by Colliander-Keel-Staffilani-Takaoka-Tao \cite{ckstt},
Keraani \cite{keraani_nls_blowup}, Kenig-Merle
\cite{kenig-merle_focusing_nls}, and others.  These arguments proceed
roughly as follows. One assumes for a contradiction that Theorem
\ref{thm:main_theorem} fails, and is then faced with two main tasks:
\begin{itemize}
\item [(1)] Prove the existence of a minimal counterexample (where 
``minimal" will be defined shortly). 
\item [(2)] Show that this counterexample violates properties obeyed 
by all solutions.
\end{itemize}

Let us elaborate a little on these steps. The energy $E(u)$ of a
solution (which equals the energy of the initial data) will serve as
our induction parameter.  By the local wellposedness theory (which we
review in Section \ref{section:local_theory}), uniform spacetime
bounds hold for all solutions with sufficiently small energy
$E(u)$. Assuming that Theorem~\ref{thm:main_theorem} fails, we obtain
a positive threshold $0 < E_c < \infty$ such that
\eqref{eqn:main_theorem_spacetime_bound} holds whenever $E(u) < E_c$
and fails when $E(u) > E_c$. The first step described above would have
us construct a solution $u_c : I \times \mf{R}^d \to \mf{C}$ on a
bounded interval $I$ with $S_I(u_c) = \infty$, and whose energy equals
precisely the critical threshold $E_c$.  Thus ``minimal" in our
setting refers to minimal energy.

We will carry out a variant of this strategy that is better adapted to
what we are trying to prove. Since the spacetime estimates of interest
are local-in-time, it suffices to prevent the blowup of spacetime norm
on arbitrarily small time intervals. That is, we need only show that
for each energy $E > 0$, there exists some $L = L(E) > 0$ so that
$S_I(u) \le C(E)$ whenever $E(u) \le E$ and $|I| \le L$. To prove this
statement, we work not with minimal-energy blowup solutions directly
but rather with the Palais-Smale compactness theorem (Proposition
\ref{prop:palais-smale}) that would beget such solutions.  Our
argument will ultimately reduce the question of global wellposedness
for \eqref{eqn:nls_with_quadratic_potential} to that of global
wellposedness and scattering for the potential-free equation
\eqref{eqn:ckstt_eqn}. In effect, we shall discover that the only
scenario where blowup could possibly occur is when the solution is
highly concentrated at a point and behaves like a solution
to \eqref{eqn:ckstt_eqn}.

This paradigm of recovering the potential-free NLS in certain limiting
regimes is by now well-known and has been applied to the study of
other equations.  See \cite{kksv_gkdv, 2d_klein_gordon, MR3006640,
  MR2925134, ionescu, kvz_exterior_convex_obstacle} for adaptations to
gKdV, Klein-Gordon, and NLS in various domains and manifolds. While
the particulars are unique to each case, a common key step is to prove
an appropriate compactness theorem in the style of Proposition
\ref{prop:palais-smale}. As in the previous work, our proof of that
proposition uses three main ingredients.

The first prerequisite is a local wellposedness theory that gives
local existence and uniqueness as well 
as stability of solutions with respect to perturbations of the 
initial data or the equation itself. In our case, local wellposedness 
will follow from familiar arguments employing the dispersive 
estimate satisfied 
by the linear propagator $e^{-itH}$, as well the fractional product and 
chain 
rules for the operators $H^{\gamma}, \ \gamma \ge 0$. We 
review the relevant results in Section \ref{section:local_theory}.

We also need a linear profile decomposition for 
the Strichartz inequality 
\begin{equation}
\label{eqn:intro_strichartz}
\|e^{-itH} f\|_{L^{\fr{2(d+2)}{d-2}}_{t, x}} 
\lesssim \|H^{\fr{1}{2}} f\|_{L^2_x}.
\end{equation}
Such a decomposition in the context of energy-critical Schr\"{o}dinger
equations was first proved by Keraani
\cite{keraani_compactness_defect} in the translation-invariant setting
for the free particle Hamiltonian $H = - \Delta$, and quantifies the
manner in which a sequence of functions $f_n$ with $\| H^{1/2}
f_n\|_{L^2}$ bounded may fail to produce a subsequence of $e^{-itH}
f_n$ converging in the spacetime norm. The defect of compactness
arises in Keraani's case from a noncompact group of symmetries of the
inequality \eqref{eqn:intro_strichartz}, which includes spatial
translations and scaling.  In our setting, there are no obvious
symmetries of \eqref{eqn:intro_strichartz}; nonetheless, compactness
can fail and in Section \ref{section:lpd} we formulate a profile
decomposition for \eqref{eqn:intro_strichartz} when $H$ is the
Hamiltonian of the harmonic oscillator.

The final ingredient 
is an analysis of \eqref{eqn:nls_with_quadratic_potential} when the 
initial data is highly concentrated 
in space, corresponding to a single profile in the linear profile 
decomposition just discussed. 
In Section \ref{section:localized_initial_data}, we show 
that blowup cannot occur in this regime. The basic idea is that 
while the solution 
to \eqref{eqn:nls_with_quadratic_potential} remains highly localized 
in space, it can be well-approximated  
up to a phase factor by the corresponding 
solution to the scale-invariant energy-critical 
NLS
\begin{equation}
\label{eqn:intro_ckstt}
(i\partial_t + \tfr{1}{2}\Delta)  u = \pm |u|^{\fr{4}{d-2}} u.
\end{equation}
By the time this approximation breaks down, the solution to the 
original equation will have dispersed and can instead be approximated by
 a solution 
to the linear equation $(i\partial_t - H)u = 0$.  
We use as a black box the nontrivial fact (which is still a conjecture
in several cases) that solutions to \eqref{eqn:ckstt_eqn}
obey global spacetime bounds. By stability theory, the spacetime bounds 
for the approximations will be transferred to the solution for 
the original equation and will therefore preclude blowup.

We have chosen to focus on the concrete potential $V(x) =
\tfr{1}{2}|x|^2$ mainly for concreteness. In a forthcoming paper we will
indicate how to extend the main result to a more general class of
subquadratic potentials.

\subsection*{Acknowledgements}
The author is indebted to his advisors Rowan Killip and Monica Visan
for their helpful discussions as well as their feedback on the
paper. This work was supported in part by NSF grants DMS-0838680
(RTG), DMS-1265868 (PI R.~Killip), DMS-0901166, and
DMS-1161396 (both PI M.~Visan).



\section{Preliminaries}

\subsection{Notation and basic estimates} 
\label{subsection:notation}
We write $X \lesssim Y$ to mean $X \le C Y$ for some constant
$C$. Similarly $X \sim Y$ means $X \lesssim Y$ and $Y \lesssim
X$. Denote by $L^p(\mf{R}^d)$ the Banach space of functions $f: \mf{R}^d \to \mf{C}$ with finite norm
\[
\|f\|_{L^p(\mf{R}^d)} = \left( \int_{\mf{R}^d}  |f|^p \, dx \right)^{\fr{1}{p}}.
\]
We will sometimes use the more compact notation
$\|f\|_p$. If $I \subset \mf{R}^d$ is an interval, the mixed Lebesgue
norms on $I \times \mf{R}^d$ are defined by
\[
\| f\|_{L^q_t L^r_x (I \times \mf{R}^d)} = \left( \int_I \left( \int_{\mf{R}^d} 
|f(t, x)|^r \, dx \right)^{\fr{q}{r}} dt 
\right)^{\fr{1}{q}} = \| f(t) \|_{L^q_t(I ; L^r_x(\mf{R}^d))},
\]
where one regards $f(t) = f(t, \cdot)$ as a function from $I$ to 
$L^r(\mf{R}^d)$.

The operator $H = -\tfr{1}{2}\Delta + \tfr{1}{2}|x|^2$ is positive on
$L^2(\mf{R}^d)$. Its associated heat kernel is given by Mehler's
formula \cite{folland}:
\begin{equation}
e^{-tH} (x, y) = e^{\tilde{\gamma}(t) (x^2 + y^2)} e^{\fr{\sinh (t) \Delta}{2}} 
(x, y) \label{eqn:mehler_heat_kernel}
\end{equation}
where 
\[
\tilde{\gamma}(t) = \fr{1 - \cosh t}{2\sinh t} = -\fr{t}{4} + O(t^3)
\quad \text{as} \quad t \to 0.
\]
By analytic continuation, the associated one-parameter unitary group
has the integral kernel
\begin{equation}
e^{-itH} f (x) = \fr{1}{ (2\pi i \sin t )^{\fr{d}{2}} } \int e^{\fr{i}{\sin t} 
\left(\fr{x^2+y^2}{2} \cos t - xy 
\right) } 
f(y) \, dy. \label{eqn:mehler}
\end{equation}
Comparing this to the well-known free propagator
\begin{equation}
e^{\fr{it\Delta}{2}} f (x) = \tfr{1}{ (2\pi it )^{\fr{d}{2}}} \int e^{ \fr{ i 
|x-y|^2 }{ 2t } } f(y) \, dy,
\end{equation}
we obtain the relation
\begin{equation}
e^{-itH}f = e^{i\gamma(t)|x|^2} e^{\fr{i \sin(t) \Delta}{2}}( e^{i\gamma(t)|x|^2} f) \label{eqn:lens_transformation}
\end{equation}
where 
\[
\gamma(t) = \fr{\cos t - 1}{2\sin t} = - \fr{t}{4} + O(t^3) \quad
\text{as} \quad t \to 0.
\]
Mehler's formula immediately implies the local-in-time dispersive estimate
\begin{equation}
\| e^{-itH} f\|_{L^\infty_x} \lesssim |\sin t|^{-\fr{d}{2}}
\|f\|_{L^1}. 
\label{eqn:dispersive_estimate}
\end{equation}
For $d \ge 3$, call a pair of exponents $(q, r)$ \emph{admissible} if $q \ge 2$ and $\fr{2}{q} 
+ \fr{d}{r} = \fr{d}
{2}$. Write
\[
\|f\|_{S(I)} = \| f\|_{L^\infty_t L^2_x } + \|
f\|_{L^2_t L^{\fr{2d}{d-2}}_x}
\]
with all norms taken over the spacetime slab $I \times \mf{R}^d$. 
By interpolation, we see that this norm controls the $L^q_t L^r_x$
norm for all other  admissible pairs. We let 
\[
\|F\|_{N(I)}  = \inf\{ \|F_1\|_{L^{q_1'}_t L^{r_1'}_x} +
\|F_2\|_{L^{q_2'}_t L^{r_2'}_x} : (q_k, r_k) \ \text{admissible}, \ F
= F_1 + F_2\},
\] 
where $(q_k', r_k')$ is the H\"{o}lder dual to $(q_k, r_k)$.

\begin{lma}[Strichartz estimates]
\label{lma:strichartz}
Let $I$ be a compact time interval containing $t_0$, and let $u: I \times 
\mf{R}^d \to \mf{C}$ be a 
solution to the inhomogeneous Schr\"{o}dinger equation
\begin{equation}
(i\partial_t - H)u = F. \nonumber
\end{equation}
Then there is a constant $C = C(|I|)$, depending only on the length of the 
interval, such that
\[
\| u\|_{S(I)} \le C ( \| u(t_0) \|_{L^2} + \| F\|_{N(I)}).
\]
\end{lma}
\begin{proof}
  This follows from the dispersive estimate \eqref{eqn:dispersive_estimate}, the unitarity of $e^{-itH}$ on $L^2$, and general
  considerations; see \cite{keel-tao}. By partitioning time
  into unit intervals, we see that the constant $C$ grows 
  at worst like $|I|^{\fr{1}{2}}$ (which corresponds to the 
  time exponent $q = 2$). 
\end{proof}

 It will be convenient to introduce the 
 operators which represent the time  
 evolution of the momentum and position operators under the linear
 propagator. These are well-known in the literature and were used in 
 \cite{carles_confining_oscillator} or
 \cite{kvz_quadratic_potentials}, for example. We define
 \begin{equation}
 \label{eqn:time_evolution_of_position_and_momentum}
 \begin{split}
 P(t) &= e^{itH} i\nabla e^{-itH} = i\nabla \cos t - x \sin t \\
 X(t) &= e^{itH} x e^{-itH} = i\nabla \sin t + x \cos t.
 \end{split}
 \end{equation}
One easily verifies the identity
\[
\| P(t) f\|_{L^2}^2 + \| X(t)f \|_{L^2}^2 = \| P(t) f\|_{L^2}^2 + \|P(t + \tfr{\pi}{2}) f\|_{L^2}^2
=\| f\|_{\Sigma}^2.
\]

We use the fractional powers $H^{\gamma}$ of the operator $H$,
defined via the Borel functional calculus, as a substitute for the usual derivative 
$(-\Delta)^{\gamma}$, which 
does not 
commute with the linear propagator $e^{-itH}$. We have trivially that
\begin{align*}
\| H^{\fr{1}{2}} f\|_{L^2} \sim  \| (-\Delta)^{\fr{1}{2}} f\|_{L^2} + \| |x| 
f\|_{L^2} \sim \| f\|_{\Sigma}.
\end{align*}
Perhaps less obvious is the fact that this equivalence generalizes to other
$L^p$ norms and other powers of $H$. Using complex interpolation,
Killip, Visan, and Zhang showed that this is the case:
\begin{lma}[{\cite[Lemma 2.7]{kvz_quadratic_potentials}}]
\label{lma:equivalence_of_norms}
For $0 \le \gamma \le 1$ and $1 < p < \infty$, one has 
\[
\| H^{\gamma} f\|_{L^p(\mf{R}^d)} \sim \| (-\Delta)^{\gamma} f\|_{L^p 
(\mf{R}^d) } + \| |x|
^{2\gamma} f\|_{L^p(\mf{R}^d)}.
\]
\end{lma}
As a consequence, $H^\gamma$ inherits many properties of
$(-\Delta)^{\gamma}$, including Sobolev embedding:
\begin{lma}[{\cite[Lemma 2.8]{kvz_quadratic_potentials}}]
\label{lma:sobolev_embedding}
Suppose $\gamma \in [0, 1]$ and $1 < p < \fr{d}{2\gamma}$, and define $p^*$
by $\fr{1}{p^*} = \fr{1}{p} - \fr{2\gamma}{d}$. Then
\[
\| f\|_{L^{p^*}(\mf{R}^d)} \lesssim \| H^{\gamma} f\|_{L^p(\mf{R}^d)}.
\]
\end{lma}

Similarly, the fractional chain and product rules carry over to the 
current setting: 

\begin{cor}[{\cite[Proposition~2.10]{kvz_quadratic_potentials}}]
\label{cor:fractional_chain_rule}
Let $F(z) = |z|^{\fr{4}{d-2}} z$. For any $0 \le \gamma \le \fr{1}{2}$ and 
$1 < p < \infty$,
\[
\| H^{\gamma} F(u)\|_{L^p(\mf{R}^d)} \lesssim \| F'(u)\|_{L^{p_0}(\mf{R}^d)} \| 
H^{\gamma}f \|
_{L^{p_1}(\mf{R}^d)}
\]
for all $p_0, p_1 \in (1, \infty)$ with $p^{-1} = p_0^{-1} +
p_1^{-1}$. 
\end{cor}

Using Lemma~\ref{lma:equivalence_of_norms} and the Christ-Weinstein
fractional product rule for $(-\Delta)^\gamma$
(e.g. \cite{taylor_tools}), we obtain
\begin{cor}
\label{cor:fractional_product_rule}
For $\gamma \in (0, 1], \  r, p_i, q_i \in (1, \infty)$ with $r^{-1} = p_i^{-1} +
q_{i}^{-1}, \ i = 1, 2$, we have 
\[
\| H^{\gamma} (fg)\|_r \lesssim \| H^{\gamma} f\|_{p_1} \|g\|_{q_1} +
\|f\|_{p_2} \|H^{\gamma} g\|_{q_2}.
\]

\end{cor}

The exponent $\gamma = \fr{1}{2}$ is particularly relevant to us, and it 
will be convenient to use the notation $L^{q}_t \Sigma^r_x( I \times 
\mf{R}^d)$ for the space 
of functions $f$ with norm
\[
\| f\|_{L^q_t \Sigma^r_x} = \| H^{\fr{1}{2}} f\|_{L^q_t L^r_x}.
\]
The superscript on $\Sigma$ is assumed to be $2$ if omitted. We shall need 
the following refinement of Fatou's Lemma due to Br\'{e}zis and 
Lieb:
\begin{lma}[Refined Fatou \cite{brezis_lieb}]
\label{lma:refined_fatou}
Fix $1 \le p < \infty$, and suppose $f_n$ is a sequence of functions in 
$L^p(\mf{R}^d)$ such that $\sup_n \|f_n\|_{p} < \infty$ and $f_n \to f$ 
pointwise. Then
\[
\lim_{n \to \infty} \int_{\mf{R}^d} \left| |f_n|^p - |f_n - f|^p - |f|^p 
\right| \, dx = 0.
\]

\end{lma}

Finally, we record an analogue of the H\"{o}rmander-Mikhlin Fourier
multiplier theorem proved by Hebisch \cite{hebisch}. It enables a
Littlewood-Paley theory adapted to $H$, as discussed in the next
section.
\begin{thm}
\label{thm:hebisch}
If $F: \mf{R} \to \mf{C}$ is a bounded function which obeys the derivative 
estimates
\[
| \partial^k F(\lambda)| \lesssim_k |\lambda|^{-k} \quad \text{for all} \quad 0
\le k \le \tfr{d}{2}  + 1, 
\]
then the operator $F(H)$, defined initially on $L^2$ via the Borel
functional calculus, is bounded from $L^p$ to 
$L^p$ for all $1 < p 
< \infty$.
\end{thm}

\subsection{Littlewood-Paley theory}
\label{subsection:LP_theory}
Owing largely to Theorem~\ref{thm:hebisch}, we can import the basic
results of Littlewood-Paley theory with little effort, the only
change being that one replaces Fourier multipliers with spectral
multipliers. We fashion two kinds of Littlewood-Paley projections, one
using compactly supported bump functions, and the other based on
the heat kernel of $H$. The parabolic 
maximum principle implies that
\begin{equation}
0 \le e^{-tH}(x, y) \le e^{\fr{t\Delta}{2}}(x, y) = \tfr{1}{(2\pi
  t)^{d/2}} e^{-\fr{|x-y|^2}{2t}}. \label{eqn:heat_kernel_comparison}
\end{equation}

Fix a smooth function $\varphi$ supported in $|\lambda| \le 2$ with
$\varphi(\lambda) = 1$ for $|\lambda| \le 1$, and let $\psi(\lambda) = \varphi(\lambda) -
\varphi(2\lambda)$. For each dyadic number $N \in 2^{\mf{Z}}$, which we
will often refer to as ``frequency,'' define
\begin{alignat*}{2}
P^H_{\le N} &= \varphi( \sqrt{ H / N^2}), &\quad P^H_N &= \psi(\sqrt{H/N^2}),\\
\tilde{P}^H_{\le N} &= e^{-H/N^2}, &  \tilde{P}^H_{N} &= e^{-H/N^2} -
e^{-4H / N^2}.
\end{alignat*}
The associated operators $P^H_{<N}, P^H_{>N}$, etc. are defined in the usual
manner. 
\begin{rmk}
As the spectrum of $H$ is bounded away from $0$, by choosing
$\varphi$ appropriately we can arrange for $P_{< 1} = 0$; thus we will
only consider frequencies $N \ge 1$.
\end{rmk}
 Similarly, let 
\begin{alignat}{2}
P^\Delta_{\le N} &= \varphi( \sqrt{ -\Delta / N^2}) &\quad P^\Delta_N &= \psi(\sqrt{-\Delta/N^2}),\\
\tilde{P}^\Delta_{\le N} &= e^{\Delta/2N^2} & \tilde{P}^\Delta_{N} &= e^{\Delta/2N^2} -
e^{2\Delta / N^2}.
\end{alignat}
denote the classical Littlewood-Paley projections. From the maximum 
principle we obtain the pointwise bound
\begin{equation}
\label{eqn:Littlewood-Paley_pointwise_comparison}
|\tilde{P}^H_N f (x)| + |\tilde{P}^H_{\le N} f(x) | \lesssim
\tilde{P}_{\le N}^{\Delta} |f| (x) + \tilde{P}_{\le N/2}^{\Delta} 
|f|(x).
\end{equation}

To reduce clutter we usually suppress the
superscripts $H$ and $\Delta$ when it is clear from the context which
type of projection we are using. For the rest of this section,
$P_{\le N}$ and $P_N$ denote $P^H_{\le N}$ and $P^H_{N}$, respectively. 
\begin{lma}[Bernstein estimates]
\label{lma:Littlewood-Paley_facts}
For $f \in C^\infty_c(\mf{R}^d), \ 1 < p \le q < \infty, \ s \ge
0$, one has the Bernstein inequalities
\begin{gather}
\| P_{\le N} f\|_p \lesssim \| \tilde{P}_{\le N} f\|_p, \quad \| P_N
f\|_p \lesssim \| \tilde{P}_N f\|_p \label{eqn:Littlewood-Paley_projector_comparison}\\
\| P_{\le N} f \|_p + \| P_N f \|_p + \| \tilde{P}_{\le N} f \|_p + \|
\tilde{P}_N f \|_p \lesssim \|f\|_p \label{eqn:Littlewood-Paley_Lp_boundedness}\\
\| P_{\le N} f \|_q + \| P_N f \|_q + \| \tilde{P}_{\le N} f\|_q + \|
\tilde{P}_N f \|_q \lesssim N^{\fr{d}{p} - \fr{d}{q}}
\|f\|_p \label{eqn:Littlewood-Paley_Bernstein}\\
N^{2s} \| P_N f\|_p \sim \| H^s P_N f \|_p \label{eqn:Littlewood-Paley_derivative_estimates}\\
  \|P_{>N} f\|_p \lesssim N^{-2s} \| H^{s} P_{>N} f\|_p. \label{eqn:Littlewood-Paley_Bernstein_2}
\end{gather}
In \eqref{eqn:Littlewood-Paley_Bernstein}, the estimates for
$\tilde{P}_{\le N}f$ and $\tilde{P}_{N}f$ also hold when $p = 1, \ q =
\infty$. Further,
\begin{align}
f = \sum_N P_N f = \sum_N \tilde{P}_N f \label{eqn:Littlewood-Paley_expansion_of_the_identity}
\end{align}
where the series converge in $L^p, \ 1 < p < \infty$. Finally, we have the 
square function estimate
\begin{align}
\| f \|_p \sim \| ( \sum_N |P_N f|^2 )^{1/2} \|_p. \label{eqn:Littlewood-Paley_square_function_estimate}
\end{align}
\end{lma}

\begin{proof}
The estimates \eqref{eqn:Littlewood-Paley_projector_comparison} follow
immediately from Theorem~\ref{thm:hebisch}. 
To see \eqref{eqn:Littlewood-Paley_Lp_boundedness}, observe
that the functions $\varphi(\sqrt{ \cdot / N^2}), \ e^{-\cdot / N^2}$
satisfy the hypotheses of Theorem~\ref{thm:hebisch} uniformly in
$N$. Next use \eqref{eqn:heat_kernel_comparison} together with Young's convolution inequality to
get
\begin{equation}
\| \tilde{P}_{\le N} f \|_q + \| \tilde{P}_N f\|_q \lesssim
N^{\fr{d}{q} - \fr{d}{p}} \|f\|_p \quad \text{for } 1 \le p \le q \le \infty.
\end{equation}
From \eqref{eqn:Littlewood-Paley_projector_comparison} we obtain the
rest of \eqref{eqn:Littlewood-Paley_Bernstein}. Now consider \eqref{eqn:Littlewood-Paley_derivative_estimates}. Let
$\tilde{\psi}$ be a fattened version of $\psi$ so that $\tilde{\psi} =
1$ on the support of $\psi$. Put $F(\lambda) = \lambda^{s} \tilde{\psi}(
\sqrt{\lambda})$. By
 Theorem~\ref{thm:hebisch}, the relation $\psi =
\tilde{\psi} \psi$, and the functional calculus,
\[
\| N^{-2s} H^s P_Nf\|_p = \| F(H/N^2)P_N f\|_p \lesssim \| P_N f\|_p.
\]
The reverse inequality follows by considering $F(x) = \lambda^{-s} \tilde{\psi}(\lambda)$.

We turn to \eqref{eqn:Littlewood-Paley_expansion_of_the_identity}. The
equality holds in $L^2$ by
the functional calculus and the fact that the spectrum of $H$ is
bounded away from $0$. For $p \ne 2$,
choose $q$ and $0 < \theta < 1$ so that $p^{-1} = 2^{-1} (1-\theta)
+ q^{-1}\theta $. By \eqref{eqn:Littlewood-Paley_Lp_boundedness}, the partial sum operators 
\[
S_{N_0, N_1} = \sum_{N_0 < N \le N_1} P_N, \qquad \tilde{S}_{N_0, N_1} =
\sum_{N_0 < N \le N_1} \tilde{P}_N
\]
are bounded on every $L^p, \ 1 < p < \infty$, uniformly in $N_0, N_1$.
Thus by H\"{o}lder's inequality,
\[
\|f - S_{N_0, N_1} f\|_p \le \|f-S_{N_0, N_1} f\|_2^{1-\theta} \| f -
S_{N_0, N_1} f \|_{q}^\theta \to 0 \text{ as } N_0 \to 0, \ N_1 \to \infty,
\]
and similarly for the partial sums $\tilde{S}_{N_0, N_1} f$. The estimate
\eqref{eqn:Littlewood-Paley_Bernstein_2} follows from 
\eqref{eqn:Littlewood-Paley_Lp_boundedness}, \eqref{eqn:Littlewood-Paley_derivative_estimates}, 
and the decomposition $P_{>N} f = \sum_{M > N} P_M f$. 

To prove the square function estimate, run the usual Khintchine's
inequality argument using Theorem~\ref{thm:hebisch} in place of the
Mikhlin multiplier theorem.
\end{proof}

\subsection{Local smoothing}
The following local smoothing lemma and its corollary will be needed when proving properties 
of the nonlinear profile decomposition in Section \ref{section:palais-smale}. 
\begin{lma}
If $u = e^{-itH} \phi, \ \phi \in \Sigma(\mf{R}^d)$, then
\begin{align*}
\int_I \int_{\mf{R}^d} | \nabla u(x)|^2 \langle R^{-1} (x-z)
\rangle^{-3} \, dx \, dt 
\lesssim R( 1 + |I| ) \| u\|_{L^\infty_tL^2_x} \| H^{1/2} 
u\|_{L^\infty_t L^2_x}.
\end{align*}
with the constant independent of $z \in \mf{R}^d$ and $R > 0$.
\end{lma}

\begin{proof}
We recall the Morawetz identity. Let $a$ be a sufficiently smooth
function of $x$; then for any $u$ satisfying the linear equation
$i\partial_t u = (-\tfr{1}{2}\Delta + V)u$, one has
\begin{equation}
\label{eqn:morawetz_identity}
\begin{split}
\partial_t \int \nabla a \cdot \opn{Im}( \overline{u}\nabla u) \, dx &= \int a_{jk} \opn{Re}(u_j 
\overline{u}
_k) \, dx - \tfr{1}{4} \int |u|^2 a_{jjkk} \, dx \\
&\quad - \tfr{1}{2} \int |u|^2
\nabla a \cdot \nabla V \, dx
\end{split}
\end{equation}
We use this identity with $a(x) =  \langle R^{-1}(x - z)
\rangle$ and $V = \fr{1}{2}|x|^2$, and compute
\begin{align*}
a_j(x) &= \fr{ R^{-2} (x_j - z_j)}{ \langle R^{-1}(x-z) \rangle}, \ a_{jk} (x) = R^{-2} 
\left[ \fr{\delta_{jk}}{ \langle R^{-1}(x-z) \rangle} - \fr{ R^{-2} (x_j - z_j)(x_k - z_k) }{ \langle R^{-1}(x-z)\rangle^3} \right]\\
\Delta^2 a(x) &\le -\fr{15 R^{-4}}{ \langle R^{-1}(x-z) \rangle^7}.
\end{align*}
As $\Delta^2a \le 0$, the right side of \eqref{eqn:morawetz_identity}
is bounded below by
\[
\begin{split}
 &R^{-2} \int \langle R^{-1} (x-z)
\rangle^{-1} \left[ |\nabla u|^2 - | \tfr{R^{-1}(x - z)}{ \langle R^{-1}(x
      - z)\rangle} \cdot \nabla u |^2 \right] \, dx - \tfr{1}{2R} \int |u|^2
  \tfr{ R^{-1}(x - z)}{ \langle R^{-1}(x-z) \rangle} \cdot x  \, dx\\
&\ge R^{-2} \int |\nabla u(x)|^2 \langle R^{-1}(x-z) \rangle^{-3} \, dx -
\tfr{R^{-1}}{2} \int |u|^2 |x| \, dx.
\end{split}
\]
Integrating in time and applying Cauchy-Schwarz, we get
\[
\begin{split}
R^{-2} \int_{I} \int_{\mf{R}^d} &\langle R^{-1}(x-z) \rangle^{-3} |\nabla
u(t, x)|^2 \, dx dt \\
&\lesssim \sup_{t \in I} R^{-1} \int \tfr{R^{-1}(x-z)}{\langle R^{-1}(x-z)
  \rangle} |u(t, x)| |\nabla u(t, x)| \, dx + \tfr{1}{2R} \int_I \int_{\mf{R}^d}|x|
|u|^2 \, dx dt\\
&\lesssim R^{-1} (1 + |I|) \|u\|_{L^\infty_t L^2_x} \|
H^{1/2}u\|_{L^\infty_t L^2_x}.
\end{split}
\]
This completes the proof of the lemma.
\end{proof}

\begin{cor}
\label{cor:local_smoothing_cor}
Fix $\phi \in \Sigma(\mf{R}^d)$. Then for all $T, R \le 1$, we have
\begin{align*}
\| \nabla e^{-itH} \phi \|_{L^{2}_{t,x} (|t-t_0| \le T, \ |x-x_0| \le R)}  
\lesssim T^{\fr{2}{3(d+2)}} 
R^{ \fr{3d+2}{ 3(d+2)}} \| \phi \|^{\fr{2}{3}}_{\Sigma} 
\| e^{-itH} \phi\|_{L^{\fr{2(d+2)}{d-2}}_{t,x}}
^{\fr{1}{3}}.
\end{align*}
When $d = 3$, we also have
 \begin{align*}
\| \nabla e^{-itH} \phi \|_{L^{\fr{10}{3}}_t L^{\fr{15}{7}}_x (  |t -t_0| \le T, 
\ |x-x_0| \le R)} \lesssim 
T^{\fr{23}{180}} R^{\fr{11}{45}} \| e^{-itH} \phi\|_{L^{10}_{t, x}}^{ 
\fr{5}{48}} \| \phi\|_{\Sigma}^{\fr{43}{48}}
\end{align*}
\end{cor}

\begin{proof}

  The proofs are fairly standard (see \cite{oberwolfach} or
  \cite{kvz_exterior_convex_obstacle}), and we present just the proof
  of the second claim, which is slightly more involved. Let $E$ the
  region $\{ |t-t_0| \le T, \ |x-x_0| \le R\}$. Norms which do
  not specify the region of integration are taken over the spacetime
  slab $\{|t-t_0| \le T\} \times \mf{R}^3$.  By H\"{o}lder,
\begin{align*}
\| \nabla e^{-itH} \phi \|_{L^{\fr{10}{3}}_t L^{\fr{15}{7}}_x(E) } \le \| \nabla 
e^{-itH} \phi\|_{L^{2}_{t, x}(E)}^{\fr{1}{3}} \| \nabla e^{-itH} \phi \|_{L^5_t L^{\fr{20}{9}}_x(E)}^{\fr{2}{3}}.
\end{align*}
By H\"{o}lder and Strichartz,
\begin{align}
\| \nabla e^{-itH} \phi \|_{L^5_t L^{\fr{20}{9}}_x(E)} \lesssim T^{\fr{1}{8}} \| 
\nabla e^{-itH} \phi \|_{L^{ \fr{40}
{3}}_t L^{ \fr{20}{9}}_x } \lesssim T^{\fr{1}{8}} \| \phi\|_{\Sigma}.  
\label{eqn:local_smoothing_cor_eqn_1}
\end{align}
We now
estimate $\| \nabla e^{-itH} \phi\|_{L^{2}_{t, x}}$. Let $N \in
2^{\mf{N}}$ be a dyadic number to be chosen later, and decompose
\begin{align*}
\| \nabla e^{-itH} \phi \|_{L^2_{t, x}(E)} \le \| \nabla e^{-itH} P^H_{\le N} \phi \|_{L^{2}_{t, x}(E)} + \| \nabla 
e^{-itH} P^H_{>N} \phi \|_{L^{2}_{t, x}(E)}.
\end{align*}
For the low frequency piece, apply H\"{o}lder and the Bernstein inequalities 
to obtain
\begin{align*}
\| \nabla e^{-itH} P^H_{\le N} \phi\|_{L^2_{t, x}} \lesssim
T^{\fr{2}{5}} R^{\fr{6}{5}} \| \nabla e^{-itH} P^H_{\le N}
\phi\|_{L^{10}_{t,x}} \lesssim T^{\fr{2}{5}} R^{\fr{6}{5}} N \|
e^{-itH} \phi \|_{L^{10}_{t,x}}.
\end{align*}
For the high-frequency piece, apply local smoothing and Bernstein:
\begin{align*}
\| \nabla e^{-itH} P^H_{>N} \phi\|_{L^2_{t, x}} \lesssim R^{\fr{1}{2}} \| P^H_{>N} \phi\|_{L^2}^{\fr{1}{2}} \| 
H^{\fr{1}{2}} \phi\|_{\Sigma}^{\fr{1}{2}} \lesssim R^{\fr{1}{2}} N^{-\fr{1}{2}} \| \phi\|_{\Sigma}.
\end{align*}
Optimizing in $N$, we obtain
\begin{align*}
\| \nabla e^{-itH} \phi\|_{L^{2}_{t, x}} \lesssim T^{\fr{2}{15}}
R^{\fr{11}{15}} \| e^{-it H} \phi\|_{L^{10}_{t,x}}^{\fr{1}{3}} \|
\phi\|_{\Sigma}^{\fr{2}{3}}. 
\end{align*}
Combining this estimate with \eqref{eqn:local_smoothing_cor_eqn_1} yields the 
conclusion of the corollary.
\end{proof}

\section{Local theory}
\label{section:local_theory}

We record some standard results concerning local-wellposedness
for \eqref{eqn:nls_with_quadratic_potential}. These are direct
analogues of the theory for the scale-invariant equation \eqref{eqn:ckstt_eqn}. By Lemma~\ref{lma:sobolev_embedding} and Corollaries
\ref{cor:fractional_chain_rule} and \ref{cor:fractional_product_rule},
we can use essentially the same proofs as in that case. We refer
the reader to \cite{claynotes} for those proofs.

\begin{prop}[Local wellposedness]
\label{prop:lwp}
Let $u_0 \in \Sigma(\mf{R}^d)$ and fix a compact time interval $0 \in I \subset 
\mf{R}$. Then there 
exists a constant $\eta_0 = \eta_0(d, |I|)$ such that whenever $\eta < \eta_0$ 
and 
\[
\| H^{\fr{1}{2}} e^{-itH} u_0 \|_{L^{\fr{2(d+2)}{d-2}}_t 
L^{\fr{2d(d+2)}{d^2+4}}_x (I \times \mf{R}^d)} 
\le \eta,
\]
there exists a unique solution $u: I \times \mf{R}^d \to \mf{C}$ to 
\eqref{eqn:nls_with_quadratic_potential} which satisfies the bounds
\begin{align*}
\| H^{\fr{1}{2}}  u \|_{L^{\fr{2(d+2)}{d-2}}_t L^{\fr{2d(d+2)}{d^2+4}}_x (I 
\times \mf{R}^d)} \le 2\eta \quad \text{and} \quad
\| H^{\fr{1}{2}} u \|_{S(I)} \lesssim \| u_0\|_{\Sigma} + 
\eta^{\fr{d+2}{d-2}}.
\end{align*}
\end{prop}

\begin{cor}[Blowup criterion]
\label{cor:blowup_criterion}
Suppose $u : (T_{\text{min}}, T_{\text{max}}) \times \mf{R}^d \to \mf{C}$ is a 
maximal lifespan 
solution to \eqref{eqn:nls_with_quadratic_potential}, and fix $T_{\text{min}} < t_0 < T_{\text{max}}$. 
If $T_{\text{max}} < 
\infty$, then 
\[
\| u\|_{L^{\fr{2(d+2)}{d-2}}_{t, x}( [t_0, T_{\text{max}} ))} = \infty.
\]
If $T_{\text{min}} > -\infty$, then
\[
\| u \|_{L^{\fr{2(d+2)}{d-2}}_{t, x} ( (T_{\text{min}}, t_0])} = \infty.
\]
\end{cor}

\begin{prop}[Stability]
\label{prop:stability}
Fix $t_0 \in I \subset \mf{R}$ an interval of unit length and let $\tilde{u}: 
I \times \mf{R}^d \to 
\mf{C}$ be 
an approximate solution to \eqref{eqn:nls_with_quadratic_potential} in the 
sense that 
\[
i\partial_t \tilde{u} = H u \pm |\tilde{u}|^{\fr{4}{d-2}} \tilde{u} + e
\]
for some function $e$. Assume that 
\begin{equation}
\label{eqn:stability_prop_hyp_1} 
\| \tilde{u}\|_{L^{\fr{2(d+2)}{d-2}}_{t,x}} \le L,
\quad \| H^{\fr{1}{2}}u \|_{L^\infty_t L^2_x} \le E,
\end{equation}
and that for some $0 < \varepsilon < \varepsilon_0(E, L)$ one has
\begin{equation}
\label{eqn:stability_prop_hyp_2} 
\| \tilde{u}(t_0) - u_0 \|_{\Sigma} + \| H^{\fr{1}{2}} e \|_{N(I)} \le \varepsilon,
\end{equation}
Then there exists a unique solution $u : I \times \mf{R}^d 
\to \mf{C}$ to \eqref{eqn:nls_with_quadratic_potential} with $u(t_0) = u_0$ 
and which further satisfies 
the estimates
\begin{equation}
\| \tilde{u} -  u \|_{L^{\fr{2(d+2)}{d-2}}_{t,x} } +
\| H^{\fr{1}{2}} (\tilde{u} - u) \|_{S(I)} \lesssim C(E, L) 
\varepsilon^c
\end{equation}
where $0 < c = c(d) < 1$ and $C(E, L)$ is a function which is nondecreasing
in each variable.
\end{prop}

\section{Concentration compactness}
\label{section:lpd}
In this section we discuss some concentration compactness 
results for the 
Strichartz inequality 
\[
\| e^{-itH} f\|_{L^{\fr{2(d+2)}{d-2}}_{t,x}(I \times \mf{R}^d)} 
\le C(|I|, d) \| f\|_{\Sigma},
\]
culminating in the linear profile decomposition of Proposition
\ref{prop:lpd}.  Our profile decomposition resembles that of
Keraani \cite{keraani_compactness_defect} in the sense that each
profile lives at a well-defined location in spacetime and has a
characteristic length scale. But since the function space
$\Sigma$ lacks both translation and scaling symmetry, the
precise definitions of our profiles will be more
complicated.

Keraani considered the analogous Strichartz estimate 
\[
\| e^{it\Delta} f\|_{L^{\fr{2(d+2)}{d-2}}_{t,x}(\mf{R} \times \mf{R}^d)} 
\lesssim \|f\|_{\dot{H}^1(\mf{R}^d)}.
\]
Recall that in that situation, if $f_n$ is a bounded sequence in
$\dot{H}^1$ with nontrivial linear evolution, then one has a
decomposition $f_n = \phi_n + r_n$ where $\phi_n = e^{it_n \Delta} G_n
\phi$, $G_n$ are certain unitary scaling and translation operators on
$\dot{H}^1$ (defined as in \eqref{eqn:inv_strichartz_frameops_1}),
and $\phi$ is a weak limit of $G_n^{-1} e^{-it_n \Delta} f_n$ in
$\dot{H}^1$. The ``bubble" $\phi_n$ is nontrivial and decouples from
the remainder $r_n$ in various norms. By applying this decomposition
inductively to the remainder $r_n$, one obtains the full collection of
profiles constituting $f_n$.

We follow the general presentation in
\cite{claynotes,oberwolfach}. Let $f_n \in \Sigma$ be a bounded
sequence.  Using a variant of Keraani's argument, we seek to obtain an
$\dot{H}^1$-weak limit $\phi$ in terms of $f_n$ and write $f_n =
\phi_n + r_n$ where $\phi_n$ is defined analogously as before by
``moving the operators onto $f_n$." However, we need to modify this
procedure in light of two issues.

The first problem is that while $f_n$ belong to $\Sigma$, an
$\dot{H}^1$ weak limit of a sequence like $G_n^{-1} e^{it_n H} f_n$
need only belong to $\dot{H}^1$. Indeed, the $\dot{H}^1$ isometries
$G_n^{-1}$ will in general have unbounded norm as operators on
$\Sigma$ because of the $|x|^2$ spatial weight, which penalizes very
wide functions. To define $\phi_n$, we need to introduce suitably
chosen spatial cutoffs to obtain functions in~$\Sigma$.

Secondly, to establish the various orthogonality assertions we must
understand how the linear propagator $e^{-itH}$ interacts with the 
$\dot{H}^1$ symmetries of translation and scaling in 
certain limits. We study this interaction in 
Section~\ref{subsection:convergence}. In particular, the convergence lemmas proved
there serve  as a substitute for the scaling relation
\[
e^{it\Delta} G_{n} = G_{n} e^{iN_n^2 t\Delta} \quad \text{where} \quad G_n \phi = 
N_n^{\fr{d-2}{2}}\phi(N_n(\cdot - x_n)).
\]
They can also be regarded as a precise form of the heuristic stated in
the introduction that as we scale the initial data to concentrate at a
point $x_0$, the potential $V(x) = |x|^2/2$ can be treated over short
time intervals as
essentially equal to the constant potential $V(x_0)$; hence for short
times the linear propagator $e^{-itH}$ can be approximated up to a
phase factor by the free particle propagator. In Section
\ref{section:localized_initial_data} we shall see a nonlinear
version of this statement.

\subsection{An Inverse Strichartz Inequality}
\label{subsection:inv_strichartz}

Unless indicated otherwise, $0 \in I$ in this section 
will denote a fixed interval of length at most $1$, and 
all spacetime norms will be taken over $I \times \mf{R}^d$.

Suppose $f_n$ is a sequence of functions in $\Sigma$ with nontrivial
linear evolution $e^{-itH} f_n$. The following refined Strichartz
estimate shows that there must be a ``frequency" $N_n$ which
makes a nontrivial contribution to the evolution.
\begin{prop}[Refined Strichartz]
\label{prop:refined_strichartz}
\[
\| e^{-itH} f\|_{L^{\fr{2(d+2)}{d-2}}_{t, x}} \lesssim
\|f\|_{\Sigma}^{\fr{4}{d+2}} \sup_N \| e^{-itH} P_N 
f\|_{L^{\fr{2(d+2)}{d-2}}_{t,x}}^{\fr{d-2}{d+2}}
\]
\end{prop}
\begin{proof}
  We quote essentially verbatim the proof of Refined Strichartz for
  the free particle propagator (\cite{oberwolfach} Lemma 3.1). Write
  $f_N$ for $P_N f$, where $P_N = P^H_N$ unless indicated
  otherwise. When $d \ge 6$, we apply the square function estimate
  \eqref{eqn:Littlewood-Paley_square_function_estimate}, H\"{o}lder,
  Bernstein, and Strichartz to get
  \begin{align*}
    &\| e^{-itH} f\|_{L_{t,x}^{ \fr{ 2(d+2) }{ d-2 }}}
    ^{\fr{2(d+2)}{d-2}} \sim \Bigl \| ( \sum_N | e^{-itH} f_N |^2
    )^{1/2} \Bigr \|_{\fr{2(d+2)}{d-2}}^{\fr{2(d+2)}{d-2}} = \iint (
    \sum_N |e^{-itH} f_N|^2 )^{\fr{d+2}{d-2}} \, dx \, dt\\
    &\lesssim \sum_{M \le N} \iint |e^{-itH} f_M|^{\fr{d+2}{d-2}}
    |e^{-itH}f_N|^{\fr{d+2}{d-2}} \, dx \, dt\\
    &\lesssim \sum_{M \le N} \|
    e^{-itH}f_M\|^{\fr{4}{d-2}}_{L_{t,x}^{\fr{ 2(d+2) }{d-2}}} \|
    e^{-itH} f_M\|_{L_{t,x}^{\fr{2(d+2)}{d-4}}} \|
    e^{-itH}f_N\|^{\fr{4}{d-2}}_{L_{t,x}^{\fr{2(d+2)}{d-2}}} \| e^{-itH} f_N 
\|_{L_{t,x}^{\fr{2(d+2)}{d}}}\\
    &\lesssim \sup_{N} \| e^{-itH}f_N
    \|^{\fr{8}{d-2}}_{L_{t,x}^{\fr{2(d+2)}{d-2}}} \sum_{M \le N} M^2
    \| e^{-itH} f_M \|_{L^{\fr{2(d+2)}{d-4}}_t L^{
        \fr{2d(d+2)}{d^2+8} }_x } \|f_N\|_{L^2}\\
    &\lesssim \sup_N \| e^{-itH}f_N \|_{ L_{t,x}^{\fr{2(d+2)}{d-2}}}^{
      \fr{8}{d-2}}
    \sum_{M \le N} M^2 \| f_M\|_{L^2_x} \| f_N\|_{L^2_x}\\
    &\lesssim \sup_N \| e^{-itH}f_N \|_{L_{t,x}^{ \fr{2(d+2)}{d-2}}}^{
      \fr{8}{d-2}} \sum_{M \le N} \fr{M}{N} \| H^{1/2} f_M\|_{L^2} \|
    H^{1/2}
    f_N\|_{L^2_x}\\
    &\lesssim \sup_{N} \| e^{-itH} f_N
    \|_{L_{t,x}^{\fr{2(d+2)}{d-2}}}^{\fr{8}{d-2}} \|f\|_{\Sigma}^2.
  \end{align*}
The cases $d = 3, 4, 5$ are handled similarly with some minor 
modifications 
in the applications of H\"{o}lder's inequality.
\end{proof}

The next proposition goes one step further and asserts that the
sequence $e^{-itH} f_n$ with nontrivial spacetime norm must in fact
contain a bubble centered at some $(t_n, x_n)$ with spatial scale
$N_n^{-1}$. We first introduce some vocabulary and notation which will
help make the presentation more systematic. Adapting terminology from
Ionescu-Pausader-Staffilani~\cite{MR3006640}, we define
\begin{define}
\label{def:frame}
A \emph{frame} is a sequence $(t_n, x_n, N_n) \in I \times \mf{R}^d
\times 2^{\mf{N}}$ conforming to one of the following scenarios:
\begin{enumerate}
\item \label{enum:frame_1} $N_n \equiv 1, \ t_n \equiv 0$, and $x_n \equiv
  0$.
\item \label{enum:frame_2} $N_n \to \infty$ and $ N_n^{-1} |x_n| \to r_\infty \in [0, \infty)$.
\end{enumerate}
\end{define}
Informally, the parameters $t_n, \ x_n, \ N_n$ will specify the
temporal center, spatial center, and (spatial) frequency of a
function. The condition that $|x_n| \lesssim N_n$ reflects the fact
that we only consider functions obeying some uniform bound in $\Sigma$, and such
functions cannot be centered arbitrarily far from the origin. We 
need to augment the frame $\{(t_n, x_n, N_n)\}$ with an auxiliary
parameter $N_n'$, which corresponds to a sequence of spatial cutoffs
adapted to the frame.
\begin{define}
\label{def:aug_frame}
An \emph{augmented frame} is a sequence $(t_n, x_n, N_n, N_n') \in I \times \mf{R}^d
\times 2^{\mf{N}} \times \mf{R}$ belonging to one of the following types:
\begin{enumerate}
\item \label{enum:aug_frame_1} $N_n \equiv 1, \ t_n \equiv 0, \ x_n \equiv
  0, \ N_n' \equiv 1$.
\item \label{enum:aug_frame_2} $N_n \to \infty, \ N_n^{-1} |x_n| \to r_\infty \in [0, \infty)$,
and either
\begin{enumerate}
\item \label{enum:aug_frame_2_1} $N_n' \equiv 1$ if  $r_\infty > 0$, or
\item \label{enum:aug_frame_2_2} 
$N_n^{1/2} \le N_n' \le N_n, \ N_n^{-1} |x_n| (\tfr{N_n}{N_n'})
\to 0$, and $\tfr{N_n}{N_n'} \to \infty$ if $r_\infty = 0$.
\end{enumerate}
\end{enumerate}
\end{define}

Given an augmented frame $(t_n, x_n, N_n, N_n')$, we define scaling and
translation operators on functions of space and of spacetime by
\begin{equation}
\label{eqn:inv_strichartz_frameops_1}
\begin{split}
(G_n \phi)(x) &= N_n^{\fr{d-2}{2}} \phi(N_n(x-x_n))\\
 (\tilde{G}_n f)(t, x) &= N_n^{\fr{d-2}{2}} f(N_n^2(t - t_n), N_n(x-x_n)).
\end{split}
\end{equation}
We also define spatial cutoff operators $S_n$ by
\begin{equation}
\label{eqn:inv_strichartz_frameops_2}
S_n \phi = \left\{ \begin{array}{cc} \phi, & \text{for frames of
      type} \ \ref{enum:aug_frame_1} \quad (\text{i.e.} \ N_n \equiv 1),\\
\chi(\tfr{N_n}{N_n'} \cdot) \phi, &  \text{for
  frames of type} \ \ref{enum:aug_frame_2} \quad (\text{i.e.} \ N_n \to \infty),
\end{array}\right.
\end{equation}
where $\chi$ is a smooth compactly supported function equal to $1$ on
the ball $\{|x| \le 1\}$. An easy computation yields the following
mapping properties of these operators:
\begin{equation}
\label{eqn:inv_strichartz_frameops_properties}
\begin{split}
\lim_{n \to \infty} S_n = I \ \text{strongly in} \ \dot{H}^1 \text{and in} \ \Sigma,\\
\limsup_{n \to \infty} \| G_n\|_{\Sigma \to \Sigma} < \infty.
\end{split}
\end{equation}

For future reference, we record a technical lemma that, as a
special case, asserts that the $\Sigma$ norm is
controlled almost entirely by the $\dot{H}^1$ norm for functions
concentrating near the origin.

\begin{lma}[Approximation]
\label{lma:approximation_lemma}
Let $(q, r)$ be an admissible pair of exponents with $2 \le r <
d$, and let $\mcal{F} = \{(t_n, x_n, N_n, N_n')\}$ be an augmented
frame of type 2.
\begin{enumerate}
\item Suppose $\mcal{F}$ is of type \ref{enum:aug_frame_2_1} in Definition
  \ref{def:aug_frame}. Then for $\{f_n\} \subseteq L^q_t
  H^{1, r}_x(\mf{R} \times \mf{R}^d)$, we have
\[
\limsup_{n} \| \tilde{G}_n S_n f_n \|_{L^q_t \Sigma^r_x} \lesssim
\limsup_{n} \| f_n \|_{L^q_t H^{1,r}_x}. 
\]
\item Suppose $\mcal{F}$ is of type \ref{enum:aug_frame_2_2} and $f_n \in L^q_t
  \dot{H}^{1,r}_x (\mf{R} \times \mf{R}^d)$. Then
\[
\limsup_{n} \| \tilde{G}_n S_n f_n \|_{L^q_t \Sigma^r_x} \lesssim 
\limsup_{n} \| f_n \|_{L^q_t \dot{H}^{1,r}_x}.
\]
\end{enumerate}
Here $H^{1, r}(\mf{R}^d)$ and $\dot{H}^{1, r}(\mf{R}^d)$ denote the
inhomogeneous and homogeneous $L^r$ Sobolev spaces, respectively,
equipped with the norms 
\[
\| f\|_{H^{1, r}} = \| \langle \nabla \rangle\|_{L^r(\mf{R}^d)}, \quad
\|f\|_{\dot{H}^{1, r}} = \| |\nabla| f\|_{L^r(\mf{R}^d)}.
\]
\end{lma}

\begin{proof}
By time translation invariance we may assume $t_n \equiv
0$. Using Lemma~\ref{lma:equivalence_of_norms}, we see that it suffices to bound
$\|\nabla \tilde{G}_n S_n f_n \|_{L^q_t L^r_x}$ and $\| |x|
\tilde{G}_n S_n f_n\|_{L^q_tL^r_x}$
separately. By a change of variables, the admissibility condition
on $(q, r)$, H\"{o}lder, and Sobolev embedding (which necessitates the
restriction $r < d$), we have
\[
\begin{split}
&\| \nabla \tilde{G}_n S_n f_n \|_{L^q_tL^r_x} = \| \nabla [
N_n^{\fr{d-2}{2}} f_n(N^2_nt, N_n(x-x_n)) \chi(N_n'(x-x_n))] \|_{L^q_t L^r_x}\\
&\lesssim  \| (\nabla f_n)(t,x)
\|_{L^q_tL^r_x} + \tfr{N_n'}{N_n} \| f_n(t, x)\|_{L^q_t L^r_x ( \mf{R} \times \{|x|
  \sim \fr{N_n}{N_n'} \})}\\
&\lesssim \| \nabla f_n\|_{L^q_t \dot{H}^{1, r}_x}.
\end{split}
\]
To estimate $\| |x| \tilde{G}_n S_n f_n\|_{L^q_t L^r_x}$ we
distinguish the two cases. Consider first the case in which $f_n \in
L^q_t H^{1,r}_x$. Using the bound $|x_n| \lesssim N_n$ and a change of
variables, we obtain
\[
\begin{split}
\| |x| \tilde{G}_n S_n f_n \|_{L^q_t L^r} &\lesssim N_n^{\fr{d}{2}} \|
f_n (N_n^2 t, N_n(x-x_n)) \|_{L^r} \lesssim \|f_n\|_{L^q_tL^r}
\lesssim \|f_n\|_{L^q_t H^{1,r}_x}.
\end{split}
\]

Now consider the second case where $f_n$ are merely assumed to lie in
$L^q_t \dot{H}^{1,r}_x$. For each $t$, we use H\"{o}lder and Sobolev
embedding to get
\[
\begin{split}
&\| |x| \tilde{G}_n S_n f_n\|_{L^r_x}^r = N_n^{\fr{dr}{2} - d - r}
\int_{|x| \lesssim \fr{N_n}{N_n'}} |x_n + N_n^{-1}x|^r |f_n(N_n^2t, x)|^r dx
\\
&\lesssim N_n^{\fr{dr}{2} - d} \left[ N_n^{-r} |x_n|^r +  N_n^{-2r}
  (\tfr{N_n}{N_n'})^{r} \right]
\int_{ |x| \lesssim \fr{N_n}{N_n'} } |f_n(N_n^2t, x)|^r dx\\
&\lesssim N_n^{\fr{dr}{2} - d} \left[  N_n^{-r} |x_n|^r (\tfr{N_n}{N_n'})^r
+ (N_n')^{-2r}\right] \| \nabla f_n(N_n^2t)\|_{L^r_x}^r.
\end{split}
\]
By the hypotheses on the parameter $N_n'$ in Definition
\ref{def:aug_frame}, the expression inside the brackets goes to $0$
as $n \to \infty$. After integrating in $t$ and performing a change of
variables, we conclude
\[
\| |x| \tilde{G}_n S_n f_n \|_{L^q_t L^r_x} \lesssim c_n
\|f_n\|_{L^q_t \dot{H}^{1,r}_x}
\]
where $c_n = o(1)$ as $n \to \infty$. This completes the proof of the lemma.
\end{proof}

\begin{prop}[Inverse Strichartz]
\label{prop:inv_strichartz}
Let $I$ be a compact interval containing $0$ of length at most $1$, and
suppose $f_n$ is a sequence of functions in $\Sigma(\mf{R}^d)$ satisfying
\[
0 < \varepsilon \le \| e^{-itH} f_n\|_{L^{\fr{2(d+2)}{d-2}}_{t,x}(I
  \times \mf{R}^d)}
\lesssim \|f_n\|_{\Sigma} \le A < \infty.
\]
Then, after passing to a subsequence, there exists an augmented frame 
\[
\mcal{F} =
\{ (t_n, x_n, N_n, N_n') \}
\]
and a sequence of functions $\phi_n\in
\Sigma$ such that one of the following holds:
\begin{enumerate}
\item \label{enum:inv_strichartz_case_1} $\mcal{F}$ is of type
  \ref{enum:aug_frame_1} (i.e. $N_n \equiv 1$) and $\phi_n = \phi$ where $\phi \in \Sigma$
  is a weak limit of $f_n$ in $\Sigma$.
\item \label{enum:inv_strichartz_case_2} $\mcal{F}$ is of type
  \ref{enum:aug_frame_2}, either $t_n \equiv 0$ or $N_n^2 t_n \to \pm
  \infty$, and $\phi_n = e^{it_n H} G_n S_n \phi$
  where 
$\phi \in \dot{H}^1(\mf{R}^d)$ is a weak limit of $G_n^{-1} e^{-it_n H} f_n$ in
$\dot{H}^1$. Moreover, if $\mcal{F}$ is of type
\ref{enum:aug_frame_2_1}, then $\phi$ also belongs to $L^2(\mf{R}^d)$.
\end{enumerate}
The functions $\phi_n$ have the following properties:
\begin{equation}
 \liminf_{n} \| \phi_n \|_{\Sigma} \gtrsim A \left( \tfr{\varepsilon}{A} 
\right)^{\fr{d(d+2)}{8}} \label{eqn:inv_strichartz_nontriviality_in_Sigma}
\end{equation}
\begin{equation}
 \lim_{n \to \infty} \|f_n\|_{\fr{2d}{d-2}}^{\fr{2d}{d-2}} - \|f_n - 
\phi_n\|_{\fr{2d}{d-2}}^{\fr{2d}{d-2}} - \| \phi_n 
\|_{\fr{2d}{d-2}}^{\fr{2d}{d-2}}  = 0. \label{eqn:inv_strichartz_potential_energy_decoupling}
\end{equation}
\begin{equation}
 \lim_{n \to \infty}  \|f_n\|_{\Sigma}^2 - \|  f_n - \phi_n\|_{\Sigma}^2 - 
\| \phi_n \|_{\Sigma}^2  = 0 \label{eqn:inv_strichartz_decoupling_in_Sigma}
\end{equation}

\end{prop}

\begin{proof}
The proof will occur in several stages. First we identify the
parameters $t_n, x_n, N_n$, which define the location of the bubble
$\phi_n$ and its characteristic size, and quickly dispose of the case where $N_n \equiv
1$. The treatment of the case where $N_n \to
\infty$ will be more involved, and we proceed in two
steps. We define the profile $\phi_n$ and verify the assertions
\eqref{eqn:inv_strichartz_nontriviality_in_Sigma} and
\eqref{eqn:inv_strichartz_decoupling_in_Sigma}. Passing to a
subsequence, we may assume that the sequence $N_n^2 t_n$ converges in
$[-\infty, \infty]$. If the limit is infinite, decoupling 
\eqref{eqn:inv_strichartz_potential_energy_decoupling} in the
$L^{\fr{2d}{d-2}}$ norm will also
follow.

After a brief interlude in which we study certain operator limits, we
finish the case where the sequence $N_n^2 t_n$ tends to a finite
limit. We show that the time parameter $t_n$ can actually be redefined
to be identically zero after making a negligible correction to the
profile $\phi_n$, and verify in Lemma~\ref{lma:inv_strichartz_part_2}
that the modified profile satisfies property
\eqref{eqn:inv_strichartz_potential_energy_decoupling} in addition to
\eqref{eqn:inv_strichartz_nontriviality_in_Sigma} and
\eqref{eqn:inv_strichartz_decoupling_in_Sigma}. One can show
\emph{a posteriori} that the original profile $\phi_n$ also obeys this
last decoupling condition.

By Proposition \ref{prop:refined_strichartz}, there exist frequencies $N_n$ 
such that 
\[
\| P_{N_n} e^{-itH} f_n\|_{L^{\fr{2(d+2)}{d-2}}_{t,x}}
\gtrsim \varepsilon^{\fr{d+2}{4}} A^{-\fr{d-2}{4}}.
\]
The comparison of Littlewood-Paley projectors 
\eqref{eqn:Littlewood-Paley_projector_comparison} implies
\[
\| \tilde{P}_{N_n} e^{-itH} f_n \|_{L^{\fr{2(d+2)}{d-2}}_{t, x}}
\gtrsim \varepsilon^{\fr{d+2}{4}} A^{-\fr{d-2}{4}}
\]
where $\tilde{P}_N = e^{-H/N^2} - e^{-4H/N^2}$ denote the projections
based on the heat kernel. By H\"{o}lder, Strichartz, and Bernstein,
\begin{align*}
\varepsilon^{\fr{d+2}{4}} A^{-\fr{d-2}{4}} &\lesssim \| \tilde{P}_{N_n}
e^{-itH} f_n \|_{L^{\fr{2(d+2)}{d-2}}_{t,x}} \lesssim \|
\tilde{P}_{N_n} e^{-itH} f_n
\|_{L^{\fr{2(d+2)}{d}}_{t,x}}^{\fr{d-2}{d}} \| \tilde{P}_{N_n}
e^{-itH} f_n \|_{L^\infty_{t, x}}^{\fr{2}{d}}\\
&\lesssim (N_n^{-1}A)^{\fr{d-2}{d}} \| \tilde{P}_{N_n} e^{-itH} f_n\|_{L^\infty_{t,x}}^{\fr{2}{d}}.
\end{align*}
Therefore, there exist $(t_n, x_n) \in I \times \mf{R}^d$ such that 
\begin{equation}
\label{eqn:inv_strichartz_concentration}
|e^{-it_n H} \tilde{P}_{N_n} f_n (x_n) | \gtrsim N_n^{\fr{d-2}{2}}
A(\tfr{\varepsilon}{A})^{\fr{d(d+2)}{8}}.
\end{equation}
The parameters $t_n, x_n, N_n$ will determine the center and
width of a bubble. We observe first that the boundedness of $f_n$ in $\Sigma$ 
limits how far the bubble can live from the spatial origin. 
\begin{lma}
\label{lma:x_n_bounded_by_N_n}
We have 
\[
 |x_n| \le C_{A, \varepsilon} N_n.
\]
\end{lma}

\begin{proof}
Put $g_n = |e^{-it_n H} f_n|$. By the kernel bound 
\eqref{eqn:Littlewood-Paley_pointwise_comparison}, 
\[
 N_n^{\fr{d-2}{2}} A (\tfr{\varepsilon}{A})^{\fr{d(d+2)}{8}}
\lesssim |\tilde{P}_{N_n} e^{-it_nH} f_n  (x_n)| \lesssim \tilde{P}_{\le 
N_n}^{\Delta} g_n(x_n) + \tilde{P}_{\le N_n/2}^{\Delta} g_n(x_n).
\]
Thus one of the terms on the right side is at least half as large as the left side, and 
we only consider the case when
\[
 \tilde{P}_{\le N_n}^{\Delta} g_n (x_n) \gtrsim N_n^{\fr{d-2}{2}} 
A(\tfr{\varepsilon}{A})^{\fr{d(d+2)}{8}}
\]
since the argument with $N_n$ replaced by $N_n/2$ differs only cosmetically. 
Informally, $\tilde{P}_{\le N_n}^{\Delta} g_n$ is essentially constant over 
length scales of order $N_n^{-1}$, so if it is large at a point $x_n$ then it 
is large on the ball $|x-x_n| \le N_n^{-1}$. More precisely, when $|x-x_n| \le 
N_n^{-1}$ we have
\[
\begin{split}
\tilde{P}_{\le N_n/2}^{\Delta} g_n (x) &= 
\tfr{N_n^d}{ 2^d(2\pi)^{\fr{d}{2}}} 
\int g_n (x - y) e^{-\fr{N_n^2 |y|^2}{8}} dy \\
&= \tfr{N_n^d}{ 2^d(4\pi)^{\fr{d}{2}}} \int g_n(x_n - y) e^{-\fr{N_n^2 
|y + x - x_n|^2}{8}} dy\\
&\ge e^{-1} \tfr{N_n^d}{ 2^d(4\pi)^{\fr{d}{2}} } \int g_n(x_n - y) 
e^{- \fr{N_n^2 |y|^2}{2}} dy = e^{-1} 2^{-d} \tilde{P}_{\le N_n}^{\Delta} g_n 
(x_n)\\
&\gtrsim N_n^{\fr{d-2}{2}} A(\tfr{\varepsilon}{A})^{\fr{d(d+2)}{8}}.
\end{split}
\]
On the other hand, the mapping properties of the heat kernel imply that
 \[
\| \tilde{P}_{\le N_n/2}^{\Delta} g_n\|_{\Sigma} \lesssim (1 + N_n^{-2})A.  
 \]
Thus,
\[
 A \gtrsim \| \tilde{P}_{\le N_n/2}^{\Delta} g_n\|_{\Sigma} \gtrsim \| x 
\tilde{P}_{\le N_n/2}^{\Delta} g_n\|_{L^2( |x-x_n| \le N_n^{-1} )} \gtrsim 
|x_n| N_n^{-\fr{d}{2}} N_n^{\fr{d-2}{2}} A 
(\tfr{\varepsilon}{A})^{\fr{d(d+2)}{8}},
\]
which yields the claim.

\end{proof}

\textbf{Case 1}. Suppose the $N_n$ have a bounded subsequence, so that
(passing to a subsequence) $N_n \equiv N_\infty$. The
$x_n$'s stay bounded by 
\ref{lma:x_n_bounded_by_N_n}, so after passing to a subsequence we may assume $x_n \to x_\infty$. We may also
assume $t_n \to t_\infty$ since the interval $I$ is compact. 
The functions $f_n$ are bounded in $\Sigma$, hence (after passing to a subsequence)
converge weakly in $\Sigma$ to a function $\phi$. 

We show that $\phi$ is nontrivial in $\Sigma$. Indeed,
\[
\begin{split}
\langle \phi, e^{it_\infty H} \tilde{P}_{N_\infty} \delta_{x_\infty}
\rangle &= \lim_{n} \langle f_n, e^{it_\infty H} \tilde{P}_{N_\infty}
\delta_{x_\infty} \rangle \\
&= \lim_{n \to \infty} [e^{-it_n H} \tilde{P}_{N_\infty} f_n (x_n)
+ \langle f_n, (e^{it_\infty H} - e^{it_n H}) \tilde{P}_{N_\infty}
\delta_{x_n} \rangle \\
&+ \langle f_n, e^{it_\infty H} \tilde{P}_{N_n} (
\delta_{x_\infty} - \delta_{x_n}) \rangle].
\end{split}
\]
Using the heat kernel bounds
\eqref{eqn:Littlewood-Paley_pointwise_comparison} and the fact that,
by the compactness of the embedding $\Sigma \subset L^2$, the sequence
$f_n$ converges to $\phi$ in $L^2$, one verifies easily that the
second and third terms on the right side vanish. So
\[
|\langle \phi, e^{it_\infty H} \tilde{P}_{N_\infty} \delta_{x_\infty}
\rangle| = \lim_{n\to \infty} |e^{-it_n H} \tilde{P}_{N_\infty} f_n
(x_n)| \gtrsim N_\infty^{\fr{d-2}{2}} \varepsilon^{\fr{d(d+2)}{8}} A^{-\fr{(d-2)(d+4)}{8}}.
\]
On the other hand, by H\"{o}lder and
\eqref{eqn:Littlewood-Paley_pointwise_comparison},
\[
\begin{split}
|\langle \phi, e^{it_\infty H} \tilde{P}_{N_\infty} \delta_{x_\infty}
\rangle| &\le \| e^{-it_\infty H} \phi\|_{L^{\fr{2d}{d-2}}} \|
\tilde{P}_{N_\infty} \delta_{x_\infty} \|_{L^{\fr{2d}{d+2}}}\\
&\lesssim \| \phi\|_{\Sigma} N_\infty^{\fr{d-2}{2}} .
\end{split}
\]
Therefore
\[
\| \phi\|_{\Sigma} \gtrsim \varepsilon^{\fr{d(d+2)}{8}} A^{-\fr{(d-2)(d+4)}{8}}.
\]

Set
\[
\phi_n \equiv \phi,
\]
and define the augmented frame $(t_n, x_n, N_n, N_n') \equiv (0, 0, 1, 1)$. The
decoupling in $\Sigma$ 
\eqref{eqn:inv_strichartz_decoupling_in_Sigma} can be proved as in
Case 2 below, and we refer the reader to the argument detailed there.
%
It remains to establish decoupling in $L^{\fr{2d}{d-2}}$. As
the embedding $\Sigma \subset L^2$ is compact, the sequence $f_n$,
which converges weakly to $\phi \in \Sigma$, converges to $\phi$ strongly in
$L^2$. After passing to a subsequence we obtain convergence pointwise a.e. The decoupling
\eqref{eqn:inv_strichartz_potential_energy_decoupling} now follows from 
Lemma~\ref{lma:refined_fatou}. This completes the case where $N_n$
have a bounded subsequence.

\textbf{Case 2}. We now address the case where $N_n \to \infty$. The
main nuisance is that the weak limits $\phi$ will usually be merely in
$\dot{H}^1(\mf{R}^d)$, not in $\Sigma$, so defining the
profiles $\phi_n$ will require spatial cutoffs.

As the functions $N_n^{-(d-2)/2} (e^{-it_n H} f_n)( N_n^{-1} \cdot +
x_n)$ are bounded in $\dot{H}^1(\mf{R}^d)$, the sequence has a weak
subsequential limit
\begin{equation}
N_n^{-\fr{d-2}{2}} (e^{-it_n H} f_n) (N_n^{-1} \cdot + x_n)
\rightharpoonup \phi \text{ in } \dot{H}^1(\mf{R}^d). \label{eqn:inv_strichartz_case_2_phi_defn}
\end{equation}
By Lemma \ref{lma:x_n_bounded_by_N_n}, after passing to a further subsequence we may 
assume
\begin{equation}
\label{eqn:inv_strichartz_x_n_N_n_limit}
\lim_{n \to \infty} N_n^{-1}|x_n| = r_\infty < \infty \quad
\text{and} \quad \lim_{n \to
  \infty} N_n^2 t_n = t_\infty \in [-\infty, \infty].
\end{equation}
It will be necessary to distinguish the cases $r_\infty > 0$ and $r_\infty =
0$, corresponding to whether the frame $\{(t_n, x_n, N_n)\}$ is type
\ref{enum:aug_frame_2_1} or \ref{enum:aug_frame_2_2}, respectively.
\begin{lma}
\label{lma:phi_additional_decay}
If $r_\infty > 0$, the function $\phi$ defined in
  \eqref{eqn:inv_strichartz_case_2_phi_defn} also belongs to $L^2$.
\end{lma}

\begin{proof}
By \eqref{eqn:inv_strichartz_case_2_phi_defn} and the Rellich-Kondrashov
compactness theorem, for each $R \ge 1$ we have
\[
N_n^{-\fr{d-2}{2}} (e^{-it_n H} f_n) (N_n^{-1}\cdot + x_n) \to \phi \text{ in } L^2( \{|x| \le R\}).
\]
By a change of variables,
\[
\begin{split}
N_n^{-\fr{d-2}{2}} (e^{-it_n H} f_n) (N_n^{-1} \cdot +
x_n)\|_{L^2(|x| \le R)} &= N_n
\| e^{-it_n H} f_n \|_{L^2( |x-x_n| \le RN_n^{-1})} \\
&\lesssim \|
xe^{-it_n H} f_n\|_{L^2}
\end{split}
\]
whenever $|x_n| \ge \tfr{N_n r_\infty}{2}$ and $RN_n^{-1} \le \tfr{r_\infty}{10}$, so we have uniformly in $R \ge 1$ that
\[
\limsup_{n} \| N_n^{-\fr{d-2}{2}} (e^{-it_nH} f_n)(N_n^{-1} \cdot + x_n)
\|_{L^2(|x| \le R)} \lesssim \sup_{n} \|e^{-it_nH} f_n \|_{\Sigma} \lesssim 1.
\]
Therefore $\| \phi\|_{L^2} = \lim_{R \to \infty} \| \phi\|_{L^2(|x|
  \le R)} \lesssim 1$. 
\end{proof}

\begin{rmk}
The claim fails if $r_\infty = 0$. Indeed, if $\phi \in
\dot{H}^1(\mf{R}^d) \setminus L^2(\mf{R}^d)$, then $f_n = N_n^{(d-2)/2} \phi(N_n\cdot ) \chi(\cdot)$
are bounded in $\Sigma$, and $N_n^{-(d-2)/2} f_n(N_n^{-1} \cdot) =
 \phi(\cdot)  \chi(N_n^{-1} \cdot) $ converges strongly in $\dot{H}^1$ to
$\phi$. 

\end{rmk}

Next we prove that $\phi$ is nontrivial in $\dot{H}^1$.
\begin{lma} 
\label{lma:inv_strichartz_case_2_phi_nontrivial}
$\| \phi\|_{\dot{H}^1} \gtrsim A \left(\tfr{
      \varepsilon}{A} \right)^{\fr{d(d+2)}{8}}$. 
\end{lma}

\begin{proof}
From \eqref{eqn:Littlewood-Paley_pointwise_comparison} and
\eqref{eqn:inv_strichartz_concentration}, 
\[
N_n^{\fr{d-2}{2}} A \left( \tfr{\varepsilon}{A}
\right)^{\fr{d(d+2)}{8}} \lesssim \tilde{P}_{\le N_n}^{\Delta}
|e^{-it_n H} f_n|(x_n) + \tilde{P}_{\le N_n/2}^{\Delta} |e^{-it_n H}
f_n| (x_n),
\]
so one of the terms on the right is at least half the left
side. Suppose first that
\[
\tilde{P}_{\le N_n}^{\Delta} |e^{-it_n H} f_n| (x_n) \gtrsim
N_n^{\fr{d-2}{2}} A \left( \tfr{\varepsilon}{A}
\right)^{\fr{d(d+2)}{8}}.
\]
Put $\check{\psi} = \tilde{P}_{\le 1}^\Delta \delta_0 = e^{\Delta}
\delta_0$. Since $\check{\psi}$ is Schwartz,
\[
| \langle |\phi|, \check{\psi} \rangle_{L^2}| \le \| \phi\|_{\dot{H}^1}
\| \check{\psi}\|_{ \dot{H}^{-1}} \lesssim \| \phi\|_{\dot{H}^1}.
\]
On the other hand, as the absolute values
$N_n^{-\fr{d-2}{2}} |e^{-it_n H} f_n| (N_n^{-1}\cdot + x_n)$ converge
weakly in $\dot{H}^1$ to $|\phi|$, we have
\begin{align*}
\langle |\phi|, \check{\psi} \rangle_{L^2} &= \lim_{n} \langle
N_n^{-\fr{d-2}{2}} | e^{-it_n H} f_n| (N_n^{-1} \cdot + x_n),
\check{\psi} \rangle_{L^2}\\
&= \lim_{n} \tilde{P}_{\le N_n}^\Delta |e^{-it_n H} f_n| (x_n) \gtrsim
A \left( \tfr{\varepsilon}{A} \right)^{\fr{d(d+2)}{8}}.
\end{align*}
from which the claim follows. Similarly if
\[
\tilde{P}_{\le N_n/2}^{\Delta} |e^{-it_n H} f_n| (x_n) \gtrsim
N_n^{\fr{d-2}{2}} A \left( \tfr{\varepsilon}{A}
\right)^{\fr{d(d+2)}{8}},
\]
then we obtain $\| \phi\|_{\dot{H}^1} \sim \| \phi(2 \cdot)
\|_{\dot{H}^1} \gtrsim N_n^{\fr{d-2}{2}} A \left( \tfr{\varepsilon}{A}
\right)^{\fr{d(d+2)}{8}}$. 
\end{proof}

Having extracted a nontrivial bubble $\phi$, we are ready to define
the $\phi_n$. The basic idea is to undo the operations applied
to $f_n$ in the definition \eqref{eqn:inv_strichartz_case_2_phi_defn}
of $\phi$. However, we need to first apply a spatial cutoff to
embed $\phi$ in $\Sigma$. 

With the frame $\{(t_n, x_n, N_n)\}$ defined according to
\eqref{eqn:inv_strichartz_concentration}, we form the augmented frame
$\{ (t_n, x_n, N_n, N_n') \}$ with the cutoff parameter $N_n'$ chosen
according to the second case of Definition \ref{def:aug_frame}. Let
$G_n, \ S_n$ be the $\dot{H}^1$ isometries and spatial cutoff operators
associated to $\{(t_n, x_n, N_n, N_n')\}$. Set
\begin{equation}
\label{eqn:inv_strichartz_case_2_phi_n_defn}
\phi_n = e^{it_n H} G_n S_n \phi = e^{it_n H} [ N_n^{\fr{d-2}{2}}
\phi(N_n(\cdot - x_n)) \chi(N_n' (\cdot - x_n) )].
\end{equation}
We now verify that the $\phi_n$ satisfy the various properties claimed in the
proposition.

\begin{lma}
\label{lma:inv_strichartz_case_2_phi_n_Sigma_bounds}
$ A \left(\tfr{\varepsilon}{A} \right)^{\fr{d(d+2)}{8}} \lesssim
\liminf_{n \to \infty} \| \phi_n\|_{\Sigma} \le \limsup_{n \to \infty} \|
\phi_n\|_{\Sigma} \lesssim 1$.
\end{lma}

\begin{proof}

By the
definition of the $\Sigma$ norm and a change of variables,
\[
\| \phi_n\|_{\Sigma} = \| G_n S_n\|_{\Sigma} \ge \| S_n \phi \|_{\dot{H}^1}.
\]
Hence Lemma~\ref{lma:inv_strichartz_case_2_phi_nontrivial} and the
remarks following Definition \ref{def:aug_frame} together imply the lower bound
\[
\liminf_{n} \| \phi_n\|_{\Sigma} \gtrsim A \left( \tfr{\varepsilon}{A} \right)^{\fr{d(d+2)}{8}}.
\]
The upper bound follows immediately from the case $(q, r) = (\infty, 2)$ in 
Lemma~\ref{lma:approximation_lemma}.
\end{proof}

We verify the decoupling property
\eqref{eqn:inv_strichartz_decoupling_in_Sigma}. By the Pythagorean
theorem,
\[
\begin{split}
\| f_n\|_{\Sigma}^2 - \| f_n - \phi_n\|_{\Sigma}^2 - \| \phi_n\|_{\Sigma}^2 
&= 2\opn{Re}( \langle f_n - \phi_n , \phi_n\rangle_{\Sigma})\\
&= 2\opn{Re}( \langle e^{-it_n H}f_n - G_n S_n \phi, G_n S_n\\
&= 2\opn{Re} ( \langle w_n, G_n S_n \phi \rangle_{\Sigma}).
\end{split}
\]
where $w_n = e^{-it_n H} f_n - G_n S_n \phi$.
By definition, 
\[\langle w_n , G_n S_n \phi \rangle_{\Sigma} = \langle
w_n, G_n S_n \phi \rangle_{\dot{H}^1} + \langle x w_n,
x G_n S_n \phi \rangle_{L^2}.
\]
From
\eqref{eqn:inv_strichartz_frameops_properties} and the definition
\eqref{eqn:inv_strichartz_case_2_phi_defn} of $\phi$, it follows that
\[
G_n^{-1} w_n \to 0 \quad \text{weakly in} \quad \dot{H}^1 \quad
\text{as} \quad n \to \infty.
\]
Hence
\[
\begin{split}
\lim_{n\to \infty} \langle w_n, G_n S_n \phi \rangle_{\dot{H}^1} =
\lim_{n\to \infty} \langle G_n^{-1} w_n,
S_n \phi \rangle_{\dot{H}^1} = \lim_{n\to\infty} \langle G_n^{-1} w_n, \phi
\rangle_{\dot{H}^1} = 0.
\end{split}
\]

We turn to the second component of the inner product. Fix $R > 0$, and estimate
\[
\begin{split}
&| \langle xw_n, xG_n S_n \phi \rangle_{L^2}| \\
&\le \int_{\{|x-x_n| \le RN_n^{-1}\}} |xw_n|
|xG_n S_n \phi| \, dx + \int_{\{|x-x-n| >RN_n^{-1} \}}
|xw_n| |xG_n S_n \phi| \, dx\\
&= (I) + (II)
\end{split}
\]
Perform a change of variable
and drop the spatial cutoff $S_n$, keeping in mind the bound $|x_n|
\lesssim N_n$, to obtain
\[
(I) \lesssim \int_{|x| \le R} |G_n^{-1} w_n| |\phi| \, dx \to 0 \quad
\text{as} \quad n \to \infty.
\]
Next, apply Cauchy-Schwartz and the upper bound of Lemma
\ref{lma:inv_strichartz_case_2_phi_n_Sigma_bounds} to see that
\[
\begin{split}
&(II)^2 \lesssim
\int_{\{|x-x_n| > RN_n^{-1} \}} |xG_n S_n \phi|^2 \, dx \\
&\lesssim N_n^{-2}
\int_{ R \le |x| \lesssim \fr{N_n}{N_n'} } |x_n + N_n^{-1} x|^2
|\phi(x)|^2 dx\\
&\lesssim (N_n^{-2} |x_n|^2 + N_n^{-2} (N_n')^{-2} ) \int_{R \le |x|
  \lesssim \fr{N_n}{N_n'}} |\phi(x)|^2 \, dx.
\end{split}
\]
Suppose that the frame $\{(t_n, x_n, N_n)\}$ is of type \ref{enum:aug_frame_2_1}, so that
$\lim_{n} N_n^{-1} |x_n| > 0$. By Lemma
\ref{lma:phi_additional_decay} and dominated convergence, the right
side above is bounded by
\[
\int_{R \le |x| } |\phi(x)|^2 \, dx \to 0 \quad \text{as} \quad R \to \infty,
\]
uniformly in $n$.
If instead $\{(t_n, x_n, N_n)\}$ is of type
\ref{enum:aug_frame_2_2}, use H\"{o}lder to see that the right side is bounded by
\[
\begin{split}
( N_n^{-2} |x_n| (\tfr{N_n}{N_n'})^2 + (N_n')^{-4}) \| \phi\|_{L^{\fr{2d}{d-2}}}.
\end{split}
\]
By Sobolev embedding and the construction of the parameter $N_n'$ in
Definition \ref{def:aug_frame}, the above vanishes as $n \to \infty$.
In either case, we obtain
\[
\lim_{R \to \infty} \limsup_{n \to \infty} (II) = 0.
\]

Combining the two estimates and choosing $R$ arbitrarily large, we
conclude as required that
\[
\lim_{n \to \infty} |\langle xw_n, xG_n S_n \phi\rangle_{L^2} | = 0.
\]

To close this subsection, we verify the $L^{\fr{2d}{d-2}}$ decoupling property
\eqref{eqn:inv_strichartz_potential_energy_decoupling} when $N_n^2 t_n
\to \pm \infty$. Assume first that the $\phi$ appearing in the
definition \eqref{eqn:inv_strichartz_case_2_phi_n_defn} of $\phi_n$
has compact support. By the dispersive estimate
\eqref{eqn:dispersive_estimate} and a change of variables, we have
\[
\lim_{n \to \infty} \| \phi_n\|_{L^{\fr{2d}{d-2}}} \lesssim |t_n|^{-1} \|
G_n \phi\|_{L^{\fr{2d}{d+2}}} \lesssim (N_n^2 |t_n|)^{-1} \|
\phi\|_{L^{\fr{2d}{d+2}}} = 0.
\]
The claimed decoupling follows immediately.

For general $\phi$ in $H^1$ or $\dot{H}^1$ (depending on whether
$\lim_{n} N_n^{-1} |x_n|$ is positive or zero), select
$\psi^\varepsilon \in C^\infty_c$ converging to $\phi$ in the
appropriate norm as $\varepsilon \to 0$. Then for all $n$ large
enough, we have
\[
\| \phi_n\|_{L^{\fr{2d}{d-2}}} \le \| e^{it_n H} G_n S_n[\phi -
\psi^\varepsilon] \|_{L^{\fr{2d}{d-2}}} +\| e^{it_n H} G_n S_n
\psi^\varepsilon \|_{L^{\fr{2d}{d-2}}},
\]
and we once again have decoupling by 
Lemmas~\ref{lma:sobolev_embedding} and \ref{lma:approximation_lemma}, and the special case just proved.
\end{proof}

\subsection{Convergence of linear propagators}
\label{subsection:convergence}
To complete the proof of Proposition~\ref{prop:inv_strichartz}, we
need a more detailed understanding of how the linear propagator
$e^{-itH}$ interacts with the $\dot{H}^1$-symmetries $G_n$ associated
to a frame in certain limiting situations. The lemmas proved in this
section are heavily inspired by the discussion surrounding \cite[Lemma
5.2]{2d_klein_gordon}, in which the authors prove analogous results
relating the linear propagators of the 2D Schr\"{o}dinger equation and
the complexified Klein-Gordon equation $-iv_t + \langle \nabla \rangle
v = 0$.  We begin by introducing some
terminology due to Ionescu-Pausader-Staffilani~\cite{MR3006640}.
\begin{define}
\label{def:equiv_frames}
We say two frames $\mcal{F}^1 = \{ (t_n^1, x_n^1, N_n^1)\}$ and $\mcal{F}^2 =
\{ (t_n^2, x_n^2, N_n^2)\}$ (where the superscripts are indices, not
exponents) are \emph{equivalent} if
\[
\tfr{N_n^1}{N_n^2} \to R_\infty \in (0, \infty), \ N_n^1(x_n^2 - x_n^1) \to
x_\infty \in \mf{R}^d, \  (N_n^1)^2 (t_n^1 - t_n^2) \to
t_\infty \in \mf{R}.
\]
If any of the above statements fails, we say that $\mcal{F}_1$ and $\mcal{F}_2$
are \emph{orthogonal}. Note that replacing the  $N_n^1$ in the second and third 
expressions above by $N_n^2$ yields an equivalent definition of orthogonality.

\end{define}

\begin{rmk}
If $\mcal{F}^1$ and $\mcal{F}^2$ are equivalent, it follows from the above
definition that they must be of the same type in
Definition \ref{def:frame}, and that $\lim_{n} (N_n^1)^{-1} |x_n^1|$
and $\lim_n (N_n^2)^{-1} |x_n^2|$ are either both zero or both positive.
\end{rmk}

One interpretation of the following lemma and its corollary is that
when acting on functions concentrated at a point,
$e^{-itH}$ can be approximated for small $t$ by
regarding the $|x|^2/2$ potential as essentially constant on the
support of the initial data; thus one obtains a modulated free particle
propagator 
$e^{-\fr{it|x_0|^2}{2}} e^{\fr{it\Delta}{2}}$
where $x_0$
is the spatial center of the initial data.

\begin{lma}[Strong convergence]
\label{lma:strong_convergence}
Suppose 
\[\mcal{F}^M = (t_n^M, x_n, M_n), \quad \mcal{F}^N = (t_n^N,
y_n, N_n)\]
are equivalent frames. Define
\[
\begin{split}
R_\infty &= \lim_{n \to \infty} \tfr{M_n}{N_n}, \ t_\infty = \lim_{n \to
  \infty} M_n^2 (t_n^M - t_n^N), \ x_\infty = \lim_{n \to \infty} M_n(y_n-x_n)\\
r_\infty &= \lim_n M_n^{-1} |x_n|  = \lim_{n} M_n^{-1} |y_n|.
\end{split}
\]
Let $G_n^M, G_n^N$ be the scaling and translation operators attached
to the frames $\mcal{F}^M$ and $\mcal{F}^N$ respectively. Then
$(e^{-it_n^N H} G_n^N)^{-1} e^{-it_n^M H} G_n^M$ converges in the
strong operator topology on $B(\Sigma, \Sigma)$ to the
operator $U_\infty$ defined by
\[
U_\infty \phi = e^{-\fr{it_\infty (r_\infty)^2}{2}}
R_\infty^{\fr{d-2}{2}} [e^{ \fr{it_\infty \Delta}{2} } \phi]( R_\infty
\cdot + x_\infty).
\]
\end{lma}

\begin{proof}
  If $M_n \equiv 1$, then by the definition of a frame we must have
  $\mcal{F}^M = \mcal{F}^N = \{ (1, 0, 0) \}$, so the claim is
  trivial. Thus we may assume that $M_n \to \infty$.  Put $t_n = t_n^M
  - t_n^N$. Using Mehler's formula \eqref{eqn:lens_transformation}, we
  write
\[
\begin{split}
&(e^{-it_n^N H} G_n^N)^{-1} e^{-it_n^M H} G_n^M = (G_n^N)^{-1} e^{-it_nH} G_n^M \phi (x) \\
&= (\tfr{M_n}{N_n})^{\fr{d-2}{2}} e^{ i\gamma(t_n)|y_n + N_n^{-1}x|^2}
e^{\fr{iM_n^2 \sin(t_n)\Delta}{2}} [ e^{i\gamma(t_n)|x_n + M_n^{-1}
  \cdot|^2} \phi] (\tfr{M_n}{N_n} x + M_n(y_n - x_n)).
\end{split}
\]
where
\[
\gamma(t) = \tfr{\cos t - 1}{2\sin t} = -\tfr{t}{4} + O(t^3).
\]
We see that
\[
e^{i\gamma(t_n) |x_n + M_n^{-1} \cdot |^2} \phi \to
e^{-\fr{it_\infty (r_\infty)^2}{4}} \phi \quad \text{in} \quad \Sigma.
\]
Indeed,
\[
\begin{split}
&\|\nabla[ e^{i\gamma(t_n) |x_n + M_n^{-1} \cdot|^2} \phi  - e^{
  i\gamma(t_n) |x_n|^2} \phi] \|_{L^2} = \| \nabla_x [( e^{i\gamma(t_n)
  [M_n^{-2} |x|^2 + M_n^{-1} x_n \cdot x] } - 1)\phi] \|_{L^2}\\
&\lesssim \| t_n (M_n^{-2} x + M_n^{-1} x_n) \phi\|_{L^2} + \| (e^{
  i\gamma(t_n) [ M_n^{-2} |x|^2 + 2M_n^{-1} x_n \cdot x]} - 1)\nabla
\phi \|_{L^2}\\
&\lesssim |t_n| M_n^{-2} \| x\phi\|_{L^2} + |t_n| |x_n| M_n^{-1}
\|\phi\|_{L^2} + \| (e^{ i\gamma(t_n)[M_n^{-2} |x|^2 2M_n^{-1} x_n
  \cdot x]} - 1) \nabla \phi\|_{L^2}.
\end{split}
\]
As $n \to \infty$, the first two terms vanish because $\|x\phi\|_2 +
\|\phi\|_{2} \lesssim \|\phi\|_{\Sigma}$, while the third term vanishes
by dominated convergence. Dominated convergence also implies that
\[
\|x[ e^{i\gamma(t_n)|x_n + M_n^{-1}x|^2}\phi - e^{
  i\gamma(t_n)|x_n|^2}\phi] \|_{L^2} \to 0 \text{ as } n \to \infty.
\]
On the other hand, since
\[
\gamma(t_n) |x_n|^2 = -\tfr{M_n^2 t_n M_n^{-2} |x_n|^2}{4} +
O(M_n^{-4}) \to -\fr{ t_\infty (r_\infty)^2}{4},
\]
it follows that
\[
\| e^{i\gamma(t_n) |x_n + M_n^{-1} \cdot|^2} \phi - e^{-\fr{i t_\infty
    (r_\infty)^2}{4}} \phi\|_{\Sigma} \to 0
\]
as claimed. Now, using that $e^{\fr{iM_n^{2} \sin(t_n)\Delta}{2}} \to
e^{\fr{it_\infty \Delta}{2}}$ in the strong operator topology on
$B(\Sigma, \Sigma)$, we obtain
\[
e^{\fr{iM_n^{2} \sin(t_n) \Delta}{2}} [ e^{i\gamma(t_n) |x_n + M_n^{-1} \cdot|^2}
\phi] \to e^{-\fr{it_\infty (r_\infty)^2}{4}} e^{ \fr{it_\infty
    \Delta}{2}} \phi \text{ in } \Sigma,
\]
and the full conclusion quickly follows.
\end{proof}

\begin{cor}
\label{cor:strong_conv_cor}
Let $\{ (t_n^M, x_n, M_n, M_n')\}$ and $\{ (t_n^N, y_n, N_n, N_n')\}$
be augmented frames such that $\{(t_n^M, x_n, M_n)\}$ and $\{(t_n^N,
y_n, N_n)\}$ are equivalent. Let $S_n^M, \ S_n^N$ be the associated
spatial cutoff operators as defined in
\eqref{eqn:inv_strichartz_frameops_2}. Then
\begin{equation}
\label{eqn:strong_conv_cor_eqn1}
\lim_{n \to \infty} \| e^{-it_n^M H} G_n^M S_n^M \phi - e^{-it_n^N H}
G_n^N S_n^N U_\infty \phi
\|_{\Sigma} = 0
\end{equation}
and
\begin{equation}
\label{eqn:strong_conv_cor_eqn2}
\lim_{n\to \infty} \| e^{-it_n^M H} G_n^M S_n^M \phi - e^{-it_n^N H} G_n^N U_\infty S_n^N \phi
\|_{\Sigma} = 0
\end{equation}
whenever $\phi \in H^1$ if the frames conform to case
\ref{enum:aug_frame_2_1} and $\phi \in \dot{H}^1$ if they conform to
case \ref{enum:aug_frame_2_2} in Definition \ref{def:aug_frame}. 
\end{cor}

\begin{proof}
  As before, the result is immediate if $M_n \equiv 1$ since all
  operators in sight are trivial. Thus we may assume $M_n \to
  \infty$. Suppose first that $\phi \in
  C^\infty_c$. Using the unitarity of $e^{-itH}$ on $\Sigma$, the
  operator bounds \eqref{eqn:inv_strichartz_frameops_properties}, and the
  fact that $S_n^M\phi = \phi$ for all $n$
  sufficiently large, we write the left side
  of \eqref{eqn:strong_conv_cor_eqn1} as
\[
\begin{split}
 &\| G_n^N [
(G_n^N)^{-1} e^{-i(t_n^M - t_n^N) H} G_n^M \phi - S_n^N U_\infty \phi]
\|_{\Sigma} \\
&\lesssim \| (G_n^N)^{-1} e^{-i(t_n^M - t_n^N) H} G_n^M \phi - S_n^NU_\infty
\phi\|_{\Sigma}\\
&\lesssim \| (G_n^N)^{-1} e^{-i(t_n^M - t_n^N) H} G_n^M \phi -
U_\infty \phi \|_{\Sigma} + \|(1 - S_n^N)U_{\infty} \phi\|_{\Sigma}
\end{split}
\]
which goes to zero by Lemma~\ref{lma:strong_convergence} and dominated
convergence. This proves \eqref{eqn:strong_conv_cor_eqn1} under the
additional hypothesis that $\phi \in C^\infty_c$. 

We now remove this crutch and take  $\phi \in H^1$ or $\dot{H}^1$
depending on whether the frames are of type
\ref{enum:aug_frame_2_1} or \ref{enum:aug_frame_2_2} in Definition \ref{def:aug_frame}, respectively. For each
$\varepsilon > 0$, choose $\phi^{\varepsilon} \in C^\infty_c$ such
that $\| \phi - \phi^{\varepsilon} \|_{H^1} < \varepsilon$ or $\| \phi
- \phi^{\varepsilon} \|_{\dot{H}^1} < \varepsilon$, respectively. Then
\[
\begin{split}
  &\| e^{-it_n^M H} G_n^M S_n^M \phi - e^{-it_n^N H} G^N_n S_n^N
  U_\infty \phi\|_{\Sigma} \le \| e^{-it_n^M H} G_n^M S^M_n (\phi
  -\phi^{\varepsilon})
  \|_{\Sigma} \\
  &+ \| e^{-it_n H} G_n^M S_n^M \phi^{\varepsilon} - e^{-it_n^N H}
  G_n^N S_n^N U_\infty \phi^{\varepsilon}\|_{\Sigma} + \| e^{-it_n^N
    H} G_n^N S_n^N U_{\infty} (\phi - \phi^{\varepsilon}) \|_{\Sigma}
\end{split}
\]
In the limit as $n \to \infty$, the middle term vanishes and we are
left with a quantity at most a constant times
\[
\limsup_{n \to \infty} \|G_n^M S_n^M (\phi - \phi^{\varepsilon})
\|_{\Sigma} + \limsup_{n \to \infty} \| G_n^N S_n^N U_\infty( \phi
- \phi^{\varepsilon} ) \|_{\Sigma}.
\]
Applying Lemma~\ref{lma:approximation_lemma} and using the mapping
properties of $U_{\infty}$ on $\dot{H}^1$ and $H^1$, we see that 
\[
\limsup_{n \to \infty} \| e^{-it_nH} G_n^M S_n^M \phi - e^{it_n^N H} G_n^N S_n^N
U_\infty \phi\|_{\Sigma} \lesssim \varepsilon
\] 
for every $\varepsilon > 0$. This proves the claim
\eqref{eqn:strong_conv_cor_eqn1}. Similar considerations handle the
second claim
\eqref{eqn:strong_conv_cor_eqn2}.
\end{proof}

\begin{lma}
\label{lma:approximate_adjoint}
Suppose the frames $\{(t_n^M, x_n, M_n)\}$ and $\{ (t_n^N,
y_n, N_n)\}$ are equivalent. Put $t_n = t_n^M - t_n^N$. Then for $f, g
\in \Sigma$ we have
\[
 \langle (G_n^N)^{-1} e^{-it_n H} G_n^M f, g \rangle_{\dot{H}^1} = \langle 
f, (G_n^M)^{-1} e^{it_n H} G_n^N g \rangle_{\dot{H}^1} + R_{n} (f, g),
\]
where $|R_n (f, g)| \le C|t_n| \| G_n^M f\|_{\Sigma} \| G_n^N g\|_{\Sigma}$.  
\end{lma}

\begin{rmk}
We regard this as an ``approximate adjoint'' formula; note that
$e^{-itH}$ is not actually defined on all of $\dot{H}^1$. 
It follows from Lemma \ref{lma:strong_convergence} that
\[
\lim_{n \to \infty} \langle (G_n^N)^{-1} e^{-it_n H} G_n^M f, g
\rangle_{\dot{H}^1} = \lim_{n \to \infty} \langle f, (G_n^M)^{-1}
e^{it_n H} G_n^N g \rangle_{\dot{H}^1}
\]
for fixed $f, g \in \Sigma$. The content of this lemma lies in the
quantitative error bound.
\end{rmk}

\begin{proof}
From the 
identities \eqref{eqn:time_evolution_of_position_and_momentum}, we obtain 
the commutator estimate
\[
\| [\nabla, e^{-itH} ] \|_{\Sigma \to L^2} = O(t).
\]
By straightforward manipulations, we obtain
\[
\langle (G_n^N)^{-1} e^{-it_n H} G_N^Mf, g \rangle_{\dot{H}^1} =
\langle f, (G_n^M)^{-1} e^{it_n H} G_n^N g \rangle_{\dot{H}^1} +
R_n(f, g)
\]
where $R_n(f, g) = \langle [\nabla, e^{-it_n H} ] G_n^M f, \nabla
G_n^N g \rangle_{L^2} - \langle \nabla G_n^M f, [ \nabla, e^{it_n H} ]
G_n^N g \rangle_{L^2}$. The claim follows from Cauchy-Schwartz and the
above commutator estimate.

\end{proof}

The next lemma is a converse to Lemma~\ref{lma:strong_convergence}.
\begin{lma}[Weak convergence]
\label{lma:weak_convergence}
Assume the frames $\mcal{F}^M = \{ (t_n^M, x_n, M_n)\}$ and $\mcal{F}^N
= \{ (t_n^N, y_n, N_n) \}$ are orthogonal. 
Then, for any $f \in \Sigma$,
\[
(e^{-it_n^N H} G_n^N)^{-1} e^{-it_n^M H} G_n^M f \to 0 \quad \text{weakly in} \quad \dot{H}^1.
\]

\end{lma}

\begin{proof}

  Put $t_n = t_n^M - t_n^N$, and suppose that $|M_n^2 t_n|
  \to \infty$. Then 
  \[
  \| (G_n^N)^{-1} e^{-it_n H} G_n^M
  f\|_{L^{\fr{2d}{d-2}}} \to 0
  \]
  for $f \in C^\infty_c$ by a change of
  variables and the dispersive estimate, thus for general $f \in
  \Sigma$ by a density argument. Therefore $(G_n^N)^{-1} e^{-it_n H}
  G_n^M f$ converges weakly in $\dot{H}^1$ to $0$. We consider next
  the case where $M_n^{2} t_n \to t_\infty \in \mf{R}$. The
  orthogonality of $\mcal{F}^M$ and $\mcal{F}^N$ implies that either
  $N_n^{-1} M_n$ converges to $0$ or $\infty$, or $M_n |x_n-y_n|$
  diverges as $n \to \infty$. In either case, one verifies easily that
  the operators $(G_n^N)^{-1} G_n^M$ converge to zero in the weak
  operator topology on $B(\dot{H}^1, \dot{H}^1)$. Applying
  Lemma~\ref{lma:strong_convergence}, we see that $(G_n^N)^{-1}
  e^{-it_n H} G_n^M f = (G_n^N)^{-1} G_n^M (G_n^M)^{-1} e^{-it_n H}
  G_n^M f$ converges to zero weakly in~$\dot{H}^1$.
\end{proof}

\begin{cor}
\label{cor:weak_conv_cor}
Let $\{(t_n^M, x_n, M_n, M_n')\}$ and $ \{ (t_n^N, y_n, N_n,
N_n')\}$ be augmented frames such that $\{(t_n^M, x_n, M_n)\}$ and
$\{(t_n^N, y_n, N_n)\}$ are orthogonal. Let $G_n^M, \ S_n^M$ and $G_n^N, S_n^N$ be the associated operators.
Then 
\[
(e^{-it_n^N H} G_n^N)^{-1} e^{-it_n^M H} G_n^M S_n^M \phi \rightharpoonup 0 \quad
\text{in} \quad \dot{H}^1
\]
whenever $\phi \in H^1$ if $\mcal{F}^M$ is of type
\ref{enum:aug_frame_2_1} and $\phi \in \dot{H}^1$ if $\mcal{F}^M$ is
of type \ref{enum:aug_frame_2_2}.
\end{cor}

\begin{proof}
If $\phi \in C^\infty_c$, then $S^M_n \phi = \phi$ for all large $n$, and the claim
follows from Lemma~\ref{lma:weak_convergence}. The case of general $\phi$ in
$H^1$ or $\dot{H}^1$ then follows from an approximation argument
similar to the one used in the proof of
Corollary~\ref{cor:strong_conv_cor}. 
\end{proof}





\subsection{End of Proof of Inverse Strichartz}
\label{subsection:inv_strichartz_part_2}
We return to the proof of Proposition~\ref{prop:inv_strichartz}. Let
us pause briefly to assess our progress. We have thus far identified a
frame $\{(t_n, x_n, N_n, N_n')\}$ and an associated profile $\phi_n$ such
that the sequence $N_n^2 t_n$ has a limit in $[-\infty, \infty]$ as $n
\to \infty$. The $\phi_n$ were shown to satisfy properties
\eqref{eqn:inv_strichartz_nontriviality_in_Sigma} through
\eqref{eqn:inv_strichartz_decoupling_in_Sigma} if either
$(t_n, x_n, N_n) = (0, 0, 1)$ or $N_n \to \infty$ and $N_n^2 t_n \to
\pm \infty$. Thus, it remains to prove that if $N_n \to \infty$ and
$N_n^2 t_n$ remains bounded, then we may modify the frame so that
$t_n$ is identically zero and find a profile $\phi_n$ corresponding to
this new frame which satisfies all the properties asserted in the
proposition. The following lemma will therefore complete the proof of Proposition \ref{prop:inv_strichartz}.

\begin{lma}
\label{lma:inv_strichartz_part_2}
Let $f_n \in \Sigma$ satisfy the hypotheses of Proposition
\ref{prop:inv_strichartz}. Suppose $\{(t_n, x_n, N_n, N_n')\}$ is an
augmented frame with $N_n \to \infty$ and $N_n^2 t_n \to t_\infty \in
\mf{R}$ as $n \to \infty$.  Then there is a profile $\phi_n' = G_nS_n
\phi'$ associated to the frame $\{ (0, x_n, N_n, N_n') \}$ such that
properties
\eqref{eqn:inv_strichartz_nontriviality_in_Sigma}, 
\eqref{eqn:inv_strichartz_potential_energy_decoupling},
and \eqref{eqn:inv_strichartz_decoupling_in_Sigma} hold with
$\phi_n'$ in place of $\phi_n$.
\end{lma}

\begin{proof}
Let $\phi_n = e^{it_n H} G_n S_n \phi$ be the profile defined by
\eqref{eqn:inv_strichartz_case_2_phi_n_defn}. We have already seen
that $\phi_n$ satisfies properties
\eqref{eqn:inv_strichartz_nontriviality_in_Sigma} and
\eqref{eqn:inv_strichartz_decoupling_in_Sigma}, and that
\[
\phi = \opn{\dot{H}^1-w-lim}_{n \to \infty} G_n^{-1} e^{-it_n H} f_n.
\]
As the sequence $G_n^{-1} f_n$ is bounded in $\dot{H}^1$, it has a
weak subsequential limit
\[
\phi' = \opn{\dot{H}^1-w-lim}_{n \to \infty} G_n^{-1} f_n.
\]
For any $\psi \in C^\infty_c$, we apply Lemma
\ref{lma:approximate_adjoint} with $f = G_n^{-1} e^{-it_n H}f_n$ to see that
\[
\begin{split}
\langle \phi', \psi \rangle_{\dot{H}^1} &= \lim_{n \to \infty} \langle
G_n^{-1} f_n, \psi \rangle_{\dot{H}^1} = \lim_{n \to \infty} \langle G_n^{-1} e^{it_n H}
G_n G_n^{-1} e^{-it_n H} f_n, \psi \rangle_{\dot{H}^1} \\
&= \lim_{n \to \infty} \langle G_n^{-1} e^{-it_n H} f_n, G_n^{-1}
e^{-it_n H} G_n \psi \rangle_{\dot{H}^1} = \langle \phi, U_{\infty}
\psi \rangle_{\dot{H}^1},
\end{split}
\]
where $U_\infty = \opn{s-lim}_{n \to \infty} G_n^{-1} e^{-it_n H} G_n$
is the strong operator limit guaranteed by Lemma
\ref{lma:strong_convergence}. As $U_\infty$ is unitary on $\dot{H}^1$,
we have the relation $\phi = U_\infty \phi'$. 

Put $\phi'_n = G_n S_n \phi'$. By Corollary~\ref{cor:strong_conv_cor}, 
\[
\| \phi_n - \phi_n' \|_{\Sigma} = \| e^{it_n H} G_n S_n \phi - G_n S_n
U_\infty^{-1} \phi \|_{\Sigma} \to 0 \quad \text{as} \quad n \to \infty.
\]
Hence $\phi_n'$ inherits property
\eqref{eqn:inv_strichartz_nontriviality_in_Sigma} from $\phi_n$. The
same proof as for $\phi_n$ shows that $\Sigma$ decoupling
\eqref{eqn:inv_strichartz_decoupling_in_Sigma} holds as well. It
remains to verify the last decoupling property
\eqref{eqn:inv_strichartz_potential_energy_decoupling}. As
$G_n^{-1} f_n$ converges weakly in $\dot{H}^1$ to $\phi'$, by
Rellich-Kondrashov and a diagonalization argument we may assume after
passing to a subsequence that $G_n^{-1} f_n$ converges to $\phi'$
almost everywhere on $\mf{R}^d$. By the Lemma
\ref{lma:refined_fatou}, the observation that $\lim_{n \to \infty} \|
G_n S_n \phi' - G_n \phi' \|_{\fr{2d}{d-2}} = 0$, and a change of
variables, we have
\[
\begin{split}
&\lim_{n \to \infty} \left[ \|f_n\|_{\fr{2d}{d-2}}^{\fr{2d}{d-2}} - \|f_n
- \phi_n' \|_{\fr{2d}{d-2}}^{\fr{2d}{d-2}} - \| \phi_n'\|_{
  \fr{2d}{d-2}}^{\fr{2d}{d-2}} \right] \\
&=
\lim_{n \to \infty} \left[ \| G_n^{-1}
f_n\|_{\fr{2d}{d-2}}^{\fr{2d}{d-2}} - \| G_n^{-1} f_n - \phi'
\|_{\fr{2d}{d-2}}^{\fr{2d}{d-2}} - \| \phi'\|_{
  \fr{2d}{d-2}}^{\fr{2d}{d-2}} \right]\\
&= 0.
\end{split}
\]
\end{proof}

\begin{rmk}
  As $\lim_{n\to \infty} \| \phi_n - \phi_n'\|_{\Sigma} = 0$, we
  see by Sobolev embedding that the decoupling
  \eqref{eqn:inv_strichartz_potential_energy_decoupling} also holds
  for the original profile $\phi_n = e^{it_n H} G_n S_n \phi$ with
  nonzero time parameter $t_n$.
\end{rmk}

\subsection{Linear profile decomposition}

We are ready to write down the linear 
profile decomposition. As before,
$I$ will denote a fixed interval containing $0$ of length at most 
$1$, and all spacetime norms are taken over $I \times \mf{R}^d$ unless
indicated otherwise.

\begin{prop}
\label{prop:lpd}
  Let $f_n$ be a bounded sequence in $\Sigma$. After passing to a
  subsequence, there exists $J^* \in \{0, 1, \dots\} \cup \{\infty\}$
  such that for each finite $1 \le j \le J^*$, there exist an
  augmented frame $\mcal{F}^j = \{(t_n^j, x_n^j, N_n^j, (N_n^j)')\}$
  and a function $\phi^j$ with the following properties.
\begin{itemize}
\item Either $t_n^j \equiv 0$ or $(N_n^j)^2 (t_n^j) \to \pm \infty$ as
  $n \to \infty$.
\item $\phi^j$ belongs to $\Sigma, \ H^1$, or $\dot{H}^1$ depending on
  whether $\mcal{F}^j$ is of type \ref{enum:aug_frame_1},
  \ref{enum:aug_frame_2_1}, or \ref{enum:aug_frame_2_2},
  respectively. 
\end{itemize}
For each finite $J \le J^*$, we have a decomposition 
\begin{equation}
\label{eqn:lpd_decomposition}
f_n = \sum_{j=1}^J e^{it_n^j H} G_n^j S_n^j \phi^j + r_n^J,
\end{equation}
where $G_n^j, \ S_n^j$ are the $\dot{H}^1$-isometry and spatial cutoff
operators associated to $\mcal{F}^j$. Writing $\phi^j_n$ for $e^{it_n^j H}
G_n^j S_n^j \phi^j$, this decomposition has the following properties:
\begin{gather}
\label{eqn:lpd_bubble_maximality}
(G_n^J)^{-1} e^{-it_n^J H} r_n^J \overset{\dot{H}^1}{\rightharpoonup}
0 \quad \text{for all} \ J \le J^*,\\
\label{eqn:lpd_sigma_decoupling}
\sup_{J} \lim_{n \to \infty} \Bigl | \|  f_n \|_{\Sigma}^2 - \sum_{j=1}^J \|
   \phi^j_n\|_{\Sigma}^2 - \| r_n^J\|_{\Sigma}^2 \Bigr | = 0,\\
\label{eqn:lpd_potential_energy_decoupling}
\sup_{J} \lim_{n \to \infty} \Bigl | \|f_n\|_{L^{\fr{2d}{d-2}}_x}^{\fr{2d}{d-2}}
  - \sum_{j=1}^J \| \phi^j_n\|_{L^{\fr{2d}{d-2}}_x}^{\fr{2d}{d-2}} -
  \| r_n^J \|_{L^{\fr{2d}{d-2}}_x}^{\fr{2d}{d-2}} \Bigr | = 0.
\end{gather}
Whenever $j \ne k$, the frames $\{ (t_n^j, x_n^j, N_n^j)\}$ and $\{ (t_n^k, x_n^k,
N_n^k)\}$ are orthogonal:
\begin{equation}
\label{eqn:lpd_orthogonality_of_frames}
\lim_{n \to \infty} \tfr{N_n^j}{N_n^k} + \tfr{N_n^k}{N_n^j} + N_n^j N_n^k |t_n^j
- t_n^k| + \sqrt{N_n^j N_n^k} |x_n^j - x_n^k| = \infty.
\end{equation}
Finally, we have
\begin{equation}
\label{eqn:lpd_vanishing_of_remainder}
\lim_{J \to J^*} \limsup_{n \to \infty} \| e^{-it_n H} r_n^J
\|_{L^{\fr{2(d+2)}{d-2}}_{t,x}} = 0,
\end{equation}
\end{prop}

\begin{rmk}
One can also show a posteriori using
\eqref{eqn:lpd_orthogonality_of_frames} and \eqref{eqn:lpd_vanishing_of_remainder} the fact, which we will
neither prove nor use, that
\[
\sup_J \lim_{n \to \infty} \Bigl| 
\|e^{-itH} f_n\|_{L^{\fr{2(d+2)}{d-2}}_{t,x}}^{
\fr{2(d+2)}{d-2}} - \sum_{j=1}^J 
\| e^{-itH} \phi^j_n\|_{L^{\fr{2(d+2)}{d-2}}_{t,x}}^{ 
\fr{2(d+2)}{d-2}} - \| e^{-itH} w_n^J\|_{L^{\fr{2(d+2)}{d-2}}_{t,x}}^{
\fr{2(d+2)}{d-2}} \Bigr | = 0.
\]
The argument uses similar ideas as in the proofs of
\cite{keraani_compactness_defect}[Lemma 2.7] or Lemma~\ref{lma:decoupling_of_nonlinear_profiles}; we omit the
details.
\end{rmk}

\begin{proof}
We proceed inductively using Proposition
\ref{prop:inv_strichartz}. Let $r_n^0 = f_n$. Assume that we have a
decomposition up to level $J \ge 0$ obeying properties
\eqref{eqn:lpd_bubble_maximality} through
\eqref{eqn:lpd_potential_energy_decoupling}. After passing to a
subsequence, we may define
\[
A_J = \lim_n \| r_n^J \|_{\Sigma} \quad \text{and} \quad 
\varepsilon_J = \lim_n \| e^{-it_n H} r_n^J
\|_{L^{\fr{2(d+2)}{d-2}}_{t,x}}.
\]
If $\varepsilon_J = 0$, stop and set $J^* = J$. Otherwise we apply
Proposition \ref{prop:inv_strichartz} to the sequence $r_n^J$ to
obtain a frame $(t_n^{J+1}, x_n^{J+1}, N_n^{J+1}, (N_n^{J+1})')$ and
functions 
\[
\phi^{J+1} \in \dot{H}^1, \quad \phi_n^{J+1} = e^{it_n^{J+1} H}
G_n^{J+1} S_n^{J+1} \phi^{J+1} \in \Sigma
\]
which satisfy the
conclusions of Proposition \ref{prop:inv_strichartz}. In particular
$\phi^{J+1}$ is the $\dot{H}^1$ weak limit of $(G_n^{J+1})^{-1}
e^{-it_n^{J+1}H} r_n^J$. Let $r_n^{J+1} = r_n^J - \phi_n^{J+1}$. By the induction hypothesis,
\eqref{eqn:lpd_sigma_decoupling} and
\eqref{eqn:lpd_potential_energy_decoupling} are satisfied with $J$ 
replaced by $J+1$. We also have 
\[
(G_n^{J+1})^{-1} e^{-it_n^{J+1} H} r_n^{J+1} = [ (G_n^{J+1})^{-1}
e^{-it_n^{J+1}H} r_n^J- \phi^{J+1}] + (1 - S_n^{J+1}) \phi^{J+1}.
\]
As $n \to \infty$, the first term goes to zero weakly in $\dot{H}^1$
while the second term goes to zero strongly. Thus
\eqref{eqn:lpd_bubble_maximality} holds at level $J+1$ as well. After
passing to a subsequence, we may define
\[
A_{J+1} = \lim_n \| r_n^{J+1} \|_{\Sigma} \quad \text{and} \quad
\varepsilon_{J+1} = \lim_n \| e^{-itH} r_n^{J+1} \|_{L^{\fr{2(d+2)}{d-2}}_{t,x}}.
\]
If $\varepsilon_{J+1} = 0$, stop and set $J^* = J+1$. Otherwise
continue the induction. If the algorithm never terminates, set $J^* =
\infty$. From \eqref{eqn:lpd_sigma_decoupling} and
\eqref{eqn:lpd_potential_energy_decoupling}, the parameters $A_J$ and
$\varepsilon_J$ satisfy the inequality
\[
A_{J+1}^2 \le A_J^2 [ 1 - C(\tfr{\varepsilon_J}{A_J}
)^{\fr{d(d+2)}{4}} ].
\]
If $\limsup_{J \to J^*} \varepsilon_J = \varepsilon_\infty > 0$, then
as $A_J$ are decreasing there would exist infinitely many $J$'s so
that
\[
A_{J+1}^2 \le A_J^2[1 - C(\tfr{\varepsilon_\infty}{A_0})^{\fr{d(d+2)}{4}}],
\]
which implies that $\lim_{J \to J^*} A_J = 0$. But this contradicts
the Strichartz inequality which dictates that $\limsup_{J \to J^*} A_J
\gtrsim \limsup_{J \to J^*} \varepsilon_J = \varepsilon_0$. We
conclude that
\[
\lim_{J \to J^*} \varepsilon_J = 0.
\]
Thus \eqref{eqn:lpd_vanishing_of_remainder} holds.

It remains to prove the assertion
\eqref{eqn:lpd_orthogonality_of_frames}. Suppose
otherwise, and let $j < k$ be the first two indices for which
$\mcal{F}^j$ and $\mcal{F}^k$ are equivalent. Thus $\mcal{F}^\ell$ and
$\mcal{F}^k$ are orthogonal for all $j < \ell < k$. By the
construction of the profiles, we have
\[
r_n^{j-1} = e^{it_n^j H} G_n^j S_n^j \phi^j + e^{it_n^k H}
G_n^k S_n^k \phi^k + \sum_{j < \ell < k}
e^{it_n^{\ell} H} G_n^{\ell} S_n^{\ell} \phi^{\ell}  + r_n^k,
\]
thus
\[
\begin{split}
  &(e^{it_n^j H} G_n^j)^{-1} r_n^{j-1} = (e^{it_n^jH} G_n^j )^{-1}
  e^{it_n^j H} G_n^j S_n^j \phi^j + (e^{it_n^j H} G_n^j)^{-1}
  e^{it_n^k H} G_n^k S_n^k \phi^k\\
  &+ \sum_{j < \ell < k} (e^{it_n^j H} G_n^j)^{-1} e^{it_n^{\ell} H}
  G_n^{\ell} S_n^{\ell} \phi^\ell + (e^{it_n^j H} G_n^j)^{-1} r_n^k.
\end{split}
\]
As $n \to \infty$, the left side converges to $\phi^j$ weakly in
$\dot{H}^1$. On the right side, we apply
Corollary~\ref{cor:strong_conv_cor} to see that the first and second terms
converge in $\dot{H}^1$ to $\phi^j$ and $U_{\infty}^{jk} \phi^k$,
respectively, for some isomorphism $U_{\infty}^{jk}$ of
$\dot{H}^1$. By Corollary~\ref{cor:weak_conv_cor}, each of the terms
in the summation converges to zero weakly in $\dot{H}^1$. Taking for
granted the claim that
\begin{equation}
\label{eqn:lpd_frame_ortho_final__weak_limit}
(e^{it_n^j H} G_n^j)^{-1} r_n^k \to 0 \quad \text{weakly in } \dot{H}^1,
\end{equation}
it follows that
\[
\phi^j = \phi^j + U_\infty^{jk}\phi^k,
\]
so $\phi^k = 0$, which contradicts the nontriviality of $\phi^k$. Therefore, the proof of the
proposition will be complete once we verify the weak limit
\eqref{eqn:lpd_frame_ortho_final__weak_limit}. As that sequence is
bounded in $\dot{H}^1$, it suffices to check that 
\[
\langle (e^{it_n^j H} G_n^j)^{-1} r_n^k, \psi \rangle_{\dot{H}^1} \to
0 \quad \text{for any } \psi \in C^\infty_c(\mf{R}^d).
\]
Writing $(e^{it_n^j H} G_n^j)^{-1}
r_n^k = (e^{it_n^j H} G_n^j)^{-1} (e^{it_n^k H} G_n^k) (e^{it_n^kH}
G_n^k)^{-1} r_n^k$, we apply Lemma~\ref{lma:approximate_adjoint} and the
weak limit \eqref{eqn:lpd_bubble_maximality} to see that
\[
\begin{split}
  \lim_{n\to \infty} \, \langle (e^{it_n^j H} G_n^j)^{-1} r_n^k, \psi
  \rangle_{\dot{H}^1} &= \lim_{n\to \infty} \, \langle (e^{it_n^kH} G_n^k)^{-1}
  r_n^k, (e^{it_n^kH} G_n^k)^{-1} (e^{it_n^jH} G_n^j)
  \psi\rangle_{\dot{H}^1}
  \\
  &= \lim_{n\to \infty} \, \langle (G_n^k)^{-1} e^{-it_n^kH} r_n^k, (U_{\infty}^{jk})^{-1} \psi
  \rangle_{\dot{H}^1} \\
&= 0.
\end{split}
\]
\end{proof}

\section{The case of concentrated initial data}
\label{section:localized_initial_data}

The next step in the proof of Theorem~\ref{thm:main_theorem} is to
establish wellposedness when the initial data consists of a highly
concentrated ``bubble''. The picture to keep in mind is that of a
single profile $\phi^j_n$ in Proposition \ref{prop:lpd} as $n\to
\infty$. In the next section we combine this special case with the
profile decomposition to treat general initial data. Although we state the following result as a conditional one to
permit a unified exposition, by Theorem~\ref{thm:ckstt} the result is
unconditionally true in most cases.

\begin{prop}
\label{prop:localized_initial_data}
Let $I = [-1, 1]$.  Assume that Conjecture \ref{conjecture:ckstt}
holds. Suppose
  \[
  \mcal{F} = \{ (t_n, x_n, N_n, N_n')\}
  \]
  is an augmented frame with $t \in I$ and
  $N_n \to \infty$, such that either $t_n \equiv 0$ or $N_n^2 t_n \to
  \pm \infty$; that is, $\mcal{F}$ is type \ref{enum:aug_frame_2_1} or
  \ref{enum:aug_frame_2_2} in Definition \ref{def:aug_frame}. Let
  $G_n, \tilde{G}_n$, and $S_n$ be the associated operators as defined
  in \eqref{eqn:inv_strichartz_frameops_1} and
  \eqref{eqn:inv_strichartz_frameops_2}. Suppose $\phi$ belongs to $H^1$
  or $\dot{H}^1$ depending on whether $\mcal{F}$ is type
  \ref{enum:aug_frame_2_1} or \ref{enum:aug_frame_2_2}
  respectively. Then, for $n$ sufficiently large, there is a unique
  solution $u_n : I \times \mf{R}^d \to \mf{C}$ to the defocusing
  equation \eqref{eqn:nls_with_quadratic_potential}, $\mu = 1$, with
  initial data
\[
u_n(0)=  e^{it_n H} G_n S_n \phi.
\]
This solution satisfies a spacetime bound
\[
\limsup_{n \to \infty} S_I(u_n) \le C( E(u_n)).
\]
Suppose in addition that $\{ (q_k, r_k) \}$ is any finite collection
of admissible pairs with $2 < r_k < d$. Then for each $\varepsilon >
0$ there exists $\psi^{\varepsilon} \in C^\infty_c(\mf{R} \times
\mf{R}^d)$ such that
\begin{equation}
\label{eqn:approximation_by_smooth_functions}
\limsup_{n \to \infty} \sum_k \| u_n - \tilde{G}_n [e^{-
  \fr{itN_n^{-2} |x_n|^2}{2}}  \psi^{\varepsilon}] \|_{L^{q_k}_t
  \Sigma^{r_k}_x (I \times \mf{R}^d)} < \varepsilon.
\end{equation}

Assuming also that $\| \nabla \phi\|_{L^2} < \| \nabla W\|_{L^2}$ and
$E_{\Delta}(\phi) < E_{\Delta}(W)$, we have the same conclusion as above for the
focusing equation \eqref{eqn:nls_with_quadratic_potential}, $\mu =
-1$.
\end{prop}

The proof proceeds in several steps. First we construct an approximate
solution on $I$ in the sense of Proposition \ref{prop:stability}.
Roughly speaking, when $N_n$ is large and $t = O(N_n^{-2})$, solutions
to \eqref{eqn:nls_with_quadratic_potential} are well-approximated up
to a phase factor by solutions to the energy-critical NLS with no
potential, which by Conjecture~\ref{conjecture:ckstt} exist globally
and scatter.
In the long-time regime $|t| >> N_n^{-2}$, the solution
to \eqref{eqn:nls_with_quadratic_potential} has dispersed and
resembles a linear evolution $e^{-itH} \phi$ (note that we are
\emph{not} claiming scattering since we consider only a fixed finite
time window).
By patching these approximations together, we obtain an approximate
solution over the entire time interval $I$ with arbitrarily small
error as $N_n$ becomes large. We then invoke Proposition
\ref{prop:stability} to conclude that for $n$ large enough
\eqref{eqn:nls_with_quadratic_potential} admits a solution on $I$ with
controlled spacetime bound. The last claim about approximating the
solution by functions in $C^\infty_c(\mf{R} \times \mf{R}^d)$ will
follow essentially from our construction of the approximate solutions.

We first record a basic commutator estimate. 
Throughout the rest of this section, 
  $P_{\le N}, P_{N}$ will
denote the standard Littlewood-Paley projectors based on $-\Delta$. 
\begin{lma}[Commutator estimate]
\label{lma:CKSTT_nonlinearity_estimates}
Let $v$ be a global solution to 
\[
(i\partial_t + \tfr{1}{2}\Delta )v =
F(v), \ v(0) \in \dot{H}^1(\mf{R}^d)
\]
where $F(z) = \pm |z|^{\fr{4}{d-2}}
z$. Then on any compact time interval $I$,
\[
\lim_{N \to \infty} \| P_{\le N} F(v) - F(P_{\le N} v)\|_{L^2_t H^{1, \fr{2d}{d+2}}_x (I \times \mf{R}^d)} = 0
\]
\end{lma}

\begin{proof}
  We recall \cite[Lemma 3.11]{matador} that as a consequence of the
  spacetime bound \eqref{eqn:ckstt_spacetime_bound}, $\nabla v$ is
  finite in all Strichartz norms:
\begin{equation}
\label{eqn:ckstt_all_strichartz_bounds}
\| \nabla v\|_{S(\mf{R})} < C( \|v(0) \|_{\dot{H}^1} ) < \infty.
\end{equation}
Clearly it will suffice to show separately that 
\begin{align}
\lim_{n \to \infty} \| P_{\le N} F(v) - F(P_{\le N} v)\|_{L^2_t
  L^{\fr{2d}{d+2}}_x} &= 0, \label{eqn:commutator_noderiv}\\
 \lim_{n \to \infty} \| \nabla[ P_{\le N}
F(v) - F(P_{\le N} v) ] \|_{L^2_t L^{\fr{2d}{d+2}}_x} &= 0. \label{eqn:commutator_deriv}
\end{align}
Write
\begin{equation}
\begin{split}
\label{eqn:CKSTT_lma_eqn_1}
\| \nabla [ P_{\le N} F(v) - F(P_{\le N} v)]\|_{L^2_t L^{\fr{2d}{d+2}}_x} &\le \| \nabla P_{> N} F(v)\|_{L^2_t
  L^{\fr{2d}{d+2}}_x }  \\
&+ \| \nabla [F(v) -
F(P_{\le N} v)] \|_{L^2_t L^{\fr{2d}{d+2}}_x}.
\end{split}
\end{equation}
As $P_{>N} = 1 - P_{\le N}$ and 
\begin{align*}
\| \nabla F(v)\|_{L^2_t L^{\fr{2d}{d+2}}_x} \lesssim
\|v\|^{\fr{4}{d-2}}_{L^{\fr{2(d+2)}{d-2}}_{t, x} } \| \nabla
v\|_{L^{\fr{2(d+2)}{d-2}}_t L^{\fr{2d(d+2)}{d^2+4}}_x } \le
C(\| v(0)\|_{\dot{H}^1}),
\end{align*}
dominated convergence implies that 
\[
\lim_{N \to \infty} \| \nabla P_{> N} F(v)\|_{L^2_t
  L^{\fr{2d}{d+2}}_x} = 0. 
\]
To treat the second term on the right side of
\eqref{eqn:CKSTT_lma_eqn_1}, observe first that with $F(z) =
|z|^{\fr{4}{d-2}}z$,
\begin{align*}
|F_{z}(z) - F_{z}(w)| + |F_{\overline{z}} (z) - F_{\overline{z}}(w)|
\lesssim \left\{ \begin{array}{cc} |z-w|(|z|^{\fr{6-d}{d-2}} +
    |w|^{\fr{6-d}{d-2}}), & 3 \le d \le 5\\
|z-w|^{\fr{4}{d-2}}, & d \ge 6.
\end{array}\right.
\end{align*}
Combining this with the pointwise bound
\[
\begin{split}
|\nabla [F(v) - F(P_{\le N} v)]| &\le ( |F_{z}(v) - F_z(P_{\le N} v)
| + |F_{\overline{z}} (v) - F_{\overline{z}}(P_{\le N} v)| ) | \nabla
v|\\
&+ ( |F_{z}(P_{\le N} v)| + |F_{\overline{z}} (P_{\le N} v)|) |\nabla
P_{>N} v|,
\end{split}
\]
H\"{o}lder, and dominated convergence, when $d \ge 6$ we have
\begin{equation}
\begin{split}
\label{eqn:CKSTT_lma_eqn_2}
&\| \nabla [ F(v) - F(P_{\le N} v)] \|_{L^2_t
  L^{\fr{2d}{d+2}}_x} \\
&\lesssim \| |P_{>N} v|^{\fr{4}{d-2}}
|\nabla v| \|_{L^2_t L^{\fr{2d}{d+2}}_x} + \| |P_{\le N}
v|^{\fr{4}{d-2}} |\nabla P_{>N} v| \|_{L^2_t
  L^{\fr{2d}{d+2}}_x} \\
&\lesssim \| P_{> N} v \|_{L^{\fr{2(d+2)}{d-2}}_{t,
    x}}^{\fr{4}{d-2}} \| \nabla v\|_{L^{\fr{2(d+2)}{d-2}}_t
  L^{\fr{2d(d+2)}{d^2+4}}_x} + \|
v\|_{L^{\fr{2(d+2)}{d-2}}_{t,x}}^{\fr{4}{d-2}} \| P_{> N}
\nabla v\|_{L^{\fr{2(d+2)}{d-2}}_t L^{\fr{2d(d+2)}{d^2+4}}_x} \\
&\to 0 \text{ as } N \to \infty. 
\end{split}
\end{equation}
If $3 \le d \le 5$, the first term in the second line of \eqref{eqn:CKSTT_lma_eqn_2} is replaced by
\begin{align*}
&\| |P_{>N} v| (|v|^{\fr{6-d}{d-2}} + |P_{\le N}
v|^{\fr{6-d}{d-2}}) |\nabla v| \|_{L^2_t L^{\fr{2d}{d+2}}_x} \\
&\le \|
P_{> N} v\|_{L^{\fr{2(d+2)}{d-2}}_{t, x}} \|
v\|_{L^{\fr{2(d+2)}{d-2}}_{t, x}}^{\fr{6-d}{d-2}} \| \nabla
v\|_{L^{\fr{2(d+2)}{d-2}}_t L^{\fr{2d(d+2)}{d^2+4}}_x}
\end{align*}
which goes to $0$ by dominated convergence. This establishes \eqref{eqn:commutator_deriv}. The proof of \eqref{eqn:commutator_noderiv}is similar. Write
\begin{align*}
\|P_{\le N} F(v) - F(P_{\le N} v)\|_{L^2_t L^{\fr{2d}{d+2}}_x} \le
\| P_{>N} F(v)\|_{L^2_t L^{\fr{2d}{d+2}}_x} + \|
F(v) - F(P_{\le N} v)\|_{L^2_tL^{\fr{2d}{d+2}}_x}.
\end{align*}
By H\"{o}lder, Bernstein, and the chain rule,
\begin{align*}
\|P_{>N} F(v) \|_{L^2_tL^{\fr{2d}{d+2}}_x} \lesssim N^{-1} \|
v|_{L^{\fr{2(d+2)}{d-2}}_{t, x}}^{\fr{4}{d-2}} \| \nabla
v\|_{L^{\fr{2(d+2)}{d-2}}_t L^{\fr{2d(d+2)}{d^2+4}}_x} = O(N^{-1}).
\end{align*}
Using Bernstein, H\"{o}lder, and Sobolev embedding, and the pointwise
bound
\[
|F(v) - F(P_{\le N}v)| \lesssim |P_{>N} v| (|v|^{\fr{4}{d-2}} +
|P_{\le N} v|^{\fr{4}{d-2}} ),
\]
we obtain
\begin{align*}
\|F(v) - F(P_{\le N} v)\|_{L^2_t L^{\fr{2d}{d+2}}_x} &\le \|(
|v|^{\fr{4}{d-2}} +  |P_{\le N}v|^{\fr{4}{d-2}}) P_{> N} v \|_{L^2_t L^{\fr{2d}{d+2}}_x}\\
&\lesssim_{ |I|} (\| \nabla v\|_{L^\infty_t L^2_x}^{\fr{4}{d-2}}+
+  \| \nabla v\|_{L^\infty_t L^2_x}^{\fr{4}{d-2}} )\| \nabla P_{>N} v \|_{L^\infty_t L^2_x}.
\end{align*}
As $v \in C^0_t \dot{H}^1_x(I \times \mf{R}^d)$, the orbit
$\{v(t)\}_{t \in I}$ is compact in  $\dot{H}^1(\mf{R}^d)$. The Riesz characterization of $L^2$ compactness therefore 
implies that the right side goes to $0$ as $N \to \infty$. 
\end{proof}

Now suppose that $\phi_n = e^{it_n H} G_n S_n \phi$ as in the 
statement of Proposition~\ref{prop:localized_initial_data}. If 
$\mu = -1$, assume also that 
$\| \phi\|_{\dot{H}^1} < \| W\|_{\dot{H}^1}, \ E(\phi) < E_{\Delta}(W)$. 
We first  construct ``quasi-approximate'' solutions $\tilde{v}_n$ which obey all 
of 
the conditions of the Proposition~\ref{prop:stability} except
possibly the hypothesis in
\eqref{eqn:stability_prop_hyp_2} about matching 
initial data. A slight modification of the $\tilde{v}_n$ will then yield 
genuine approximate solutions.

If $t_n \equiv 0$, let $v$ be the solution to the potential-free
problem \eqref{eqn:ckstt_eqn}
provided by Conjecture~\ref{conjecture:ckstt} with $v(0) 
= \phi$. If $N_n^2 t_n \to \pm \infty$, let $v$ be the solution to 
\eqref{eqn:ckstt_eqn} which scatters in 
$\dot{H}^1$ to
 $e^{\fr{it\Delta}{2}} \phi$ as $t 
\to \mp \infty$. Note the reversal of signs. 

Put 
\begin{equation}
\tilde{N}_n' = ( \tfr{N_n}{N_n'} )^{\fr{1}{2}},
\label{eqn:approx_solution_freq_cutoffs}
\end{equation}
Let $T > 0$ denote a large constant to be chosen later, 
and define
\begin{equation}
\label{eqn:localized_initial_data_v_n-tilde_defn}
\tilde{v}_n^T(t) = \left\{ \begin{array}{cc}  e^{ -\fr{it|x_n|^2}{2}} \tilde{G}_n[
    S_n P_{\le \tilde{N}_n'} v] (t+t_n)  & |t| \le TN_n^{-2}\\
e^{-i(t-TN_n^{-2}) H} \tilde{v}_n^T(TN_n^{-2}), & TN_n^{-2} \le t \le 2\\
e^{-i(t+TN_n^{-2})H} \tilde{v}_n^T(-TN_n^{-2}), & -2 \le t \le -TN_n^{-2}
\end{array}\right.
\end{equation}
The time translation by $t_n$ is needed to undo
the time translation built into the operator $\tilde{G}_n$; see
\eqref{eqn:inv_strichartz_frameops_1}. We will suppress the
superscript $T$ unless we need to emphasize the role of that parameter.
Introducing the notation
\[
\begin{split}
v_n(t, x) &= [\tilde{G}_nv](t + t_n, x) = N_n^{\fr{d-2}{2}} v(N_n^2t,
N_n(x-x_n)),\\
 \chi_n(x) &= \chi(N_n'(x - x_n)),
\end{split}
\]
where $\chi$ is the function used to define the spatial cutoff
operator $S_n$ in \eqref{eqn:inv_strichartz_frameops_2}, and using the
identity $\tilde{G}_n \chi =
\chi_n \tilde{G}_n$, we can also write the top expression in
\eqref{eqn:localized_initial_data_v_n-tilde_defn} as
\[
\tilde{v}_n(t) = e^{-\fr{it|x_n|^2}{2}} \chi_n P_{\le
      \tilde{N}_n' N_n} v_n, \quad |t| \le TN_n^{-2}.
\]

As discussed previously, inside the ``potential-free'' window $\tilde{v}_n$ is essentially 
a modulated solution to \eqref{eqn:ckstt_eqn} with cutoffs applied in both space, to place the solution in 
$C_t \Sigma_x$, 
and frequency, to enable taking an extra derivative in the error
analysis below. 

On the time interval $|t| \le TN_n^{-2}$, we use Lemma
\ref{lma:approximation_lemma} and the fact that
$\|v\|_{L^\infty_t \dot{H}^1_x} \lesssim \| \phi\|_{\dot{H}^1}$ to deduce
\[
\limsup_{n} \| \tilde{v}_n\|_{L^\infty_t \Sigma_x ( |t| \le TN_n^{-2})} \lesssim \| \phi\|_{\dot{H}^1},
\]
therefore 
\begin{align}
\limsup_{n} \| \tilde{v}_n\|_{L^\infty_t \Sigma_x ( [-2, 2])} \lesssim \| \phi\|_{\dot{H}^1}. \label{eqn:approx_solution_energy_bound}
\end{align}
 From \eqref{eqn:ckstt_spacetime_bound}, \eqref{eqn:approx_solution_energy_bound}, and Strichartz, 
we obtain 
\begin{align}
\| \tilde{v}_n\|_{L^{\fr{2(d+2)}{d-2}}_{t, x}([-2,2] \times \mf{R}^d)} \le
C( \| \phi\|_{\dot{H}^1}) \label{eqn:approx_solution_scattering_size_bound} \quad \text{for
} n \text{ large}.
\end{align}
Let 
\[e_n = (i\partial_t - H) \tilde{v}_n -F(\tilde{v}_n).
\]
We show that
\begin{align}
  \lim_{T \to \infty} \limsup_{n \to \infty} \| H^{\fr{1}{2}}
  e_n\|_{N([-2, 2])} =
  0, \label{eqn:approximate_solution_nonlinear_error_estimate}
\end{align}
so that by taking $T$ large enough the $\tilde{v}_n$ will satisfy the
second error condition in \eqref{eqn:stability_prop_hyp_2} for all $n$
sufficiently large. Our first task is to deal with the time interval
$|t| \le TN_n^{-2}$.
\begin{lma}
\label{lma:short_time_approximation}
$\lim_{T \to \infty} \limsup_{n \to \infty} \| H^{\fr{1}{2}}
e_n\|_{N( |t| \le TN_n^{-2})} = 0.$
\end{lma}

\begin{proof}
When $-TN_n^{-2} \le
t \le TN_n^{-2}$, we compute
\begin{align*}
e_n &= e^{- \fr{it |x_n|^2}{2}} [ 
\chi_n P_{\le \tilde{N}_n' N_n} F(v_n)   -
\chi_n^{\fr{d+2}{d-2}}F(P_{\le \tilde{N}_n' N_n} v_n)  \\ 
&+ \fr{ |x_n|^2 - |x|^2}{2} (P_{\le \tilde{N}_n' N_n} v_n) \chi_n +
\fr{1}{2} (P_{\le \tilde{N}_n' N_n} v_n) \Delta \chi_n + (\nabla P_{
  \le \tilde{N}_n' N_n} v_n) \cdot \nabla \chi_n]\\
&= e^{- \fr{it|x_n|^2}{2}} [(a) + (b) + (c) + (d)],
\end{align*}
and estimate each term separately in the dual Strichartz space $N(\{|t|
\le TN_n^{-2} \})$. Write
\begin{align*}
(a) &= \chi_n P_{\le \tilde{N}_n' N_n}  F(v_n)  -
\chi_n^{\fr{d+2}{d-2}} F(
P_{\le \tilde{N}_n' N_n} v_n) \\
&= \chi_n[ P_{\le \tilde{N}_n' N_n} F(v_n) - F( P_{\le_{\tilde{N}_n' N_n}}
v_n)]  + \chi_n(1 - \chi_n^{\fr{4}{d-2}}) F (P_{\le \tilde{N}_n' N_n} v_n) \\
&= (a') + (a'').
\end{align*}
By the Leibniz rule and a change of variables,
\begin{equation}
\begin{split}
&\| \nabla (a')\|_{L^2_t L^{\fr{2d}{d+2}}_x (|t| \le TN_n^{-2})}
\\
&\le \| \nabla[P_{\le \tilde{N}_n'} F(v) - F(P_{\le \tilde{N}_n'}
v)]\|_{L^2_t L^{\fr{2d}{d+2}}_x (|t| \le T)} \\
&+ \| [P_{\le
  \tilde{N}_n' N_n } F(v_n) - F( P_{\le \tilde{N}_n' N_n} v_n)] \nabla
\chi_n \|_{L^2_t L^{\fr{2d}{d+2}}_x(|t| \le TN_n^{-2})}. \label{eqn:a'_eqn_0}
\end{split}
\end{equation}
By Lemma~\ref{lma:CKSTT_nonlinearity_estimates}, the first term
disappears in the limit as $n \to \infty$. That lemma also applies to the second term after a
change of variables to give
\begin{align*}
&\| [P_{\le
  \tilde{N}_n' N_n } F(v_n) - F( P_{\le \tilde{N}_n' N_n} v_n)] \nabla
\chi_n \|_{L^2_t L^{\fr{2d}{d+2}}_x(|t| \le TN_n^{-2})}\\
 &\lesssim N_n' \| P_{\le \tilde{N}_n' N_n} F(v_n) - F(P_{\le
  \tilde{N}_n' N_n} v_n)\|_{L^2_t L^{\fr{2d}{d+2}}_x (|t| \le T
  N_n^{-2})}\\
&\lesssim \tfr{N_n'}{N_n} \| P_{\le \tilde{N}_n'} F(v) - F(P_{\le
  \tilde{N}_n'} v) \|_{L^2_t L^{\fr{2d}{d+2}}_x (|t| \le T)} \to 0 \text{ as } n \to \infty.
\end{align*}
Therefore 
\[
\lim_{n \to \infty} \| \nabla(a')\|_{L^2_t
  L^{\fr{2d}{d+2}}_x(|t| \le TN_n^{-2})} = 0.
\]
 By changing variables,
using the bound $|x_n| \lesssim N_n$, and referring to Lemma~\ref{lma:CKSTT_nonlinearity_estimates} once more,
\begin{align*}
&\| |x|(a')\|_{L^2_t L^{\fr{2d}{d+2}}_x} \lesssim N_n \| P_{\le
  \tilde{N}_n' N_n} F(v_n) - F(P_{\le \tilde{N}_n' N_n}
v_n)\|_{L^2_tL^{\fr{2d}{d+2}}_x (|t| \le TN_n^{-2} )}\\
&\lesssim \| P_{\le \tilde{N}_n'} F(v) - F(P_{\le \tilde{N}_n'} v) \|_{L^2_t
  L^{\fr{2d}{d+2}}_x (|t| \le T)} \to 0 \text{ as } n \to \infty.
\end{align*}
It follows from Lemma \ref{lma:equivalence_of_norms} that
\begin{align*}
\lim_{n \to \infty} \| H^{\fr{1}{2}} (a')\|_{L^{2}_t
  L^{\fr{2d}{d+2}}_x (|t| \le TN_n^{-2})} = 0.
\end{align*}
To estimate $(a'')$, we use the Leibniz rule, a change of variables,
H\"{o}lder, Sobolev embedding, the bound
\eqref{eqn:ckstt_all_strichartz_bounds}, and dominated convergence to
obtain
\begin{align*}
&\| \nabla(a'') \|_{L^2_t L^{\fr{2d}{d+2}}_x} \lesssim \| |P_{\le \tilde{N}_n N_n} v_n|^{\fr{4}{d-2}}
\nabla P_{\le \tilde{N}_n' N_n} v_n \|_{L^2_t
  L^{\fr{2d}{d+2}}_x(|t| \le TN_n^{-2}, \ |x-x_n| \sim (N_n')^{-1})}
\\
&+ \tfr{N_n'}{N_n}\|P_{\le \tilde{N}_n' N_n }
v_n\|_{L^{\infty}_t L^{\fr{2d}{d-2}}_x }^{\fr{d+2}{d-2}}\\
&\lesssim \| \nabla v\|_{L^{\fr{2(d+2)}{d-2}}_t
  L^{\fr{2d(d+2)}{d^2+4}}_x} \| P_{\le \tilde{N}_n'} v\|^{\fr{4}{d-2}}_{L^{\fr{2(d+2)}{d-2}}_{t, x}
  (|t| \le T, \ |x| \sim \fr{N_n}{N_n'})}  + O(\tfr{N_n'}{N_n})\\
&\lesssim C(E(v)) ( \|P_{> \tilde{N}_n'} v\|_{L^{\fr{2(d+2)}{d-2}}_{t,
    x} } + \| v\|_{L^{\fr{2(d+2)}{d-2}}_{t, x} (|t|\le T, |x|
  \gtrsim \tfrac{N_n}{N_n'})})^{\fr{4}{d-2}} + O(\tfr{N_n'}{N_n})\\
&= o(1) + O(\tfr{N_n'}{N_n}).
\end{align*}
Similarly,
\begin{align*}
&\| |x| (a'')\|_{L^2_t L^{\fr{2d}{d+2}}_x} \sim \| F(P_{\le
  \tilde{N}_n'} v)\|_{L^{\fr{2(d+2)}{d-2}}_t L^{\fr{2d}{d-2}}_x (|t|
  \le T, |x| \sim \tfr{N_n}{N_n'})}\\
&\lesssim ( \|P_{> \tilde{N}_n'} v\|_{L^{\fr{2(d+2)}{d-2}}_t L^{\fr{2d}{d-2}}_x (|t|
  \le T)} + \| v\|_{L^{\fr{2(d+2)}{d-2}}_t L^{\fr{2d}{d-2}}_x
  (|t| \le T, |x| \sim \tfr{N_n}{N_n'})})^{\fr{d+2}{d-2}}\\
&= o(1).
\end{align*}
Therefore 
\[
\lim_{N \to \infty}\| H^{\fr{1}{2}} (a'')\|_{L^2_t L^{\fr{2d}{d+2}}_{t, x}
  (|t| \le TN_n^{-2})} = 0
\] 
as well. This completes the analysis for $(a)$.

The estimates for $(b), \ (c), \ (d)$ are less involved. For $(b)$,
note that on the support of the function we have $\left| |x_n|^2 -
  |x|^2\right| = |x_n - x||x_n + x| \sim N_n (N_n')^{-1}.$ Thus by
H\"{o}lder and Sobolev embedding,
\begin{align*}
&\| \nabla (b)\|_{L^1_t L^2_x(|t| \le TN_n^{-2})} \\
&\lesssim
\tfr{N_n}{N_n'} \| \nabla P_{\le \tilde{N}_n' N_n} v_n\|_{L^1_tL^2_x
  (|t| \le TN_n^{-2})} + N_n \| P_{\le \tilde{N}_n' N_n} v_n
\|_{L^1_t L^2_x (|t| \le TN_n^{-2}, \ |x-x_n| \sim (N_n')^{-1})}\\
&\lesssim (N_n'N_n)^{-1} \| \nabla v_n \|_{L^\infty_t L^2_x} \to 0
\quad \text{as} \quad n \to \infty.
\end{align*}
Using H\"{o}lder and Sobolev embedding, we have
\begin{align*}
\| |x| (b)\|_{L^1_t L^2_x(|t| \le TN_n^{-2})} &\sim \tfr{N_n^2}{N_n'} \| P_{\le
  \tilde{N}_n' N_n} v_n\|_{L^1_t L^2_x (|t| \le TN_n^{-2}, |x -
  x_n| \lesssim (N_n')^{-1})}\\
&\lesssim \left\{\begin{array}{cc} (N_n')^{-2} \| \nabla
    v_n\|_{L^\infty_t L^2_x}, & \lim_{n\to\infty} N_n^{-1} |x_n| = 0\\
\|v_n\|_{L^\infty_t L^2_x} = O(N_n^{-1}), & \lim_{n\to \infty} N_n^{-1} |x_n| > 0,
\end{array}\right.
\end{align*}
which vanishes as $n \to \infty$ in either case. Thus $\| H^{1/2}
(b)\|_{L^1_t L^2_x} \to 0$. The term $(c)$ is dealt with
similarly. For (d), use H\"{o}lder, Bernstein, and the
definition \eqref{eqn:approx_solution_freq_cutoffs} of the frequency cutoffs $\tilde{N}_n'$ to obtain
\begin{align*}
\| \nabla (d)\|_{L^1_t L^2_x (|t| \le TN_n^{-2})} &\lesssim N_n' \|
|\nabla|^2 P_{\le \tilde{N}_n' N_n} v_n\|_{L^1_t L^2_x} + \| |\nabla
P_{\le \tilde{N}_n' N_n} v_n| ( |\nabla|^2 \chi_n) \|_{L^1_t L^2_x}\\
&\lesssim \left[ \left(\tfr{N_n'}{N_n}\right)^{\fr{1}{2}} + \left(
      \tfr{N_n'}{N_n} \right)^2 \right] \| \nabla v_n\|_{L^\infty_t
    L^2_x} \to 0.
\end{align*}
Applying H\"{o}lder in the time variable, we get
\begin{align*}
\| |x|(d) \|_{L^1_t L^2_x (|t| \le TN_n^{-2})} \lesssim
\tfr{N_n'}{N_n} \| \nabla v_n\|_{L^\infty_t L^2_x} \to 0.
\end{align*}
This completes the proof of the lemma.
\end{proof}

Next, we estimate the error over the time intervals $[-2,
TN_n^{-2}]$ and $[TN_n^{-2}, 2]$.

\begin{lma}
\label{lma:long_time_approximation}
$\lim_{T \to \infty} \limsup_{n \to \infty} \| H^{\fr{1}{2}} e_n
\|_{N([-2, TN_n^{-2}] \cup [TN_n^{-2}, 2 ])} = 0$.
\end{lma}

\begin{proof}
We consider just the forward time interval as the other interval is
treated similarly. Since $\tilde{v}_n^T$ solves the linear equation, the error
$e_n$ is just the nonlinear term:
\begin{align*}
e_n = (i\partial_t - H)\tilde{v}_n^T - F(\tilde{v}_n^T) = -F(\tilde{v}_n^T).
\end{align*}
By the chain rule (Corollary~\ref{cor:fractional_chain_rule}) and Strichartz, 
\begin{align*}
\| H^{\fr{1}{2}} e_n \|_{N ([ TN_n^{-2}, 2 ])} \lesssim \|
\tilde{v}_n^T \|_{L^{\fr{2(d+2)}{d-2}}_{t, x}([TN_n^{-2},
  2 ])}^{\fr{4}{d-2}} \| \tilde{v}_n^T (TN_n^{-2}) \|_{\Sigma}.
\end{align*}
By definition $\tilde{v}_n^T(TN_n^{-2}) = e^{-\fr{iTN_n^{-2}
    |x_n|^2}{2}} \tilde{G}_n S_n P_{\le \tilde{N}_n'} v (TN_n^{-2} - t_n)$, so Lemma
\ref{lma:approximation_lemma} implies that
\begin{align*}
\limsup_{n \to \infty} \| \tilde{v}_n^T (TN_n^{-2}) \|_{\Sigma}
&\lesssim \left\{ \begin{array}{cc} \| v|_{L^\infty_t \dot{H}^1_x}, &
  \lim_{n \to \infty} N_n^{-1} |x_n| = 0,\\
\| v\|_{L^\infty_t H^1_x}, & \lim_{n \to \infty} N_n^{-1} |x_n| > 0
\end{array}\right.
\end{align*}
is bounded in either case. Using Strichartz and interpolation, it
suffices to show
\begin{align*}
\lim_{T \to \infty} \limsup_{n \to \infty} \| \tilde{v}_n^T
\|_{L^\infty_T L^{\fr{2d}{d-2}}_x ( [TN_n^{-2}, 2 ])} = 0.
\end{align*}

As we are assuming Conjecture~\ref{conjecture:ckstt}, there exists $v_\infty \in \dot{H}^1$
so that 
\[
\lim_{t \to \infty} \| v(t) - e^{\fr{it\Delta}{2}} v_\infty
\|_{\dot{H}^1_x} = 0.
\]
Then one also has
\begin{align*}
\lim_{t \to\infty} \limsup_{n \to \infty} \| P_{\le \tilde{N}_n'} v(t)
- e^{\fr{it\Delta}{2}} v_\infty \|_{\dot{H}^1_x} = 0,
\end{align*}
and Lemma~\ref{lma:approximation_lemma} implies that 
\begin{align*}
\lim_{T \to \infty} \limsup_{n \to \infty} \| \tilde{v}_n (T
N_n^{-2}) - e^{-\fr{iTN_n^{-2} |x_n|^2}{2}}  G_n S_n (e^{\fr{iT\Delta}{2}} v_\infty) 
\|_{\Sigma} = 0.
\end{align*}
An application of Strichartz and Corollary~\ref{cor:strong_conv_cor}
yields
\begin{align*}
\tilde{v}_n(t) &= e^{-i(t-TN_n^{-2})H} [ \tilde{v}_n(TN_n^{-2})] \\
&=e^{-i(t-TN_n^{-2})H} [e^{-\fr{iTN_n^{-2}|x_n|^2}{2}}  G_n S_n
e^{\fr{i T\Delta}{2}}  v_\infty] + \text{error}\\
&= e^{-itH} [G_n S_n v_\infty] + \text{error}
\end{align*}
where $\lim_{T \to \infty} \limsup_{n \to \infty} \|
\text{error} \|_{\Sigma} = 0$ uniformly in $t$. By Sobolev embedding,
\begin{align*}
  &\lim_{T \to \infty} \limsup_{n \to \infty} \| \tilde{v}_n
  \|_{L^\infty_t L^{\fr{2d}{d-2}}_x ([TN_n^{-2}, 2 ])} \\
  &= \lim_{T\to \infty} \limsup_{n \to \infty} \| e^{-itH} [
  G_n S_n v_\infty] \|_{L^\infty_t L^{\fr{2d}{d-2}}_x
    ([TN_n^{-2}, 2 ])}.
\end{align*}
A standard density argument using the dispersive
estimate for $e^{-itH}$ shows that the last limit is zero. 
\end{proof}
Lemmas \ref{lma:short_time_approximation} and
\ref{lma:long_time_approximation} together establish
\eqref{eqn:approximate_solution_nonlinear_error_estimate}. 

\begin{lma}[Matching initial data]
\label{lma:matching_initial_data}
Let $u_n(0) = e^{it_n H} G_n S_n \phi$ as in
Proposition~\ref{prop:localized_initial_data}.  Then
\begin{align*}
\lim_{T \to \infty} \limsup_{n \to\ \infty} \| \tilde{v}_n^T(-t_n) -
u_n(0) \|_{\Sigma} = 0. 
\end{align*}
\end{lma}

\begin{proof}
  If $t_n \equiv 0$, then by definition $\tilde{v}_n^T(0) = G_n S_n
  P_{\le N_n'} \phi$, so Lemma \ref{lma:approximation_lemma} and the
  definition \eqref{eqn:approx_solution_freq_cutoffs} of the frequency
  parameter
  $N_n'$ imply
\[
\lim_{n \to \infty} \| \tilde{v}_n^T(0) - u_n(0)\|_{\Sigma} \lesssim \lim_{n \to
  \infty} \left\{\begin{array}{cc} \|P_{> N_n'} \phi\|_{H^1}, &
    \lim_{n \to \infty} N_n^{-1} |x_n| > 0\\
\| P_{>N_n'} \phi \|_{\dot{H}^1}, & \lim_{n \to \infty} N_n^{-1} |x_n|
= 0
\end{array}\right\} = 0
\]

Next we consider the case $N_n^2 t_n \to \infty$; the case $N_n^2 t_n
\to -\infty$ works similarly. Arguing as in the previous lemma and
recalling that in this case, the solution $v$ was chosen
to scatter \emph{backward} in time to $e^{\fr{it\Delta}{2}}\phi$, for
$n$ large we have
\begin{align*}
\tilde{v}_n^T(-t_n) = e^{it_n H} [ G_n S_n
\phi] + \text{error}
\end{align*}
where $\lim_{T \to \infty} \limsup_{n \to \infty} \|\text{error}
\|_{\Sigma} \to 0$. The claim follows.
\end{proof}

For each fixed $T > 0$, set
\begin{align}
\tilde{u}_n^T(t) = \tilde{v}_n^T(t - t_n), \label{eqn:defn_of_approximate_solution}
\end{align}
which is defined for $t \in [-1, 1]$. Then for a fixed large value of
$T$, this is an approximate solution for all $n$ sufficiently large in
the sense of Proposition~\ref{prop:stability}. Indeed, by
\eqref{eqn:approx_solution_energy_bound} and
\eqref{eqn:approx_solution_scattering_size_bound}, $\tilde{u}_n^T$
satisfy the hypotheses \eqref{eqn:stability_prop_hyp_1} with $E \lesssim
\|\phi\|_{\dot{H}^1}$ and $L = C( \| \phi\|_{\dot{H}^1})$. Lemmas
\ref{lma:short_time_approximation}, \ref{lma:long_time_approximation},
\ref{lma:matching_initial_data}, Sobolev embedding, and Strichartz
show that for any $\varepsilon > 0$, there exists $T > 0$ so that
$\tilde{u}_n^T$ satisfies the hypotheses
\eqref{eqn:stability_prop_hyp_2} for all large $n$. Invoking
Proposition~\ref{prop:stability}, we obtain the first claim of
Proposition~\ref{prop:localized_initial_data} concerning the existence
of solutions

The remaining assertion of Proposition~\ref{prop:localized_initial_data} regarding approximation by smooth functions will follow
from the next lemma. Recall that we use the notation
\[
\|f\|_{L^q_t \Sigma^r_x} = \| H^{\fr{1}{2}} f\|_{L^q_t L^r_x}.
\]

\begin{lma}
\label{lma:approximation_by_smooth_functions}
  Fix finitely many admissible $(q_k, r_k)$ with
  $2 \le r_k < d$. For every $ \varepsilon > 0$, there
  exists a smooth function $\psi^\varepsilon \in C^\infty_c( \mf{R}
  \times \mf{R}^d)$ such that for all $k$
  \begin{align*}
    \limsup_{T \to \infty} \limsup_{n \to \infty} \|
    \tilde{v}_n^T -  \tilde{G}_n [e^{-\fr{it N_n^{-2}|x_n|^2}{2}} \psi^
    \varepsilon](t+t_n) \|_{L^{q_k}_T \Sigma^{r_k}_x
      ( [ -2, 2 ])} < \varepsilon.
  \end{align*}
\end{lma}

\begin{proof}
We continue using the notation defined at the beginning. Let
\begin{align*}
  \tilde{w}_n^T = \left\{\begin{array}{cc} e^{-\fr{it|x_n|^2}{2}}
      \tilde{G}_n [
      S_n v](t+t_n), & |t| \le TN_n^{-2}\\
      e^{-i(t-TN_n^{-2})H}[ \tilde{w}_n^T(TN_n^{-2}) ], & t \ge TN_n^{-2}\\
      e^{-i(t + TN_n^{-2})H} [  \tilde{w}_n^T(-TN_n^{-2}) ], & t \le
      -TN_n^{-2}
    \end{array}\right.
\end{align*}
This is essentially $\tilde{v}_n^T$ in \eqref{eqn:localized_initial_data_v_n-tilde_defn} without the frequency cutoffs. 
We observe first that $\tilde{v}_n^T$ can be well-approximated by
$\tilde{w}_n^T$ in spacetime:
\begin{equation}
\label{eqn:approximation_by_smooth_functions_eqn_1}
\begin{split}
  \limsup_{n \to \infty} \| \tilde{v}_n^T - \tilde{w}_n^T
  \|_{L^{q_k}_t \Sigma^{r_k}_x ( [ -2, 2 ])} &= 0, \\
  \sup_{T > 0} \limsup_{n \to \infty} \| \tilde{w}_n^T
  \|_{L^{q_k}_t \Sigma^{r_k}_x( [ -2, 2 ])} &<
  \infty. 
\end{split}
\end{equation}
Indeed by dominated convergence,
\begin{align*}
\| \nabla (v - P_{\le \tilde{N}_n'} v) \|_{L^{q_k}_t L^{r_k}_x( \mf{R}
\times \mf{R}^d)} \to 0 \text{ as } n \to \infty,
\end{align*}
thus \eqref{eqn:approximation_by_smooth_functions_eqn_1} follows from Lemma
  \ref{lma:approximation_lemma} and the Strichartz inequality for
  $e^{-itH}$.

The next observation is that most of the spacetime norm of
$\tilde{w}_n^T$ is concentrated in the time interval $|t| \le
TN_n^{-2}$:
\begin{equation}
\label{eqn:approximation_by_smooth_functions_eqn_2}
  \lim_{T \to \infty} \limsup_{n \to \infty} \|
  \tilde{w}_n^T\|_{ L^{q_k}_t \Sigma^{r_k}_x ( [ -2, -TN_n^{-2}] 
\cup [TN_n^{-2},
    2 ])} = 0. 
\end{equation}
  To see this, it suffices by symmetry to consider the forward interval. Recall that $v$
  scatters forward in $\dot{H}^1$ to some $e^{\fr{it\Delta}{2}}v_\infty$. By Lemma~\ref{lma:approximation_lemma},
  \begin{align*}
    \lim_{T \to \infty} \limsup_{n \to \infty} \| (\tilde{G}_n S_n v
    (TN_n^{-2} - t_n)  - G_n S_n (e^{\fr{iT\Delta}{2}} v_\infty) \|_{\Sigma} = 0.
  \end{align*}
  By Strichartz,
  \begin{align*}
    \lim_{T \to \infty} \limsup_{n \to \infty} \| e^{\fr{iTN_n^{-2}|x_n|^2}{2}} \tilde{w}_n^T - e^{-i(t- TN_n^{-2}) H}[
    G_n S_n ( e^{\fr{iT\Delta}{2}} v_\infty)] \|_{L^{q_k}_t 
\Sigma^{r_k}_x ([TN_n^{-2}, 2 ])} = 0
  \end{align*}
  By Corollary~\ref{cor:strong_conv_cor} and Strichartz, for each $T
  > 0$ we have
  \[
\begin{split}
    &\limsup_{n \to \infty} \| e^{-i(t-TN_n^{-2})H} [ G_n S_n
    (e^{\fr{iT\Delta}{2}} v_\infty) ] -
    e^{\fr{iT (r_\infty)^2}{2}} e^{-itH} [ G_n S_n v_\infty]
    \| _{L^{q_k}_t \Sigma^{r_k}_x ( [TN_n^{-2}, 2 ])} = 0.
\end{split}
    \]
  For each $\varepsilon > 0$, choose $v_\infty^{\varepsilon} \in
  C^\infty_c$ such that $\| v_ \infty - v_\infty^\varepsilon
  \|_{\dot{H}^1} < \varepsilon$. By the dispersive estimate,
  \begin{align*}
    \| e^{-itH}[ G_n v_\infty^\varepsilon ] \|_{L^{q_k}_t L^{r_k}_x ( 
[TN_n^{-2}, 2 ])} \lesssim T^{-\fr{1}{q_k}}
    \| v_\infty^\varepsilon \|_{L^{r_k'}_x}
  \end{align*}
Combining the above with Strichartz and
  Lemma~\ref{lma:approximation_lemma}, we get
  \begin{align*}
    \limsup_{n \to \infty} \| \tilde{w}_n^T \|_{L^{q_k}_t \Sigma^{r_k}_x (
      [TN_n^{-2}, 2 ])} \lesssim o(1) + \varepsilon + O_{\varepsilon,
      q_k} ( T^{-\fr{1}{q_k} } ) \text{ as } T \to \infty.
  \end{align*}
  Taking $T \to
  \infty$, we find
  \begin{align*}
    \limsup_{T \to \infty} \limsup_{n \to \infty} \|
    \tilde{w}_n^T \|_{L^{q_k}_t \Sigma^{r_k}_x( [TN_n^{-2}, 2 ])} 
\lesssim
    \varepsilon
  \end{align*}
  for any $\varepsilon > 0$, thereby establishing \eqref{eqn:approximation_by_smooth_functions_eqn_2}.

Choose $\psi^{\varepsilon} \in C^\infty_c (\mf{R} \times
\mf{R}^d)$ such that $\sum_{k=1}^N \| v - \psi^\varepsilon
\|_{L^{q_k}_t \dot{H} _x^{1, r_k} } < \varepsilon$. By combining Lemma
\ref{lma:approximation_lemma} with \eqref{eqn:approximation_by_smooth_functions_eqn_1} and \eqref{eqn:approximation_by_smooth_functions_eqn_2}, we get
\begin{align*}
 \lim_{T \to \infty} \limsup_{n \to \infty} \|
  \tilde{v}_n(t, x) - e^{-\fr{it |x_n|^2}{2}} \tilde{G}_n
  \psi^\varepsilon (t+t_n) \|_{L^{q_k}_t
    \Sigma^{r_k}_x( [ -2, 2 ] )} \lesssim \varepsilon.
\end{align*}
This completes the proof of the lemma, hence 
Proposition~\ref{prop:localized_initial_data}.
\end{proof}

\begin{rmk}
From the proof it is clear that that the proposition also holds if the interval
$I = [-1, 1]$ is replaced by any smaller interval.
\end{rmk}

\section{Palais-Smale and the proof of Theorem~\ref{thm:main_theorem}}
\label{section:palais-smale}

In this section we prove a Palais-Smale condition on sequences
of blowing up solutions to \eqref{eqn:nls_with_quadratic_potential}. 
This will quickly lead to the proof of Theorem~\ref{thm:main_theorem}. 

For a maximal solution $u$ to \eqref{eqn:nls_with_quadratic_potential}, 
define 
\[
S_*(u, L) = \sup \{ S_I(u) : I \text{ is an open interval with } \le L\},
\]
where we set $S_I(u) = \infty$ if $u$ is not 
defined on $I$. All solutions in this section are assumed 
to be maximal. By the triangle inequality, 
finiteness of $S_*(u, L)$ for some $L$ implies finiteness for all $L$.
Set 
\[
\begin{split}
\Lambda_d(E, L) &= \sup \{ S_*(u, L) : u \text{ solves \eqref{eqn:nls_with_quadratic_potential}, } \mu = +1, \ E(u) = E\}\\
\Lambda_f(E, L) &= \sup \{ S_*(u, L) : u \text{ solves
  \eqref{eqn:nls_with_quadratic_potential}, } \mu = -1, \ E(u) = E, \\
&\| \nabla u(0)\|_{L^2} < 
\| \nabla W \|_{L^2}\}.
\end{split}
\]
Note as before that finiteness for some $L$ is 
equivalent to finiteness for all $L$. Finally, define
\[
\begin{split}
\Lambda_d(E) &= \lim_{L \to 0} \Lambda_d(E, L), \quad 
\Lambda_f(E) = \lim_{L \to 0} \Lambda_f(E, L), \\
\mcal{E}_d &= \{ E : \Lambda_d(E) < \infty\}, \quad
\mcal{E}_f 
= \{ E : \Lambda_f(E) < \infty\}.
\end{split}
\]
By the local theory, Theorem~\ref{thm:main_theorem} is equivalent to 
the assertions 
\[
\mcal{E}_d = [0, \infty), \quad \mcal{E}_f = [0, E_{\Delta}(W) ).
\]

Suppose Theorem 
\ref{thm:main_theorem} failed. By the small data theory, 
$\mcal{E}_d, \ \mcal{E}_f$ are nonempty and open, and the failure of 
Theorem~\ref{thm:main_theorem} implies the existence of a critical 
energy $E_c > 0$, with $E_c < E_{\Delta}(W)$ in the focusing case such 
that 
$\Lambda_d (E), \ \Lambda_f(E) = \infty$ for $E > E_c$ and $\Lambda_d
(E), \ \Lambda_f(E)  < \infty$ for all $E < E_c$.

Define the spaces
\begin{align*}
\dot{X}^1 = \left\{ \begin{array}{cc} L^{10}_{t, x} \cap
    L^{5}_t \Sigma^{\fr{30}{11}}_x ([-\tfr{1}{2}, \tfr{1}{2}]
    \times \mf{R}^d), & d = 3\\
L^{\fr{2(d+2)}{d-2}}_{t, x} \cap L^{\fr{2(d+2)}{d}}_t 
\Sigma^{\fr{2(d+2)}{d}}_x ( [-\tfr{1}{2}, \tfr{1}{2}] \times
\mf{R}^d), & d \ge 4. \end{array}\right.
\end{align*}
When $d = 3$, we also define
\begin{align*}
\dot{Y}^1 = \dot{X}^1 \cap L^{\fr{10}{3}}_t \Sigma^{\fr{10}{3}}_x (
[-\tfr{1}{2}, \tfr{1}{2}] \times \mf{R}^3).
\end{align*}

\begin{prop}[Palais-Smale]
\label{prop:palais-smale}
Assume Conjecture \ref{conjecture:ckstt} holds. Suppose that $u_n :
(t_n - \tfr{1}{2}, t_n + \tfr{1}{2}) \times \mf{R}^d \to \mf{C}$ is a
sequence of solutions with
\[
\lim_{n \to \infty} E(u_n) =E_c, \ \lim_{n \to \infty} S_{(t_n-\fr{1}{2}, t_n]} (u_n) = 
\lim_{n \to \infty} S_{ [t_n, t_n+\fr{1}{2})} (u_n) = \infty.
\]
In the focusing case, assume also that $E_c < E_{\Delta}(W)$ and $\|
\nabla u_n (t_n)\|_{L^2} < \| \nabla W\|_{L^2}$. Then there exists a 
subsequence such that $u_n(t_n)$ converges 
in $
\Sigma$.
\end{prop}

\begin{proof}[Proof of Proposition~\ref{prop:palais-smale}]

  By replacing $u_n(t)$ with $u_n(t + t_n)$, we may
  assume $t_n \equiv 0$. Note that by energy conservation and
  Corollary \ref{cor:appendix_cor_1}, this time translation does not
  change the hypotheses of the focusing case.

  Observe (referring to the discussion in Section
  \ref{section:focusing_blowup} for the focusing case) that the
  sequence $u_n(0)$ is bounded in $\Sigma$. Applying
  Proposition~\ref{prop:lpd}, after passing
  to a subsequence we have a decomposition
\begin{align*}
u_n(0) = \sum_{j=1}^J e^{it_n^j H} G_n S_n \phi^j + w^J_n = \sum_{j = 1}^J \phi^j_n + w^J_n
\end{align*}
with the properties stated in that proposition. In particular, the
remainder has asymptotically trivial linear evolution:
\begin{equation}
\label{eqn:palais_smale_lpd_remainder_decay}
\lim_{J \to J^*} \limsup_{n \to \infty} \| e^{-itH} w_n^J \|_{L^{\fr{2(d+2)}{d-2}}_{t,x}},
\end{equation}
and we have asymptotic decoupling of energy:
\begin{equation}
\label{eqn:palais-smale_energy_decoupling}
\sup_{J} \lim_{n \to \infty} | E(u_n) - \sum_{j=1}^J E(\phi^j_n) - E(w^J_n) | = 
0 .
\end{equation}
Observe that $\liminf_{n} E(\phi^j_n) \ge 0$. This is obvious in the
defocusing case. In the focusing case,
\eqref{eqn:lpd_sigma_decoupling} and the discussion in Section
\ref{section:focusing_blowup} imply that
\[
\sup_j \limsup_n \| \phi^j_n\|_{\Sigma} \le\|u_n\|_{\Sigma} < \| \nabla W\|_{L^2},
\]
so the claim follows from Lemma \ref{lma:appendix_lma_1}.
Therefore, there are two possibilities.

\textbf{Case 1}: $\sup_{j} \limsup_{n \to \infty} E(\phi^j_n) = E_c$.

By combining \eqref{eqn:palais-smale_energy_decoupling} with 
the fact that 
the 
profiles $\phi^j_n$ are nontrivial in $\Sigma$, we deduce that $J^* = 1$ and
\[
u_n(0) = e^{i t_n H} G_n S_n \phi + w_n, \ \lim_{n \to \infty} \| 
w_n\|_{\Sigma} = 0.
\]
We will show that $N_n \equiv 1$ (thus $x_n = 0$ and $t_n =
0$). Suppose $N_n \to \infty$. 

 Proposition~\ref{prop:localized_initial_data}
implies that for all large $n$, there 
exists a unique 
solution $u_n$ on $[-\fr{1}{2}, \fr{1}{2}]$ with 
$u_n(0) = e^{it_n H} G_n S_n \phi$ and 
$\limsup_{n \to \infty} 
S_{(-\fr{1}{2}, \fr{1}{2})} (u_n) \le C(E_c)$. 
By perturbation theory (Proposition
\ref{prop:stability}), 
\[
\limsup_{n \to \infty} S_{[-\fr{1}{2}, \fr{1}{2}]} ( u_n) \le C(E_c),
\]
which is a contradiction. 
Therefore, $N_n \equiv 1, \ t^j_n \equiv 0, \ x_n^j 
\equiv 0$, and 
\[
u_n(0) = \phi + w_n
\]
for some $\phi \in \Sigma$. This is the desired conclusion.

\textbf{Case 2}: $\sup_{j} \limsup_{n \to \infty} E(\phi^j_n) \le E_c - 
2\delta$ for some $\delta > 0$. 

By the definition of $E_c$, there exist solutions 
$v_n^j: (-\tfr{1}{2}, \tfr{1}{2}) \times \mf{R}^d \to \mf{C}$ 
with 
\[
\| v^j_n\|_{L^{\fr{2(d+2)}{d-2}}_{t,x} ([-\fr{1}{2}, \fr{1}{2}])}
\lesssim_{E_c, \delta} E(\phi^j_n)^{\fr{1}{2}}.
\]
By standard arguments (c.f. \cite[Lemma 3.11]{matador}), this
implies the seemingly stronger bound
\begin{equation}
\| v^j_n\|_{\dot{X}^1} \lesssim_{E_c, \delta} E(\phi^j_n)^{\fr{1}{2}}. 
\label{eqn:palais-smale_nonlinear_profile_scattering_size}
\end{equation}
In the case $d = 3$, we also have $\| v^j_n \|_{\dot{Y}^1} \lesssim 
E(\phi^j_n)^{\fr{1}{2}}$. Put
\begin{equation}
u_n^J = \sum_{j=1}^J v_n^j + e^{-it H} w_n^J. \label{eqn:palais-smale_approx_solution}
\end{equation}
We claim that for sufficiently large $J$ and $n$, $u_n^J$ is an approximate 
solution in the sense 
of Proposition~\ref{prop:stability}. To prove this claim, we check that $u_n^J$ has the 
following three 
properties:
\begin{itemize}
\item [(i)] $\lim_{J \to J^*} \limsup_{n \to \infty} \| u_n^J(0) - u_n(0) 
\|_{\Sigma} = 0$.
\item [(ii)] $\limsup_{n \to \infty} \| u_n^J\|_{L^{\fr{2(d+2)}{d-2}}_{t, x}( 
[-T, T ])} \lesssim_{E_c, 
\delta} 1$ 
uniformly in $J$.
\item [(iii)] $\lim_{J \to J^*} \limsup_{n \to \infty} \| H^{\fr{1}{2}} e^J_n \|_{N( 
[-\fr{1}{2}, \fr{1}{2}])} = 0$, where 
\[
e_n = (i\partial_t - H) u_n^J - F(u_n^J).
\]
\end{itemize}

There is nothing to check for part (i) as $u^J_n(0) = u_n(0)$ by construction. 
The verification of 
(ii) 
relies on the asymptotic decoupling of the nonlinear profiles $v_n^j$, which we 
record in the 
following two lemmas.

\begin{lma}[Orthogonality]
\label{lma:orthogonality_of_parameters}
Suppose that two frames $\mcal{F}^j = (t^j_n, x_n^j, N_n^j), \
\mcal{F}^k = (t^k, x_n^k , N_n^k)$ are orthogonal, and let
$\tilde{G}_n^j, \ \tilde{G}_n^k$ be the associated spacetime scaling
and translation operators as defined in
\eqref{eqn:inv_strichartz_frameops_1}. Then for all
$\psi^j, \ \psi^k$ in $C^\infty_c (\mf{R} \times \mf{R}^d)$,
\begin{align*}
&\| (\tilde{G}^j_n \psi^j) (\tilde{G}^k_n \psi^k) \|_{L^{\fr{d+2}{d-2}}_{t, x}} + \| (\tilde{G}_n^j 
\psi^j) \nabla (\tilde{G}^k_n \psi^k) \|
_{L^{\fr{d+2}
{d-1}}_{t, x} } + \| |x| (\tilde{G}^j_n \psi^j) (\tilde{G}^k_n \psi^k) 
\|_{L^{\fr{d+2}{d-1}}_{t, x}} \\
&+ \| |x|^2 (\tilde{G}^j_n \psi^j) (\tilde{G}^k_n \psi^k) \|_{L^{\fr{d+2}{d}}_{t,x}} + \| ( 
\nabla \tilde{G}^j_n \psi^j) 
( \nabla \tilde{G}^k_n \psi^k) \|_{L^{\fr{d+2}{d}}_{t,x} } \to 0
\end{align*}
as $n \to \infty$.
When $d = 3$, we also have
\begin{align*}
\| |x|^2 (\tilde{G}^j_n \psi^j) (\tilde{G}^k_n \psi^k) \|_{L^5_t L^{\fr{15}{11}}_x} + \| 
(\nabla \tilde{G}^j_n \psi^j) (\nabla 
\tilde{G}^k_n \psi^k)\|_{L^5_t L^{\fr{15}{11}}_x} \to 0.
\end{align*}
\end{lma}

\begin{proof}
The arguments for each term are similar, and we only supply the details
for the second term.
Suppose $N_n^k (N_n^j)^{-1} \to \infty$. By the chain rule, a change of variables, and H\"{o}lder,
\[
\begin{split}
\| (\tilde{G}^j_n \psi^j) \nabla (\tilde{G}^k_n \psi^k) \|_{L^{\fr{d+2}{d-1}}_{t, x}} 
= \| \psi^j \nabla (\tilde{G}_n^j)^{-1} \tilde{G}_n^k \psi^k \|_{L^{\fr{d+2}{d-1}}_{t, x}}  \le \| \psi^j \chi_n \|_{L^{\fr{2(d+2}{d-2}}_{t, x}} \| \nabla \psi^k\|_{L^{\fr{2(d+2)}{d}}_{t, x}},
\end{split}
\]
where $\chi_n$ is the characteristic function of the support of
$\nabla (\tilde{G}_n^j)^{-1} \tilde{G}_n^k \psi^k$. As the support of
$\chi_n$ has measure shrinking to zero, we have
\[
\lim_{n \to \infty} \| \psi^j \chi_n\|_{L^{\fr{2(d+2)}{d-2}}_{t,x}} = 0.
\]
A similar argument 
deals with the case where $N_n^j (N_n^k)^{-1} \to \infty$. Therefore, we may suppose that 
\[
\tfr{N_n^k}{N_n^j} \to N_\infty \in (0, \infty).
\]
Make the same change 
of variables as before, and compute
\[
\nabla (\tilde{G}_n^j)^{-1} \tilde{G}_n^k \psi^k (t, x) = (\tfr{N_n^k}{N_n^j})^{\fr{d}{2}} (\nabla \psi^k)
[ \tfr{N_n^k}{N_n^j} t + (N_n^k)^2 (t_n^j - t_n^k), \tfr{N_n^k}{N_n^j} x + N_n^k( x_n^j - x_n^k) ].
\]
The decoupling statement \eqref{eqn:lpd_orthogonality_of_frames} implies that 
\[
(N_n^k)^2 (t_n^j - t_n^k) + N_n^k|x_n^j - x_n^k| \to \infty.
\]
Therefore, the supports of 
$\psi^j$ and $\nabla (\tilde{G}_n^j)^{-1} \tilde{G}_n^k \psi^k$ are
disjoint for large $n$.
\end{proof}

\begin{lma}[Decoupling of nonlinear profiles]
\label{lma:decoupling_of_nonlinear_profiles}
Let $v^j_n$ be the nonlinear solutions defined above. Then when $d \ge 4$,
\begin{align*}
\| v^j_n v^k_n\|_{L^{\fr{2(d+2)}{d-2}}_{t, x}} &+ \| v^j_n \nabla v^k_n 
\|_{L^{\fr{d+2}{d-1}}_{t, x}} + \| |
x| v^j_n v^k_n \|
_{L^{\fr{d+2}{d-1}}_{t, x}} \\
&+ \| ( \nabla v^j_n) (\nabla v^k_n) \|_{L^{\fr{2(d+2)}{d}}_{t,x}} + \| |x|^2 
v^j_n v^k_n 
\|_{L^{ \fr{2(d+2)}{d} }_{t,x} } \to 0
\end{align*}
as $n \to \infty$. When $d = 3$, the same statement holds with the last two 
expressions replaced 
by
\begin{align*}
\| (\nabla v^j_n) (\nabla v^k_n) \|_{L^5_t L^{\fr{15}{11}}_x} + \| |x|^2 v^j_n 
v^k_n \|_{L^5_t L^{\fr{30}
{11}}_x} \to 0.
\end{align*}
\end{lma}

\begin{proof}
  We spell out the details for the $\| v^j_n |x| v^k_n \|_{L^{\fr{d+2}{d-1}}_{t, x}}$
  term. Consider first the case $d \ge 4$.  As $2 < \tfr{2(d+2)}{d} <
  d$, by Proposition~\ref{prop:localized_initial_data} we can approximate $v^j_n$ in $\dot{X}^1$ by test functions
 \[
c^j_n \tilde{G}_n \psi^j, \quad \psi^j \in C^ \infty_c( \mf{R} \times
  \mf{R}^d), \quad c_n^j(t) = e^{-\fr{i(t-t_n^j) |x_n^j|^2}{2}}.
\]
By H\"{o}lder and a change of variables,
\begin{align*}
&\| v^j_n |x| v^k_n\|_{L^{\fr{d+2}{d-1}}_{t, x}} \le \| (v^j_n - c^j_n
\tilde{G}^j_n \psi^j)|x|v^k_n \|_{L^{\fr{d+2}{d-1}}_{t,x}} \\
&+ \| |x| \tilde{G}^j_n \psi^j
(v^k_n - c^k_n \tilde{G}^k_n\psi^k) \|_{L^{\fr{d+2}{d-1}}_{t,x}} 
+ \| |x| \tilde{G}^j_n
\psi^j \tilde{G}^k_n \psi^k \|_{L^{\fr{d+2}{d-1}}_{t,x}}\\
&\le \|( v^j_n - c_n^j \tilde{G}^j_n \psi^j)\|_{L^{\fr{2(d+2)}{d-2}}_{t, x} } \| 
 v^k_n\|_{\dot{X}^1 }  \\
&+ \| \psi^j \|_{L^{\fr{2(d+2)}{d-2}}_{t,x} } \| ( v^k_n - c^k_n 
\tilde{G}^k_n 
\psi^k) \|_{ \dot{X}^1 } + \| (\tilde{G}^j_n \psi^j) |x| (\tilde{G}^k_n \psi^k) 
\|_{L^{\fr{d+2}{d-1}}_{t, x}}
\end{align*}
By first choosing $\psi^j$, then $\psi^k$, then invoking the previous lemma, one 
obtains for any $\varepsilon > 0$ that
\begin{align*}
\limsup_{n \to \infty} \| v^j_n |x| v^k_n \|_{L^{\fr{d+2}{d-1}}_{t, x}} \le 
\varepsilon.
\end{align*}
When $d = 3$, we also approximate $v^j_n$ in $\dot{X}^1$ (which is
possible because the exponent $ \fr{30}{11}$ in the definition of
$\dot{X}^1$ is less than $3$), and estimate
\begin{align*}
&\| v^j_n |x| v^k_n \|_{L^{\fr{5}{2}}_{t,x}} \\
&\le \| (v^j_n - c^j_n
\tilde{G}^j_n \psi^j)|x|v^k_n \|_{L^{\fr{5}{2}}_{t,x}} + \| |x| \tilde{G}^j_n \psi^j
(v^k_n - c^k_n \tilde{G}^k_n\psi^k) \|_{L^{\fr{5}{2}}_{t,x}} + \| |x| \tilde{G}^j_n
\psi^j \tilde{G}^k_n \psi^k \|_{L^{\fr{5}{2}}_{t,x}}\\
&\le \|( v^j_n - c_n^j \tilde{G}^j_n \psi^j)\|_{L^{10}_{t, x} } \| v^k_n\|_{\dot{Y}^1 }\\
&+ \| \psi^j \|_{L^5_t L^{30}_x} \| v^k_n - c^k_n \tilde{G}^k_n \psi^k \|_{\dot{X}^1 } + \| (\tilde{G}^j_n \psi^j) |x| (\tilde{G}^k_n \psi^k) 
\|_{L^{\fr{5}{2}}_{t, x}}
\end{align*}
which, just as above, can be made arbitrarily small as $n \to \infty$. Similar 
approximation arguments deal with the other terms.
\end{proof}

Let us verify Claim (ii) above. In fact we will show that 
\begin{equation}
\label{eqn:palais_smale_stability_property_ii_enhanced}
\limsup_{n \to \infty} \| u_n^J \|_{\dot{X}^1([-\fr{1}{2}, \fr{1}{2}])} \lesssim_{E_c, \delta} 1 
\text{ uniformly in } J.
\end{equation}
First, we have
\begin{align*}
S(u_n^J) = \iint | \sum_{j=1}^J v^j_n  + e^{-itH} w_n^J |^{\fr{2(d+2)}{d-2}} \, dx 
dt \lesssim 
S( \sum_{j=1}^J 
v^j_n ) 
+ S( e^{-itH} w^J_n).
\end{align*}
By the properties of the LPD, $\lim_{J \to J^*} \limsup_{n \to \infty} 
S(e^{-itH} w^J_n) = 0$. 
Recalling \eqref{eqn:palais-smale_nonlinear_profile_scattering_size}, we have
\begin{align*}
S(\sum_{j=1}^J v^j_n ) &= \Bigl \| ( \sum_{j=1}^J v^j_n )^2  
\Bigr \|_{L^{\fr{d+2}{d-2}}_{t, x}}^{\fr{d+2}
{d-2}} \le ( \sum_{j=1}
^J \| v^j_n \|_{L^{\fr{2(d+2)}{d-2}}_{t, x}}^2 + \sum_{j \ne k} \| v^j_n v^k_n 
\|_{L^{\fr{d+2}{d-2}}_{t, 
x}} )^{ \fr{d+2}{d-2}}\\
&\lesssim ( \sum_{j=1}^J E( \phi^j_n) + o_J(1) )^{\fr{d+2}{d-2}}
\end{align*}
where for the last line we invoked Lemma~\ref{lma:decoupling_of_nonlinear_profiles}. Since 
energy 
decoupling implies $\limsup_{n \to \infty} \sum_{j=1}^J E(\phi^j_n) \le E_c$, 
we obtain $\lim_{J \to J^*} \limsup_{n \to \infty}  S(u_n^J) \lesssim_{E_c, 
\delta} 1$. 

Mimicking this argument lets us show that 
\[
\limsup_{n \to \infty} ( \| \nabla u_n^J \|_{L^{\fr{2(d+2)}{d}}_{t,x} } + \| |x| 
u_n^J\|_{L^{\fr{2(d+2)}{d}}
_{t,x}}) \lesssim_{E_c, \delta} 1 
\text{ uniformly in } J.
\]
This completes the verification of 
property (ii) in the case $d \ge 4$. The case $d = 3$ is dealt with in a similar 
fashion.

\begin{rmk}
The above argument shows that for each $J$ and each $\eta > 0$, there exists $J' 
\le J$ such 
that 
\[
\limsup_{n \to \infty} \| \sum_{j= J'}^J v^j_n \|_{\dot{X}^1 ([-\fr{1}{2}, \fr{1}{2}])} \le 
\eta.
\]
\end{rmk}

It remains to check property (iii) above, namely, that
\begin{equation}
\label{eqn:palais_smale_stability_iii}
\lim_{J \to J^*} \limsup_{n \to \infty} \| H^{1/2} e^J_n \|_{N( [-\fr{1}{2}, \fr{1}{2}])} = 0.
\end{equation}
Writing $F(z) = |z|^{\fr{4}{d-2}} z$, we decompose
\begin{align*}
e^J_n = [ \sum_{j=1}^J F(v^j_n) - F( \sum_{j=1}^J v^j_n ) ] + [ F(u_n^J - 
e^{-itH} w^J_n) - 
F(u^J_n)] = (a) + (b).
\end{align*}
Consider (a) first. Suppose $d \ge 6$. Using the chain rule $\nabla F( u) = 
F_z(u) \nabla u + 
F_{\overline{z}} (u) 
\overline{ \nabla u}$ and the estimates
\begin{align*}
|F_{z}(z)| + |F_{\overline{z}} (z)| = O(|z|^{\fr{4}{d-2}}), \ |F_{z}(z) - 
F_{z}(w)| + |F_{\overline{z}}(z) - 
F_{\overline{z}}(w)| = O( |z-w|^{\fr{4}{d-2}}),
\end{align*}
we compute
\begin{align*}
|\nabla (a)| \lesssim \sum_{j=1}^J \sum_{k \ne j} |v^k_n|^{\fr{4}{d-2}} |\nabla 
v^j_n| .
\end{align*}
By H\"{o}lder, Lemma~\ref{lma:decoupling_of_nonlinear_profiles}, and the 
induction hypothesis 
\eqref{eqn:palais-smale_nonlinear_profile_scattering_size},
\begin{align*}
\| \nabla (a)\|_{L^{\fr{2(d+2)}{d+4}}_{t, x}} \lesssim \sum_{j=1}
^J \sum_{k \ne j} \|  |v^k_n| | \nabla v^j_n| \|_{L^{\fr{d+2}{d-1}}_{t, 
x}}^{\fr{4}{d-2}} \| \nabla v^k_n \|
_{L^{\fr{2(d+2)}{d}}_{t,x}}^{\fr{d-6}{d-2}} = o_J(1)
\end{align*}
as $n \to \infty$. When $3 \le d \le 5$, we have instead
\begin{align*}
|\nabla(a)| \lesssim \sum_{j=1}^J \sum_{k \ne j} |v^k_n| | \nabla
v^j_n| O( \Bigl | 
\sum_{k=1}^J v^k_n \Bigr |
^{\fr{6-d}{d-2}} + |v^j_n|^{\fr{6-d}{d-2}}),
\end{align*}
thus
\begin{align*}
\| \nabla(a)\|_{L^{\fr{2(d+2)}{d+4}}_{t,x}} &\lesssim_J  \left( \sum_{j=1}^J \| 
v^j_n\|_{L^{\fr{2(d+2)}
{d-2}}_{t,x}}^{\fr{6-d}{d-2}} \right) \sum_{j=1}^J \sum_{k \ne j} \| |v^k_n| | 
\nabla v^j_n| \|_{L^{\fr{d
+2}{d-1}}_{t,x}} = o_J(1).
\end{align*}
Similarly, writing
\begin{align*}
|(a)| \le \sum_{j=1}^J \Bigl | |v^j_n|^{\fr{4}{d-2}} - | \sum_{k=1}^J v^k_n 
|^{\fr{4}{d-2}} \Bigr | |v^j_n| 
\lesssim \sum_{j=1}^J 
\sum_{k\ne j} |v^j_n| |v^k_n|^{\fr{4}{d-2}},
\end{align*}
we have
\begin{align*}
\| x (a) \|_{L^{\fr{2(d+2)}{d+4}}_{t, x}}  \sum_{j=1}^J 
\sum_{k \ne j} \| |x| v^j_n\|_{L^{\fr{2(d+2)}{d}}_{t,x}}^{\fr{d-6}{d-2}} \| |x| 
v^j_n v^k_n \|_{L^{\fr{d+2}
{d-1}}_{t, x}}^{\fr{4}{d-2}} = o_J(1).
\end{align*}
When $3 \le d \le 5$, 
\begin{align*}
|(a)| \lesssim \sum_{j=1}^J \sum_{k \ne j} |v^j_n |v^k_n| O( \Bigl |\sum_{k=1}^J 
v^k_n \Bigr |^{\fr{6-d}{d-2}} + |v^j_n|^{\fr{6-d}{d-2}}),
\end{align*}
hence also
\begin{align*}
\| |x|(a)\|_{L^{ \fr{2(d+2)}{d+4}}_{t,x}} = o_J(1).
\end{align*}
Summing up, 
\begin{align*}
\| H^{1/2} (a)\|_{L^{\fr{2(d+2)}{d+4}}_{t, x}} \lesssim \| \nabla (a)
\|_{L^{\fr{2(d+2)}{d+4}}_{t, x}} + \| x (a)\|_{L^{\fr{2(d+2)}{d+4}}_{t, x}} = o_J(1).
\end{align*}
We now estimate (b), restricting temporarily to dimensions $d \ge 4$. When $d 
\ge 6$, write
\[
\begin{split}
(b) &= F(u^J_n - e^{-it H} w^J_n) - F(u^J_n) \\
&= ( |u^J_n - e^{-it H} w^J_n 
|^{\fr{4}{d-2}} - |u^J_n|
^{\fr{4}{d-2}} ) \sum_{j=1}^J   v^j_n - (e^{-itH} w^J_n) |u^J_n|^{\fr{4}{d-2}}\\
&= O(|e^{-it H} w^J_n|^{\fr{4}{d-2}}) \sum_{j=1}^J v^j_n - (e^{-itH} w^J_n) | 
u^J_n |^{\fr{4}{d-2}},
\end{split}
\]
and apply H\"{o}lder's inequality:
\begin{equation}
\label{eqn:palais_smale_stability_iii_eq1}
\begin{split}
\| |x| (b) \|_{L^{\fr{2(d+2)}{d+4}}_{t,x}} &\lesssim \| e^{-itH} w_n^J 
\|_{L^{\fr{2(d+2)}{d-2}}_{t,x}}
^{\fr{4}{d-2}} \| \sum_{j=1}^J |x| v^j_n \|_{L^{\fr{2(d+2)}{d}}_{t,x}}
\\
&+ \| 
|x| 
u^J_n\|_{L^{\fr{2(d+2)}{d}}
_{t,x}}^{\fr{4}{d-2}} \| |x| e^{-itH} 
w^J_n\|_{L^{\fr{2(d+2)}{d}}_{t,x}}^{\fr{d-6}{d-2}} \| e^{-itH} w^J_n\|
_{L^{\fr{2(d+2)}{d-2}}_{t,x}}^{\fr{4}{d-2}}
\end{split}
\end{equation}
When $d = 4, 5$, 
\begin{align*}
(b) = (e^{-itH} w^J_n) O( |u^J_n|^{ \fr{6-d}{d-2}} + |u^J_n- e^{-itH} 
w^J_n|^{\fr{6-d}{d-2}}) 
\sum_{j=1}^J v^j_n - (e^{-itH} w^J_n) |u^J_n|^{\fr{4}{d-2}},
\end{align*}
thus
\[
\begin{split}
&\| |x|(b) \|_{L^{\fr{2(d+2)}{d+4}}_{t,x}} \\
&\lesssim \| e^{-itH} 
w^J_n\|_{L^{\fr{2(d+2)}{d-2}}_{t,x}} \| |x| 
\sum_{j=1}^J v^j_n \|_{L^{\fr{2(d+2)}{d}}_{t,x}} ( \| 
u^J_n\|_{L^{\fr{2(d+2)}{d-2}}_{t,x}}^{\fr{6-d}{d-2}} 
+ \| e^{-itH} w^J_n\|_{L^{\fr{2(d+2)}{d-2}}_{t,x}}^{\fr{6-d}{d-2}})\\
&+ \| e^{-itH} w^J_n \|_{L^{\fr{2(d+2)}{d-2}}_{t,x}} \| xu^J_n
\|_{L^{\fr{2(d+2)}{d}}_{t,x}} \| u_n^J\|_{L^{\fr{2(d+2)}{d-2}}_{t,x}}^{\fr{6-d}{d-2}}.
\end{split}
\]
Using \eqref{eqn:palais_smale_stability_property_ii_enhanced}, Strichartz,
 and the decay property \eqref{eqn:palais_smale_lpd_remainder_decay},
we get
\[
\lim_{J \to J*} \limsup_{n \to \infty} \| |x| (b) 
\|_{L^{\fr{2(d+2)}{d+4}}_{t,x}}= 0.
\]

It remains to bound $\nabla (b)$. By the chain rule,
\begin{align*}
\nabla(b) &\lesssim |  e^{-itH} w^J_n|^{\fr{4}{d-2}} | \Bigl | \sum_{j=1}^J  
\nabla v^j_n \Bigr |  +| 
u_n^J|^{\fr{4}{d-2}} |\nabla e^{-itH} w^J_n|\\
&= (b') + (b'').
\end{align*}
The first term $(b')$ can be handled in the manner of
\eqref{eqn:palais_smale_stability_iii_eq1} above. We now concern
ourselves with (b''). Fix a small parameter $\eta > 0$, and use the
above remark to obtain $J' = J'(\eta) \le J$ such that
\begin{align*}
\| \sum_{j = J'}^J v^j_n \|_{\dot{X}^1} \le \eta.
\end{align*}
By the subadditivity of $z \mapsto |z|^{\fr{4}{d-2}}$ (which is true up to a 
constant when $d = 4, 5$) and H\"{o}lder,
\begin{align*}
&\| (b'')\|_{L^{\fr{2(d+2)}{d+4}}_{t,x}} = \| | \sum_{j=1}^J v^j_n + e^{-itH} 
w^J_n|^{\fr{4}{d-2}} |\nabla 
e^{-itH} w^J_n| \|_{L^{\fr{2(d+2)}{d+4}}_{t,x}} \\
&\lesssim \| e^{-itH} w^J_n \|_{L^{\fr{2(d+2)}{d-2}}_{t,x}}^{\fr{4}{d-2}} \| 
H^{1/2} e^{-itH} w^J_n \|
_{L^{\fr{2(d+2)}{d}}_{t,x}} + \| \sum_{j = J'}^{J} v^j_n 
\|_{L^{\fr{2(d+2)}{d-2}}_{t,x}}^{\fr{4}{d-2}} \| 
H^{1/2} e^{-itH} w^J_n \|_{L^{\fr{2(d+2)}{d}}_{t,x}} \\
&+ C_{J'}\sum_{j=1}^{J'-1} \| \nabla e^{-itH} 
w^J_n\|_{L^{\fr{2(d+2)}{d}}_{t,x}}^{\fr{d-6}{d-2}} \| |v^j_n| |
\nabla e^{-itH} w^J_n \|_{L^{\fr{d+2}{d-1}}_{t,x}}^{\fr{4}{d-2}}.
\end{align*}
By Strichartz and the decay of $e^{-itH} w^J_n$ in $L^{\fr{2(d+2)}{d-2}}_{t, 
x}$, the first term goes 
to $0$ as 
$J 
\to \infty, \ n \to \infty$. By Strichartz and the definition of $J'$, the 
second term is bounded by
\begin{align*}
\eta^{\fr{4}{d-2}} \| w^J_n \|_{\Sigma}
\end{align*}
which can be made arbitrarily small since $\limsup_{n \to \infty} \| w^J_n 
\|_{\Sigma}$ is bounded 
uniformly in $J$. To finish, we show that for each fixed $j$
\begin{align}
\lim_{J \to J^*} \limsup_{n \to \infty} \| |v^j_n|  \nabla e^{-itH} 
w^J_n\|_{L^{\fr{d+2}{d-1}}_{t,x}} = 0. 
\label{eqn:palais-smale_remainder_derivative_estimate}
\end{align}
For any $\varepsilon > 0$, there exist $\psi^j \in C^\infty_c(\mf{R} \times 
\mf{R}^d)$ such that if
\[
c_n^j = e^{-\fr{i(t-t_n^j) |x_n^j|^2}{2}}
\]
then
\begin{align*}
\limsup_{n\to \infty} \| v^j_n - c_n^j \tilde{G}^j_n \psi^j \|_{L^{\fr{2(d+2)}{d-2}}_{t, x}(
  [-\fr{1}{2}, \fr{1}{2}] )} < 
\varepsilon,
\end{align*}
Note that $\tilde{G}^j_n \psi^j$ is supported on the set 
\[
\{ |t - t^j_n| \lesssim (N_n^j)^{-2}, \ |x-x^j_n| \lesssim 
(N_n^j)^{-1}\}.
\]
Thus for all $n$ sufficiently large,
\begin{align*}
&\| v^j_n \nabla e^{-itH} w^J_n\|_{L^{\fr{d+2}{d-1}}_{t,x}} \\
&\le \| v^j_n - c_n^j\tilde{G}^j_n \psi^j\|
_{L^{\fr{2(d+2)}{d-2}}_{t, x}} \| \nabla e^{-itH} w^J_n 
\|_{L^{\fr{2(d+2)}{d}}_{t,x} } + \| \tilde{G}^j_n \psi^j 
\nabla e^{-itH} 
w^J_n \|_{L^{\fr{d+2}{d-1}}_{t,x}}\\
&\lesssim_{E_c} \varepsilon + \| (\tilde{G}^j_n \psi^j ) \nabla e^{-itH} w_n^J\|_{L^{ 
\fr{d+2}{d-1}}_{t,x}}.
\end{align*}
By H\"{o}lder, noting that $\fr{d+2}{d-1} \le 2$ whenever $d \ge 4$,
\begin{align*}
 \| (\tilde{G}^j_n \psi^j ) \nabla e^{-itH} w_n^J\|_{L^{\fr{d+2}{d-1}}_{t,x}} 
&\lesssim_{\varepsilon} 
(N_n^j)^{\fr{d-2}{2}} \| \nabla e^{-itH} w_n^J \|_{L^{\fr{d+2}{d-1}}_{t,x} ( |t - 
t^j_n| \lesssim 
(N_n^j)^{-2}, |x-x^j_n| \lesssim (N_n^j)^{-1})}\\
&\lesssim N_n^j \| \nabla e^{-itH} w_n^J \|_{L^2_{t,x}( |t-t_n^j| \lesssim 
(N_n^j)^{-2}, |x-x_n^j| 
\lesssim (N_n^j)^{-1})} 
\end{align*}
Since $(N_n^j)^{-2} |x^j_n| = O( (N_n^j)^{-1})$, Corollary~\ref{cor:local_smoothing_cor} implies 
\begin{align*}
\| v^j_n \nabla e^{-itH} w^J_n\|_{L^{\fr{d+2}{d-1}}_{t,x}} \lesssim \varepsilon 
+ C_{\varepsilon, E_c} 
\| e^{-itH} w^J_n \|_{L^{\fr{2(d+2)}{d-2}}_{t,x}}^{\fr{1}{3}}.
\end{align*}
Sending $n \to \infty$, then $J \to J^*$, then $\varepsilon \to 0$ establishes 
\eqref{eqn:palais-smale_remainder_derivative_estimate}, and with it, Property 
(iii).

When $d = 3$, we estimate $(b)$ instead with the $L^{ \fr{5}{3}}_t L^{ 
\fr{30}{23}}_x$ dual 
Strichartz norm. Write
\begin{align*}
(b) &= (e^{-it H} w^J_n) v^j_n O( |u^J_n|^3 +  | u^J_n - e^{-itH} w^J_n|^3) 
\sum_{j=1}^J v^j_n - (e^{- itH} w^J_n) | u^J_n |^4,
\end{align*}
and apply H\"{o}lder's inequality:
\begin{equation}
\label{eqn:palais_smale_stability_iii_eq2}
\begin{split}
\| |x| (b) \|_{L^{\fr{5}{3}}_t L^{\fr{30}{23}}_x } &\lesssim  \| e^{-it H} 
w^J_n \|_{L^{10}_{t, x} } \| 
u^J_n \|_{L^{10}_{t, x}}^3 \| H^{1/2} u^J_n \|_{L^5_t L^{\fr{30}{11}}_x} \\
&+ \| e^{-itH} w^J_n \|_{L^{10}_{t, 
x}} ( \| u_n^J \|_{L^{10}_{t, x}}
^3 + \| e^{-itH} w^J_n \|_{L^{10}_{t, x}}^3)  \| 
H^{\fr{1}{2}} \sum_{j=1}^J  v^j_n \|_{L^{5}_t L^{\fr{30}{11}}_x}.
\end{split}
\end{equation}
Using \eqref{eqn:palais_smale_lpd_remainder_decay} and
\eqref{eqn:palais_smale_stability_property_ii_enhanced}, 
we have
\begin{align*}
\lim_{J \to J*} \limsup_{n \to \infty} \| |x| (b) \|_{L^{\fr{5}{3}}_t 
L^{\fr{30}{11}}_x } = 0.
\end{align*}
It remains to bound $\nabla (b)$. By the chain rule,
\begin{align*}
\nabla(b) &= O\left( (| u^J_n - e^{-itH} w^J_n|^4 - |u^J_n|^4) \nabla 
\sum_{j=1}^J  v^j_n \right) +| 
u_n^J|^4 |\nabla e^{-itH} w^J_n|\\
&= (b') + (b'').
\end{align*}
The first term $(b')$ can be treated in the manner of $\| |x| (b) 
\|_{L^{\fr{5}{3}}_t L^{\fr{30}{23}}_x}$ 
above. We now concern ourselves with $(b'')$. Fix a small parameter $\eta > 0$, 
and use the 
above 
remark to obtain $J' = J'(\eta) \le J$ such that 
\begin{align*}
\| \sum_{j = J'}^J v^j_n \|_{\dot{X}^1} \le \eta.
\end{align*}
Thus by the triangle inequality and H\"{o}lder,
\begin{align*}
\| (b'')\|_{L^{\fr{5}{3}}_t L^{\fr{30}{23}}_x} &= \| | \sum_{j=1}^J v^j_n + 
e^{-itH} w^J_n |^4 (e^{-itH} 
w^J_n) \|_{L^{\fr{5}{3}}_t L^{\fr{30}{23}}_x}\\
&\lesssim \| e^{-itH} w^J_n\|_{L^{10}_{t, x}}^4 \| H^{\fr{1}{2}} e^{-itH} 
w^J_n\|_{L^{5}_t L^{\fr{30}
{11}}_x} \\
&+  \| |\sum_{j=J'}^J v^j_n |^4 |\nabla e^{-itH} w^J_n | \|_{L^{\fr{5}{3}}_t 
L^{\fr{30}{23}}_x}  + C_{J'}
\sum_{j=1}^{J'} \|  |v^j_n|^4 \nabla e^{-itH} w^J_n \|_{L^{\fr{5}{3}}_t 
L^{\fr{30}{23}}_x}\\
&\lesssim  \| e^{-itH} w^J_n\|_{L^{10}_{t, x}}^4 \| H^{\fr{1}{2}} e^{-itH} 
w^J_n\|_{L^{5}_t L^{\fr{30}
{11}}_x} \\
&+  \| \sum_{j=J'}^J v^j_n \|_{\dot{X}^1}^4 \| |\nabla e^{-itH} w^J_n | \|_{L^5_t 
L^{\fr{30}{11}}_x}  + 
C_{J'} \sum_{j=1}^{J'} \|  |v^j_n|^4 \nabla e^{-itH} w^J_n \|_{L^{\fr{5}{3}}_t 
L^{\fr{30}{23}}_x}
\end{align*}
By Strichartz and the decay of $e^{-itH} w^J_n$ in $L^{10}_{t, x}$, the first 
term goes to $0$ as 
$J 
\to \infty, \ n \to \infty$. By Strichartz and the definition of $J'$, the 
second term is bounded by
\begin{align*}
\eta^4 \| w^J_n \|_{\Sigma}
\end{align*}
which can be made arbitrarily small since $\limsup_{n \to \infty} \| w^J_n 
\|_{\Sigma}$ is bounded 
uniformly in $J$. To finish, we show that for each fixed $j$
\begin{align*}
\lim_{J \to J^*} \limsup_{n \to \infty} \| |v^j_n|^4  \nabla e^{-itH} 
w^J_n\|_{L^{\fr{5}{3}}_t L^{\fr{30}
{11}}_x} = 0.
\end{align*}
By H\"{o}lder,
\begin{align*}
\| |v^j_n|^4 \nabla e^{-itH} w^J_n \|_{L^{\fr{5}{3}}_t L^{\fr{30}{23}}_x} \le \| 
v^j_n\|_{L^{10}_{t, x}}^3 
\| v^j_n \nabla e^{-itH} w^J_n \|_{L^{\fr{10}{3}}_t L^{ \fr{15}{7}}_x},
\end{align*}
so by \eqref{eqn:palais-smale_nonlinear_profile_scattering_size} it suffices to 
show
\begin{align}
\lim_{J \to J^*} \limsup_{n \to \infty} \| v^j_n \nabla e^{-itH} w^J_n 
\|_{L^{\fr{10}{3}}_t L^{\fr{15}{7}}
_x} 
= 0. \label{eqn:palais-smale_remainder_derivative_estimate_3d}
\end{align}
For any $\varepsilon > 0$, there exists $\psi^j \in C^\infty_c(\mf{R} \times 
\mf{R}^3)$ and functions $c_n^j(t), \ |c_n^j| \equiv 1$ such that
\begin{align*}
\limsup_{n \to \infty} \| v^j_n - c_n^j \tilde{G}^j_n \psi^j \|_{L^{10}_{t, x}( [-\fr{1}{2}, \fr{1}{2}] )} < \varepsilon,
\end{align*}
Note that $\tilde{G}^j_n \psi^j$ is supported on the set 
\[
\{ |t - t^j_n| \lesssim (N_n^j)^{-2}, \ |x-x^j_n| \lesssim 
(N_n^j)^{-1}\}.
\]
Thus for all $n$ sufficiently large,
\begin{align*}
&\| v^j_n \nabla e^{-itH} w^J_n\|_{L^{\fr{10}{3}}_t L^{\fr{15}{7}}_x}
\\
&\le \| 
v^j_n - c_n^j \tilde{G}^j_n \psi^j\|
_{L^{10}_{t, x}} \| \nabla e^{-itH} w^J_n \|_{L^5_t L^{\fr{30}{11}}_x} + \| 
\tilde{G}^j_n \psi^j \nabla e^{-itH} 
w^J_n \|_{L^{\fr{10}{3}}_t L^{\fr{15}{7}}_x}\\
&\lesssim_{E_c} \varepsilon + \| (\tilde{G}^j_n \psi^j ) \nabla e^{-itH} w_n^J\|_{L^{ 
\fr{10}{3}}_t L^{\fr{15}
{7}}
_x}.
\end{align*}
From the definition of the operators $\tilde{G}_n^j$, we have
\begin{align*}
 \| (\tilde{G}^j_n \psi^j ) \nabla e^{-itH} w_n^J\|_{L^{ \fr{10}{3}}_t L^{\fr{15}{7}}_x} 
&\lesssim_{\varepsilon} 
N_n^{\fr{1}{2}} \| \nabla e^{-itH} w_n^J \|_{L^{\fr{10}{3}}_t L^{\fr{15}{7}}_x ( 
|t - t^j_n| \lesssim 
(N_n^j)^{-2}, |x-x^j_n| \lesssim (N_n^j)^{-1})}. 
\end{align*}
Since $(N_n^j)^{-2} |x^j_n| = O( (N_n^j)^{-1})$, Corollary~\ref{cor:local_smoothing_cor} implies 
\begin{align*}
\| v^j_n \nabla e^{-itH} w_n^J\|_{L^{ \fr{10}{3}}_t L^{\fr{15}{7}}_x} &\lesssim 
\varepsilon + C_
\varepsilon \| e^{-itH} w^J_n\|_{L^{10}_{t, x}}^{ \fr{1}{9}} \| w^J_n 
\|_{\Sigma}^{\fr{8}{9}}.
\end{align*}
Sending $n \to \infty$, then $J \to J^*$, then $\varepsilon \to 0$ establishes 
\eqref{eqn:palais-smale_remainder_derivative_estimate_3d}, and with it, Property 
(iii). This completes the treatment of the case $d = 3$.

By perturbation theory, $\limsup_{n \to \infty} S_{(-T, T)} \le C(E_c) < 
\infty$, contrary to the 
Palais-Smale hypothesis. This rules out Case 2 and completes the proof of 
Proposition~\ref{prop:palais-smale}.
\end{proof}



Armed with Proposition~\ref{prop:palais-smale}, we can finish the 
proof of 
Theorem~\ref{thm:main_theorem}.
\begin{proof}[Proof of Theorem~\ref{thm:main_theorem}]
Suppose the theorem failed. In the defocusing case, 
there exist $E_c \in (0, \infty)$
and a sequence of 
solutions  $u_n$ with
$E(u_n) \to E_c$ and $S_{(-\fr{1}{n}, 0]}(u_n) \to \infty$ and 
$S_{[0, \fr{1}{n})}(u_n) \to \infty$. The same is true in the focusing 
case except $E_c$ is restricted to the interval $(0, E_{\Delta}(W))$
and $\limsup_{n} \|u_n(0) \|_{\dot{H}^1} < \| W\|_{\dot{H}^1}$. By Proposition 
\ref{prop:palais-smale}, after passing to a subsequence 
$u_n(0)$ converges in $\Sigma$ to some $\phi$. Let
$u_\infty$ be the maximal solution to 
\eqref{eqn:nls_with_quadratic_potential}
with $u_\infty(0) = \phi$. By the stability theory and dominated 
convergence, 
$\lim_{n \to \infty} S_{[-\fr{1}{n}, \fr{1}{n}]} (u_n) = 
\lim_{n \to \infty} S_{[-\fr{1}{n}, \fr{1}{n}]}(u_\infty) = 0$, 
which is a contradiction.
\end{proof}

\section{Proof of Theorem \ref{thm:focusing_blowup}}
\label{section:focusing_blowup}
We begin by recalling some facts about the \emph{ground state}
\[
W(x) = ( 1 + \tfr{|x|^2}{d(d-2)} )^{-\fr{d-2}{2}} \in \dot{H}^1(\mf{R}^d)
\]
This function satisfies the elliptic PDE
\[
\tfr{1}{2}\Delta W + W^{\fr{4}{d-2}} W = 0.
\]
It is well-known (c.f. Aubin \cite{aubin} and Talenti \cite{talenti}) that 
the functions witnessing the sharp constant in the Sobolev inequality
\[
\| f\|_{L^{\fr{2d}{d-2}}(\mf{R}^d)} \le C_d \| \nabla f\|_{L^2(\mf{R}^d)},
\]
are precisely those of the form 
$f(x) = \alpha W(\beta (x-x_0) ), \ \alpha \in \mf{C}, \ \beta > 0, \
x_0 \in \mf{R}^d$. 

For the reader's convenience, we reiterate the definitions of the
energy associated to the focusing energy-critical NLS with and without potential:
\begin{align*}
E_{\Delta}(u) &= \int_{\mf{R}^d} \tfr{1}{2} |\nabla u|^2 -
(1 - \tfr{2}{d}) |u|^{\fr{2d}{d-2}} \, dx,\\
E(u) &= E_{\Delta}(u) + \tfr{1}{2} \|xu\|_{L^2}^2.
\end{align*}

\begin{lma}[Energy trapping {\cite{kenig-merle_focusing_nls}}]
\label{lma:appendix_lma_1}
Suppose  $E_{\Delta}(u) \le (1-\delta_0) E_{\Delta}(W)$ .

\begin{itemize}
\item If $\| \nabla u\|_{L^2} \le \| \nabla W\|_{L^2}$, then there exists $\delta_1 > 0$ 
depending on $\delta_0$ such that 
\[
\| \nabla u\|_{L^2} \le (1-\delta_1) \| W\|_{L^2},
\]
and $E_{\Delta}(u) \ge 0$.

\item If $\| \nabla u\|_{L^2} \ge \| \nabla W\|_{L^2}$ then there exists $\delta_2 > 0$ depending on $\delta_0$ 
such that 
\[
\| \nabla u\|_{L^2} \ge (1+\delta_2) \| \nabla W\|_{L^2},
\]
and $\tfr{1}{2} \| \nabla u\|_{L^2}^2 - \| u\|_{L^{\fr{2d}{d-2}}}^{\fr{2d}{d-2}} \le -\delta_0 E_{\Delta}(W)$.
\end{itemize}
\end{lma}



Now suppose $E(u) < E_{\Delta}(W)$ and $\| \nabla u \|_{L^2} \le \|
\nabla W\|_{L^2}$. 
The energy inequality can be written as 
\[
\|u\|_{\Sigma}^2 + (1-\tfr{2}{d}) ( \|
W\|_{L^{\fr{2d}{d-2}}}^{\fr{2d}{d-2}} - \|
u\|_{L^{\fr{2d}{d-2}}}^{\fr{2d}{d-2}} ) \le \| \nabla W\|_{L^2}^2.
\]
By the variational characterization of $W$, the difference of norms on
the left side is nonnegative; therefore
\[
\| u\|_{\Sigma} \le \| \nabla W\|_{L^2}.
\]
Combining the above with conservation of energy and a
continuity argument, we obtain
\begin{cor}
\label{cor:appendix_cor_1}
Suppose $u: I \times \mf{R}^d \to \mf{C}$ is a solution to the focusing equation~\eqref{eqn:nls_with_quadratic_potential} with $E(u) \le (1-\delta_0) E_{\Delta}(W)$. Then 
there exist $\delta_1, \delta_2 > 0$, depending on $\delta_0$, such that

\begin{itemize}
\item If $\| u(0)\|_{\dot{H}^1} \le \| W\|_{\dot{H}^1}$, then 
\[\sup_{t \in I} \| u(t) \|_{\Sigma} \le (1-\delta_1) \|
W\|_{\dot{H}^1} \quad \text{and} \quad E(u) \ge 0.
\]

\item If $\|u(0)\|_{\dot{H}^1} \ge \|W\|_{\dot{H}^1}$, then
\[
\inf_{t \in I} \| u(t) \|_{\Sigma} \ge (1+\delta_2) \|
W\|_{\dot{H}^1} \quad \text{and} \quad \tfr{1}{2}\| \nabla u\|_2^2 -
\| u \|_{L^{\fr{2d}{d-2}}}^{\fr{2d}{d-2}} \le -\delta_0 E_{\Delta}(W).
\]
\end{itemize}

\end{cor}

\begin{proof}[Proof of Theorem~\ref{thm:focusing_blowup}]
Let $u$ be the 
maximal solution to \eqref{eqn:nls_with_quadratic_potential} with 
\[
u(0) = u_0, \quad E(u_0) < E_{\Delta}(W), \quad \| \nabla u_0 \|_2 \ge \| \nabla
W\|_2.
\]  
Let $f(t) = \int_{\mf{R}^d} |x|^2 |u(t, x)|^2 \, dx$. 
It can be shown \cite{cazenave} that $f$ is $C^2$ on the interval of existence and
\[
f''(t) = \int | \nabla u (t,x)|^2 - 2|u(t,x)|^{\fr{2d}{d-2}} - \tfr{1}{2}|x|^2 |u(t,x)|^2 \, dx.
\]
By the corollary, $f''$ is bounded above by some fixed $C < 0$. Therefore 
\[
f(t) \le A + Bt + \tfr{C}{2} t^2
\]
for some constants $A$ and $B$. It follows
that $u$ has a finite lifespan in both time directions. 
\end{proof}

\section{Bounded linear potentials}
\label{section:bounded_linear_potentials}
In this section we show, using a perturbative argument, that 
\begin{equation}
\label{eqn:nls_with_bounded_linear_potential}
i\partial_t u = (-\tfr{1}{2}\Delta + V)u + |u|^{\fr{4}{d-2}}u, \quad u(0) = u_0 \in H^1(\mf{R}^d)
\end{equation}
is globally wellposed whenever $V$ is a real-valued function with
\[
V_{max} := \|V\|_{L^\infty} + \| \nabla V\|_{L^\infty} < \infty.
\]
 This
equation defines the Hamiltonian flow of the energy functional
\begin{equation}
E(u(t)) = \int_{\mf{R}^d} \tfr{1}{2} |\nabla u(t,x)|^2 + V |u(t,x)|^2 + \tfr{d-2}{d}
|u|^{\fr{2d}{d-2}} \, dx = E(u(0)).
\end{equation}
Solutions to \eqref{eqn:nls_with_bounded_linear_potential} also conserve \emph{mass}:
\[
M(u(t)) = \int_{\mf{R}^d} |u(t,x)|^2 \, dx = M(u(0)).
\]
It will be convenient to assume $V$ is positive and
bounded away from $0$. This hypothesis allows us to bound the $H^1$
norm of $u$ purely in terms of $E$ instead of both $E$ and $M$, and
causes no loss of generality because for sign-indefinite $V$ we could simply
consider the conserved quantity $E + CM$ in place of $E$, where $C$ is
some positive constant. 

\begin{thm}
\label{thm:bounded_linear_potential_gwp}
For any $u_0 \in H^1(\mf{R}^d)$, \eqref{eqn:nls_with_bounded_linear_potential} has a unique global solution $u \in C^0_{t, loc} H^1_x
(\mf{R} \times \mf{R}^d)$. Further, $u$ obeys the spacetime bounds
\[
S_{I}(u) \le C(\|u_0\|_{H^1}, |I|)
\]
for any compact interval $I \subset \mf{R}$.
\end{thm}

As alluded to at the beginning of this section, the proof uses the
strategy pioneered by \cite{matador} and treats the term 
$Vu$ as a perturbation to \eqref{eqn:ckstt_eqn}, which is globally
wellposed. Thus Duhamel's formula reads
\begin{equation}
u(t) = e^{\fr{it\Delta}{2}} u(t_0) - i \int_0^t
e^{\fr{i(t-s)\Delta}{2}} [ |u(s)|^{\fr{4}{d-2}} u(s) + V u(s)]
ds. \label{eqn:duhamel_for_bounded_linear_potential}
\end{equation}

We record mostly without proof some standard results in the local
theory of \eqref{eqn:nls_with_bounded_linear_potential}. Introduce the notation
\[
\begin{split}
\| u\|_{X(I)} = \| \nabla u\|_{L^{\fr{2(d+2)}{d-2}}_{t}
  L^{\fr{2d(d+2)}{d^2+4}}_x (I \times \mf{R}^d)}.
\end{split}
\]

\begin{lma}[Local wellposedness]
\label{lma:bounded_linear_potential_lwp}
Fix $u_0 \in H^1(\mf{R}^d)$, and suppose $T_0 > 0$ is such that 
\[
\| e^{\fr{it\Delta}{2}} u_0\|_{X([-T_0, T_0])} \le \eta \le \eta_0
\]
where $\eta_0 = \eta_0(d)$ is a fixed parameter. Then there exists a
positive 
\[T_1 = T_1(\|u_0\|_{H^1}, \eta, V_{max})
\]
such that \eqref{eqn:nls_with_bounded_linear_potential}
has a unique (strong) solution $u \in C^0_t H^1_x([-T_1, T_1] \times \mf{R}^d)$. Further, if $(-T_{min}, T_{max})$ is the maximal
lifespan of $u$, then $\|\nabla u\|_{S(I)} < \infty$ for every
compact interval $I \subset (-T_{min}, T_{max})$, where $\| \cdot \|_{S(I)}$ is
the Strichartz norm defined in Section~\ref{subsection:notation}. 
\end{lma}

\begin{proof}[Proof sketch]
Run the usual contraction mapping argument using the Strichartz
estimates to show that 
\[
\mcal{I}(u)(t) = e^{\fr{it\Delta}{2}} u_0 - i \int_0^t
e^{\fr{i(t-s)\Delta}{2}} [ |u(s)|^{\fr{4}{d-2}} u(s) + Vu(s)] dx
\]
has a fixed point in a suitable function space. Estimate the terms
involving  $V$ in the $L^1_t L^2_x$ dual Strichartz
norm and choose the parameter $T_1$ to make those terms sufficiently
small after using H\"{o}lder in time.
\end{proof}

\begin{lma}[Blowup criterion]
\label{lma:bounded_linear_potential_blowup_criterion}
Let $u: (T_0, T_1) \times \mf{R}^d \to \mf{C}$ be a solution to \eqref{eqn:nls_with_bounded_linear_potential} with 
\[
\|u\|_{X( (T_0, T_1))} < \infty.
\]
If $T_0 > -\infty$ or $T_1 < \infty$, then $u$ can be extended to a
solution on a larger time interval.
\end{lma}

The key result we will rely on is the stability theory for the
energy-critical NLS \eqref{eqn:ckstt_eqn}.
\begin{lma}[Stability \cite{tao-visan_stability}]
\label{lma:ckstt_stability}
Let $\tilde{u}: 
I \times \mf{R}^d \to 
\mf{C}$ be 
an approximate solution to equation \eqref{eqn:ckstt_eqn} in the 
sense that 
\[
i\partial_t \tilde{u} = -\tfr{1}{2} \Delta u \pm |\tilde{u}|^{\fr{4}{d-2}} \tilde{u} + e
\]
for some function $e$. Assume that 
\begin{equation}
\label{eqn:ckstt_stability_prop_hyp_1} 
\| \tilde{u}\|_{L^{\fr{2(d+2)}{d-2}}_{t,x}} \le L,
\quad \| \nabla u \|_{L^\infty_t L^2_x} \le E,
\end{equation}
and that for some $0 < \varepsilon < \varepsilon_0(E, L)$ one has
\begin{equation}
\label{eqn:ckstt_stability_prop_hyp_2} 
\| \tilde{u}(t_0) - u_0 \|_{\dot{H}^1} + \| \nabla e \|_{N(I)} \le \varepsilon,
\end{equation}
where $\| \cdot \|_{N(I)}$ was defined in Section~\ref{subsection:notation}. 
Then there exists a unique solution $u : I \times \mf{R}^d 
\to \mf{C}$ to \eqref{eqn:ckstt_eqn} with $u(t_0) = u_0$ 
which further satisfies 
the estimates
\begin{equation}
\| \tilde{u} -  u \|_{L^{\fr{2(d+2)}{d-2}}_{t,x} } +
\| \nabla (\tilde{u} - u) \|_{S(I)} \lesssim C(E, L) 
\varepsilon^c
\end{equation}
where $0 < c = c(d) < 1$ and $C(E, L)$ is a function which is nondecreasing
in each variable.
\end{lma}


\begin{proof}[Proof of Theorem~\ref{thm:bounded_linear_potential_gwp}]

  It suffices to show that for $T$ sufficiently small depending only
  on $E = E(u_0)$, the solution $u$ to
  \eqref{eqn:nls_with_bounded_linear_potential} on $[0, T]$ satisfies
  an \emph{a priori} estimate
\begin{equation}
\| u\|_{X([0, T])} \le C(E). \label{eqn:bounded_linear_potential_a_priori_bound}
\end{equation}
From Lemma \ref{lma:bounded_linear_potential_blowup_criterion} and energy
conservation, it will follow that $u$ is a global solution with the desired
spacetime bound.

By Theorem~\ref{thm:ckstt}, the equation
\[
(i\partial_t + \tfr{1}{2}\Delta)w = |w|^{\fr{4}{d-2}} w, \quad w(0) = u(0).
\]
has a unique global solution $w \in C^0_{t, loc}\dot{H}^1_x(\mf{R} \times
\mf{R}^d)$ with the spacetime bound
\eqref{eqn:ckstt_spacetime_bound}. 
Fix a small parameter $\eta > 0$ to be
determined shortly, and
partition $[0, \infty)$ into $ J(E, \eta)$ intervals $I_j = [t_j, t_{j+1})$ so that 
\begin{equation}
\|w\|_{X(I_j)} \le \eta. \label{eqn:bounded_linear_potential_eqn_1}
\end{equation}
For some $J' < J$, we then have
\[
[0, T] = \bigcup_{j=0}^{J'-1} ([0, T] \cap I_j).
\]
 We make two preliminary estimates. By H\"{o}lder in time,
\begin{equation}
\label{eqn:bounded_linear_potential_eqn_1.5}
\| Vu\|_{N(I_j)} + \| \nabla (Vu)\|_{N(I_j)} \lesssim C_V T \|u\|_{L^\infty_t H^1_x(I_j)}
\le \varepsilon
\end{equation}
for any $\varepsilon$ provided that $T = T(E, V, \varepsilon)$ is sufficiently
small.
Further, observe that
\begin{equation}
\| e^{\fr{i(t-t_j)\Delta}{2}} w(t_j) \|_{X(I_j)} \le
2\eta \label{eqn:bounded_linear_potential_eqn_2}
\end{equation}
for $\eta$ sufficiently small depending only on $d$. Indeed, from the
Duhamel formula
\begin{equation}
w(t) = e^{\fr{i(t-t_j)\Delta}{2}} w(t_j) - i\int_{t_j}^{t}
e^{\fr{i(t-s)\Delta}{2}} (|w|^{\fr{4}{d-2}} w) (s)ds, \label{eqn:ckstt_duhamel}
\end{equation}
Strichartz, and the chain rule, we find that
\begin{align*}
\| e^{\fr{i(t-t_j)\Delta}{2}} w(t_j)\|_{X(I_j)} &\le \| w
\|_{X(I_j)} + c_d \| \nabla (|w|^{\fr{4}{d-2}} w) \|_{L^2_t
  L^{\fr{2d}{d+2}}_x (I_j)}\\
&\le \eta + c_d \|w\|_{X(I_j)}^{\fr{d+2}{d-2}}\\
&\le \eta + c_d \eta^{\fr{d+2}{d-2}}.
\end{align*}
Taking $\eta$ sufficiently small relative to $c_d$, we obtain
 \eqref{eqn:bounded_linear_potential_eqn_2}.

Now, choosing $\varepsilon < \eta$ in
\eqref{eqn:bounded_linear_potential_eqn_1.5} (and adjusting $T$ accordingly), we use the Duhamel
formula \eqref{eqn:duhamel_for_bounded_linear_potential}, Strichartz,
H\"{o}lder, and \eqref{eqn:bounded_linear_potential_eqn_2} to obtain
\begin{align*}
\| u\|_{X (I_0)} &\le \| e^{\fr{it\Delta}{2}} u(0)
\|_{X(I_0)} + c_d \| u\|_{X(I_0)}^{\fr{d+2}{d-2}} +
C \|Vu\|_{L^1_t H^1_x(I_0)}\\
&\le 2\eta + c_d \|u\|_{X(I_0)}^{\fr{d+2}{d-2}} + C_V T
\|u\|_{L^\infty_t H^1_x(I_0)}\\
&\le 3\eta + c_d \|u \|_{X(I_0)}^{\fr{d+2}{d-2}}
\end{align*}
By a continuity
argument,
\begin{align}
 \| u\|_{X(I_0)} \le 4 \eta. \label{eqn:bounded_linear_potential_eqn_3}
\end{align}
Choosing $\varepsilon$ sufficiently small in
\eqref{eqn:bounded_linear_potential_eqn_1.5} so that the smallness
condition \eqref{eqn:ckstt_stability_prop_hyp_2} is satisfied, we
apply Lemma~\ref{lma:ckstt_stability} with $ \|u(0) -
w(0)\|_{\dot{H}^1} = 0$ to find that
\begin{equation}
\label{eqn:bounded_linear_potential_eqn_4}
\| \nabla(u - w)\|_{S(I_0)} \le C(E) \varepsilon^c
\end{equation}

On the interval $I_1$, use \eqref{eqn:bounded_linear_potential_eqn_2}, \eqref{eqn:bounded_linear_potential_eqn_4}, and the usual
estimates to obtain
\begin{align*}
\|u\|_{X(I_1)} &\le \| e^{\fr{i(t-t_1)\Delta}{2}} u(t_1)
\|_{X(I_1)} + c_d \| u\|_{X(I_1)}^{\fr{d+2}{d-2}} + C_V
T \|u\|_{L^\infty_t H^1_x(I_1)}\\
&\le C(E)\varepsilon^c + 2\eta + c\| u\|_{X(I_1)}^{\fr{d+2}{d-2}}
+ \eta,
\end{align*}
where the $C(E)$ in the last line has absorbed the Strichartz constant
$c$; this redefinition of $C(E)$ will cause no trouble because the
number of times it will occur depends only on $E$, $d$, and $V$. By
by choosing $\varepsilon$ sufficiently small relative to $\eta$ and
using continuity, we find that
\[
\| u\|_{X( I_1)} \le 4 \eta.
\]
As before, by choosing $T$ sufficiently small we obtain 
\begin{align*}
\|\nabla (Vu)\|_{\dot{N}^0(I_1)} &\le \varepsilon\\
\|e^{\fr{i(t-t_1)\Delta}{2}}[ u(t_1) - w(t_1)] \|_{X(I_1)}
&\le C(E)\varepsilon^c
\end{align*}
for any $\varepsilon \le \varepsilon_0(E, L)$. Hence by Lemma
\ref{lma:ckstt_stability} and possibly after modifying
$T$ depending on $E$, we get
\[
\|\nabla (u - w)\|_{S(I_1)} \le C(E)\varepsilon^c.
\]

We emphasize that the parameters $\eta, \varepsilon, T$ are
chosen so that each depends only on the preceding parameters and on
the fixed quantities $d, E, V$.

After iterating at most $J'$ times and summing the bounds over $0 \le
j \le J' - 1$, we conclude that for $T$
sufficiently small depending on $E$ and $V$,
\[
\|u\|_{X([0, T])} \le 4J' \eta \le C(E).
\]
This establishes the bound \eqref{eqn:bounded_linear_potential_a_priori_bound}.
\end{proof}

\bibliographystyle{plain}

\bibliography{bibliography}

\end{document}